\documentclass[12pt]{article}
\usepackage{graphicx}
\usepackage{amsmath,amsthm,amssymb,enumerate}
\usepackage{euscript,mathrsfs}
\usepackage{color}
\usepackage{dsfont}
\usepackage[citecolor=blue,colorlinks=true]{hyperref}
\usepackage[left=2cm,right=2cm,top=2.5cm,bottom=2.5cm]{geometry}
\usepackage{color}
\usepackage[framemethod=tikz]{mdframed}
\allowdisplaybreaks

\usepackage{bm}
\usepackage{subcaption}

\usepackage{soul}

\catcode`\@=11 \@addtoreset{equation}{section}

\catcode`\@=12

\allowdisplaybreaks

\newtheorem{Theorem}{Theorem}[section]
\newtheorem{Proposition}[Theorem]{Proposition}
\newtheorem{Lemma}[Theorem]{Lemma}
\newtheorem{Corollary}[Theorem]{Corollary}

\theoremstyle{definition}
\newtheorem{Definition}[Theorem]{Definition}

\newtheorem{Remark}[Theorem]{Remark}

\newcommand{\bTheorem}[1]{
\begin{Theorem} \label{T#1} }
\newcommand{\eT}{\end{Theorem}}

\newcommand{\bProposition}[1]{
\begin{Proposition} \label{P#1}}
\newcommand{\eP}{\end{Proposition}}

\newcommand{\bLemma}[1]{
\begin{Lemma} \label{L#1} }
\newcommand{\eL}{\end{Lemma}}

\newcommand{\bCorollary}[1]{
\begin{Corollary} \label{C#1} }
\newcommand{\eC}{\end{Corollary}}

\newcommand{\bRemark}[1]{
\begin{Remark} \label{R#1} }
\newcommand{\eR}{\end{Remark}}

\newcommand{\bDefinition}[1]{
\begin{Definition} \label{D#1} }
\newcommand{\eD}{\end{Definition}}

\newcommand{\prst}{\mathcal{P}}

\newcommand{\Vrmh}{\{ \vrh, \vmh \} }

\newcommand{\jump}[1]{\left[ \left[ #1 \right] \right]}

\newcommand{\Td}{\mathbb{T}^d}

\newcommand{\bfF}{\mathbf{F}}
\newcommand{\vrh}{\vr_h}
\newcommand{\vrhk}{\vr_{h_k}}
\newcommand{\vuhk}{\vu_{h_k}}
\newcommand{\vmh}{\vm_h}

\newcommand{\tvm}{\widetilde{\vc{m}}}

\newcommand{\bfphi}{\boldsymbol{\varphi}}

\newcommand{\ds}{\,\mathrm{d}\sigma}

\newcommand{\bFormula}[1]{
\begin{equation} \label{#1}}
\newcommand{\eF}{\end{equation}}

\newcommand{\grid}{\mathcal{T}}

\newcommand{\vuh}{\vu_h}
\newcommand{\vvh}{\vv_h}

\newcommand{\intSh}[1] {\int_{\sigma} #1 \ds }

\newcommand{\Divh}{{\rm div}_h}

\newcommand{\Ov}[1]{\overline{#1}}

\newcommand{\aleq}{\stackrel{<}{\sim}}

\newcommand{\avs}[1]{\left\{\!\!\left\{ #1\right\}\!\!\right\}}
\newcommand{\vQh}{\vc{Q}_h}
\newcommand{\pd}{\partial}

\newcommand{\gradd}{\nabla_{\mathcal D}}

\newcommand{\muh}{h^\varepsilon}

\newcommand{\norm}[1]{\left\lVert#1\right\rVert}

\newcommand{\vr}{\varrho}

\newcommand{\tvr}{\tilde \vr}

\newcommand{\vu}{\vc{u}}
\newcommand{\vm}{\vc{m}}

\newcommand{\vn}{\vc{n}}

\newcommand{\vc}[1]{{\bf #1}}
\renewcommand{\vc}[1]{{\bm #1}}

\newcommand{\Div}{{\rm div}_x}
\newcommand{\Grad}{\nabla_x}

\newcommand{\dx}{\,{\rm d} {x}}

\newcommand{\dt}{\,{\rm d} t }

\newcommand{\dxdt}{\dx  \dt}

\newcommand{\intTd}[1]{\int_{\mathbb{T}^d} #1 \ \dx}

\newcommand{\intTO}[1]{\int_0^T \int_{\Td} #1 \ \dxdt}
\newcommand{\inttaO}[1]{\int_0^{\tau} \int_{\Td} #1 \ \dxdt}

\newcommand{\intT}[1]{\int_0^{T}  #1 \dt}

\newcommand{\vv}{\vc{v}}

\newcommand{\ep}{\varepsilon}

\newcommand{\expe}[1]{ \mathbb{E} \left[ #1 \right] }

\newcommand{\br}{ \nonumber \\ }

\newcommand{\cred}{\color{red}}

\def\softd{{\leavevmode\setbox1=\hbox{d}%
          \hbox to 1.05\wd1{d\kern-0.4ex{\char039}\hss}}}
\definecolor{Cgrey}{rgb}{0.85,0.85,0.85}
\definecolor{Cblue}{rgb}{0.50,0.85,0.85}
\definecolor{Cred}{rgb}{1,0,0}
\definecolor{fancy}{rgb}{0.10,0.85,0.10}

\newcommand\Cbox[2]{%
    \newbox\contentbox%
    \newbox\bkgdbox%
    \setbox\contentbox\hbox to \hsize{%
        \vtop{
            \kern\columnsep
            \hbox to \hsize{%
                \kern\columnsep%
                \advance\hsize by -2\columnsep%
                \setlength{\textwidth}{\hsize}%
                \vbox{
                    \parskip=\baselineskip
                    \parindent=0bp
                    #2
                }%
                \kern\columnsep%
            }%
            \kern\columnsep%
        }%
    }%
    \setbox\bkgdbox\vbox{
        \color{#1}
        \hrule width  \wd\contentbox %
               height \ht\contentbox %
               depth  \dp\contentbox
        \color{black}
    }%
    \wd\bkgdbox=0bp%
    \vbox{\hbox to \hsize{\box\bkgdbox\box\contentbox}}%
    \vskip\baselineskip%
}

\mdfdefinestyle{MyFrame}{%
	linecolor=black,
	outerlinewidth=1pt,
	roundcorner=5pt,
	innertopmargin=\baselineskip,
	innerbottommargin=\baselineskip,
	innerrightmargin=10pt,
	innerleftmargin=10pt,
	backgroundcolor=white!20!white}

\date{}

\newcommand{\Tkd}{\mathbb{T}^{kd}}

\begin{document}


\title{Convergence and error analysis of compressible fluid flows with random data: Monte Carlo method\footnotetext{This work was partially supported
by the Mathematical Research Institute Oberwolfach by the Research in Pairs stay in 2022. The authors gratefully acknowledge the hospitality of the institute  and its stimulating working atmosphere.}}

\author{Eduard Feireisl\thanks{The work of E.F. and B.S was  supported by the
Czech Sciences Foundation (GA\v CR), Grant Agreement
21--02411S. The Institute of Mathematics of the Academy of Sciences of
the Czech Republic is supported by RVO:67985840.
\newline
\hspace*{1em} $^\clubsuit$M.L. has been funded by the Deutsche Forschungsgemeinschaft (DFG, German Research
Foundation) - Project number 233630050 - TRR 146 as well as by TRR 165 Waves to Weather. She is grateful
to the Gutenberg Research College and Mainz Institute of Multiscale Modelling for supporting her research.
\newline \hspace*{1em} $^\clubsuit$ The
research of Y.Y. was funded by Sino-German (CSC-DAAD) Postdoc Scholarship Program in 2020 - Project number
57531629.
}
\and M\' aria Luk\' a\v cov\' a -- Medvi\softd ov\' a$^{\clubsuit}$
\and Bangwei She$^{*,\spadesuit}$
\and Yuhuan Yuan$^{\clubsuit}$
}

\date{}

\maketitle

\medskip

\centerline{$^*$ Institute of Mathematics of the Academy of Sciences of the Czech Republic}

\centerline{\v Zitn\' a 25, CZ-115 67 Praha 1, Czech Republic}

\medskip

\centerline{$^\clubsuit$Institute of Mathematics, Johannes Gutenberg-University Mainz}

\centerline{Staudingerweg 9, 55 128 Mainz, Germany}

\medskip
\centerline{$^\spadesuit$Academy for Multidisciplinary studies, Capital Normal University}
\centerline{ West 3rd Ring North Road 105, 100048 Beijing, P. R. China}

\begin{abstract}
The goal of this paper is to  study
 convergence and error estimates of the Monte Carlo method  for  the Navier--Stokes equations with random data. To discretize  in space and time, the Monte Carlo method is combined with a suitable deterministic discretization scheme, such as a finite volume method. We assume that the initial data, force and the viscosity  coefficients are random variables and
study both, the statistical convergence rates as well as the approximation errors. Since the compressible Navier--Stokes equations are not known to be uniquely solvable in the class of global weak solutions, we cannot apply pathwise arguments to analyze the random Navier--Stokes equations. Instead we have to apply intrinsic stochastic
compactness arguments via the Skorokhod representation theorem
and the Gyöngy–Krylov method. Assuming that the numerical solutions are bounded in probability, we prove that the Monte Carlo finite volume method converges to a
statistical strong solution. The convergence rates are discussed as well. Numerical experiments illustrate theoretical results.

\end{abstract}





\section{Introduction}
\label{i}

Many problems in science and engineering are inherently random due to uncertain data. To quantify the uncertainty propagation various methods have been developed in literature, such as the stochastic collocation method, stochastic Galerkin method and the Monte Carlo method. All of them have their pros and contras, but the Monte Carlo method is the most frequently used, in particular for complex problems arising in engineering or meteorology.

In this paper we concentrate on the compressible fluid flows with random data. Our goal is to establish a suitable theoretical background to perform numerical analysis for the governing random partial differential equations.  We focus on  the \emph{Monte Carlo method} to establish the expected outcome of a random event.
To this end, we need:
\begin{enumerate}
	\item a deterministic predictive model to identify the values of the dependent variables in terms of the data;
	\item a characteristic distribution of the data based on a judicious judgement and/or historical observations of the phenomena to be predicted;
	\item a large number of identically distributed independent samples of the data to compute the expected output via the Monte Carlo method.
	\end{enumerate}

The above general ideas will be applied to a simple model of a barotropic viscous fluid:

\begin{mdframed}[style=MyFrame]

{\bf Navier--Stokes system}
\begin{align}
\partial_t \vr + \Div (\vr \vu) &= 0, \label{i1}	\\
\partial_t (\vr \vu) + \Div (\vr \vu \otimes \vu) + \Grad p(\vr) &= \Div \mathbb{S}(\mu, \lambda, \Grad \vu) + \vr \vc{g}, \label{i2} \\
\mathbb{S} (\mu, \lambda, \Grad \vu) &= \mu \left( \Grad \vu + \Grad^t \vu - \frac{2}{d} \Div \vu \mathbb{I} \right) + \lambda \Div \vu \mathbb{I} \label{i3}
	\end{align}

\noindent
{\bf Space periodic boundary conditions}
\begin{equation} \label{i4}
	x \in \Td \equiv \left( [-1,1]|_{\{ -1, 1\} } \right)^d,\ d=2,3
\end{equation}

\noindent
{\bf Initial conditions}
\begin{equation} \label{i5}
	\vr(0, \cdot) = \vr_0, \ \vr \vu (0, \cdot) = \vm_0
\end{equation}

\end{mdframed}	
 Uncertainty is represented by {random data}:

\begin{mdframed}[style=MyFrame]

{\bf Random data.}

driving force \dotfill $\vc{g} = \vc{g}(x)$

viscosity coefficients \dotfill $\mu > 0$, $\lambda \geq 0$

initial data \dotfill $\vr_0$, $\vm_0$

\end{mdframed}
The \emph{dependent variables} can be identified with the conservative quantities: the mass density $\vr = \vr(t,x)$ and the momentum $\vm = (\vr \vu)(t,x)$.
They are determined as solutions of the initial value problem \eqref{i1}--\eqref{i5} defined on a time interval $[0, T]$.

The major stumbling block in applying any statistical method in a direct manner is the problem of well--posedness (uniqueness) of solutions to our predictive model -- the Navier--Stokes system. On the one hand, the initial--boundary
value problem \eqref{i1} -- \eqref{i5} is locally well posed in the class of smooth initial data, see e.g.
Tani \cite{TAN}, Matsumura and Nishida \cite{MANI}, Valli and Zajaczkowski \cite{VAZA}. On the other hand, the recent results of Buckmaster et al. \cite{BuCLGS} and Merle et al. \cite{MeRaRoSz} strongly indicate that originally
regular solutions may develop a blow up in a finite time. The weak solutions exist globally in time, see Lions \cite{LI,LI4} and the extension in \cite{FNP}, however, they are not known to be unique in terms of the initial data, at least
in the physically relevant cases $d=2,3$.

To avoid ambiguity in the choice of suitable solutions, we follow \cite{FanFei} introducing the class of statistical solutions based on a measurable semi--flow selection. In particular, any statistical solution
defined in \cite{FanFei} always coincides
with the strong solutions as long as the latter exists (weak--strong uniqueness principle).
With a well defined output at hand, we apply the standard probabilistic methods to obtain a suitable version
of the Strong law of large numbers and the Central limit theorem.

Our main goal is to approximate the statistical solutions of the Navier--Stokes system via the Monte Carlo method combined with a suitable numerical method for space--time discretization
and study its convergence and errors. For definiteness, we consider a finite volume (FV) method proposed in  \cite[Chapter 11]{FeLMMiSh}, but our results directly apply also {to the Marker-and-Cell} method \cite[Chapter 14]{FeLMMiSh}, \cite{BS_2}. {Roughly speaking, the present approach applies
to any numerical scheme that is (i) energy dissipative, (ii) convergent to regular continuous solutions
of the limit problem on its life--span.}
	
In agreement with the above mentioned theoretical results, we anticipate that the method may not converge for \emph{arbitrary} choice of the random parameters, however, convergence takes place for a statistically significant number of cases. Specifically, similarly to \cite{FeiLuk}, we suppose that the numerical output is \emph{bounded in probability}. Under these circumstance, we show that the Monte Carlo finite volume solutions converge in probability to their continuous counterparts selected for the Navier--Stokes system.
In the literature there are several theoretical results on the convergence and error analysis of the Monte Carlo method applied to random partial differential equations, see, e.g., Babu\v{s}ka et al.~\cite{babuska}, Badwaik et al.~\cite{klingel}, Fjordholm et al. \cite{FjLyMiWe}, Kolley et al.~\cite{weber}, Mishra and Schwab \cite{Mishra_Schwab}.
Unlike these studies, that are based on deterministic ``pathwise'' methods, our approach is genuinely stochastic requiring compactness arguments via the Skorokhod representation theorem and  the Gy\" ongy--Krylov method.

Finally, we derive qualitative \emph{error estimates} for the finite volume approximation. They follow from a variant of relative energy inequality and depend on the smoothness properties of the limit solution. Accordingly, the error is controlled only in probability.

The paper is organized as follows. In Section \ref{StSo} we introduce the concept of statistical solution of the Navier--Stokes system and establish suitable versions of the Strong law of large numbers and Central limit theorem. In particular, we obtain qualitative estimates of the statistical error in the Monte Carlo approximation. These results are generalized in Sections~\ref{high} and \ref{dev_var} to higher order statistical moments. Section \ref{FVA} is devoted to the finite volume method for approximation in time and space. We show convergence in
both deterministic and stochastic framework. Finally, we combine Monte Carlo sampling of the random data with the finite volume method and
prove the convergence and  error estimates of the Monte Carlo finite volume method for the Navier--Stokes equations in Section~\ref{MOCA}, and Section~\ref{sec_error}, respectively. The paper closes with Section~\ref{num} that presents numerical experiments illustrating theoretical results.

\section{Statistical analysis of the Navier--Stokes system}
\label{StSo}

Following \cite{FanFei} we introduce the concept of statistical solution for the Navier--Stokes system and apply statistical analysis to obtain a version of the Strong law of large numbers and \textcolor{red}{the} Central limit theorem.

\subsection{Weak solutions}
\label{w}

We recall the standard concept of (deterministic) \emph{weak solution} to the Navier--Stokes system.

\begin{mdframed}[style=MyFrame]

\begin{Definition}[weak solution] \label{Dw1}
	We say that $(\vr, \vu)$ is a \emph{weak solution} of the Navier--Stokes system
	\eqref{i1}--\eqref{i5} in a time interval $(0,T)$   if the following holds:
	\begin{itemize}
		\item
		{\bf Integrability.}
		\begin{align}
			\vr \geq 0, \ \vr &\in C_{\rm weak}([0,T] ; L^\gamma(\Td)) \cap C([0,T]; L^1(\Td)),\ p(\vr) \in C_{\rm weak}([0,T] ; L^1(\Td)),\ r > 1; \br
			\vu &\in L^2(0,T; W^{1,2} (\Td; R^d)), \quad
			\vr \vu \in C_{\rm weak}([0,T]; L^{\frac{2 \gamma}{\gamma + 1}}(\Td; R^d)).
			\label{w1}
		\end{align}
		
		\item {\bf Equation of continuity.}
		\begin{equation} \label{w2}
		\int_0^T \intTd{ \Big[ \vr \partial_t \varphi + \vr \vu \cdot \Grad \varphi \Big] } \dt = - \intTd{ \vr_0 \varphi (0, \cdot) }
		\end{equation}
	for any $\varphi \in C^1_c([0,T) \times \Td)$;
\begin{equation} \label{w3}
	\int_0^T \intTd{ \left[ b (\vr) \partial_t \varphi + b(\vr) \vu \cdot \Grad \varphi +
		\Big( b(\vr) - b'(\vr) \vr \Big) \Div \vu \varphi \right] } \dt = - \intTd{ b(\vr_0) \varphi (0, \cdot) }
\end{equation}
for any $\varphi \in C^1_c([0,T) \times \Td)$, and any $b \in C^1[0, \infty)$, $b' \in C_c [0, \infty)$.

	\item {\bf Momentum equation.}
\begin{align}
	\int_0^T &\intTd{ \Big[ \vr \vu \cdot \partial_t \bfphi + \vr \vu \otimes \vu :  \Grad \bfphi +
		p(\vr) \Div \bfphi \Big] } \dt \br &= \int_0^T \intTd{ \Big[ \mathbb{S}(\mu, \lambda, \Grad \vu) : \Grad \bfphi  - \vr \vc{g} \cdot \bfphi \Big]} \dt
	 - \intTd{ \vm_0 \cdot \bfphi (0, \cdot) }
	 \label{w4}
\end{align}	
for any $\bfphi \in C^1_c([0,T) \times \Td; R^d)$.
\item {\bf Energy inequality.}	
\begin{align}
	\int_0^T \partial_t \psi &\intTd{ \left[ \frac{1}{2} \vr |\vu|^2 + P(\vr) \right] } \dt - \int_0^T \psi \intTd{ \mathbb{S}(\mu, \lambda, \Grad \vu) : \Grad \vu } \dt \br &\geq -
	\int_0^T \psi \intTd{ \vr \vc{g} \cdot \vu } \dt - \psi(0) \intTd{  \left[ \frac{1}{2} \frac{|\vm_0|^2}{\vr_0} + P(\vr_0) \right] }
	\label{w5}
	\end{align}
for any $\psi \in C^1_c[0,T)$, $\psi \geq 0$, where
	\[
	P'(\vr) \vr - P(\vr) = p(\vr),\ P(0) = 0.
	\]

\end{itemize}
	
	\end{Definition}

\end{mdframed}

For the sake of simplicity, we consider the \emph{isentropic} equation of state:
\begin{equation} \label{w6}
p(\vr) = a \vr^\gamma, \ a > 0,\ \gamma > 1 \ \mbox{with the associated pressure potential}\
P(\vr) = \frac{a}{\gamma - 1} \vr^\gamma.
\end{equation}
In what follows we will assume that $\gamma > \frac d 2 .$  We note that in this case a global weak solution to the Navier--Stokes equations
exists, cf.~\cite{EF70}, \cite{LI4}.
In \eqref{w5}, it is convenient to define the total energy $E$ as a convex l.s.c. function of the conservative
variables $(\vr, \vm) \in R^{d+1}$:
\begin{equation} \label{w7}
E( \vr, \vm) = \left\{ \begin{array}{l} \frac{1}{2} \frac{ |\vm|^2 }{\vr} + P(\vr) \ \mbox{if}\ \vr > 0, \\
	0 \ \mbox{if} \ \vr = 0,\ \vm = 0, \\
	\infty \ \mbox{otherwise}.
	\end{array} \right.	
	\end{equation}

Applying Gronwall's lemma we obtain from the energy inequality \eqref{w5} the boundedness of the energy and the dissipation by the data
\begin{align}
{\sup_{t \in (0,T)}}  \intTd{ E \Big( \vr, \vm \Big) (t, \cdot)  }
& +  {\mu}\int_0^T  \intTd{ |\Grad \vu|^2}  \dt \br
&\aleq c\left(T, \| \vc{g} \|_{C(\Td; R^d)} \right)\left(1 +   \intTd{ E \Big( \vr_0, \vm_0 \Big)   } \right),
\label{energy_bb}
\end{align}
which yields the following bounds
\begin{equation}
\| \vr (t, \cdot) \|_{L^\gamma(\Td)}^\gamma + \| \vm(t, \cdot) \|_{L^{\frac{2\gamma}{\gamma + 1}}(\Td;R^d)}^{\frac{2\gamma}{\gamma + 1}}
\aleq  c\left(T, \| \vc{g} \|_{C(\Td; R^d)} \right)\left(1 +   \intTd{ E \Big( \vr_0, \vm_0 \Big)   } \right)
\label{bounds1}
\end{equation}
for any $t \in [0,T].$ Moreover, we can use a  Sobolev-Poincar\'e type inequality, cf.~Appendix~\ref{APDA},  in order to derive {a bound on} the velocity. Indeed, assuming that the total mass is initially bounded from below by a positive constant
$$
\intTd{ \vr_0 } \geq \underline{R} > 0\ {\Rightarrow \ \intTd{ \vr(t, \cdot) }
	\geq \underline{R} }\   \mbox{ for  } t \in (0,T)
$$
we have, {according to Appendix~\ref{APDA},}
\begin{align}\label{SP}
 \intT{  \norm{\vu}_{L^q(\Td;R^d)}^2 } &\aleq   \intT{\norm{\Grad \vu }_{L^2(\Td;R^{d\times d})}^2  }   \\
&+  \norm{\vr}_{L^\infty(0,T; L^{\gamma}(\Td))} \intT{ \norm{\Grad \vu }_{L^2(\Td;R^{d\times d})}^2} + \intTO{E(\vr_0,\vm_0)},
\end{align}
{where $q = 6$ if $d=3$, $q \geq 1$ arbitrary finite if $d=2$.}
Thus, we obtain
\begin{align*}
\intT{  \norm{\vu}_{L^q(\Td;R^d)}^2 }   \aleq
c\left(T, \| \vc{g} \|_{C(\Td, R^d)} \right)\left(1 +   \left(\intTd{ E ( \vr_0, \vm_0 )    }\right)^{\frac{\gamma+1}{\gamma}}\right),
\end{align*}
{Summing up we have shown}
\begin{align}
  \norm{\vu}_{L^2(0,T; W^{1,2}(\Td;R^d))}^2     \aleq
c\left(T, {\mu}, \| \vc{g} \|_{C(\Td, R^d)} \right)\left(1 + \left(\intTd{ E ( \vr_0, \vm_0 )   } \right)^{\frac{\gamma+1}{\gamma}}\right).
\label{velocity_est}
\end{align}

\subsection{Measurable semigroup selection}
\label{M}

As shown in \cite{EF70,LI4}, the weak solutions specified in Definition \ref{Dw1} exist for any finite energy initial data and sufficiently regular driving force $\vc{g}$ provided $\gamma > \frac{d}{2}$.
Unfortunately, uniqueness of weak solutions in terms of the data is still an outstanding open problem with
possibly negative conclusion.

We introduce the space of data,
\begin{align}
D = \Big\{ [\vr, \vm, \mu, \lambda, \vc{g} ] \Big|	\ &\vr \in L^1(\Td), \ \vm \in L^1(\Td; R^d),\
\intTd{ \vr } \geq \underline{R} > 0, \ \intTd{ E(\vr, \vm) } < \infty, \br
&\mu \geq \underline{\mu} > 0,\ \lambda \geq 0,\ \vc{g} \in C^\ell(\Td; R^d),\ \ell \geq 3 \Big \}  \
\label{M2}
\end{align}
which is considered as a Borel subset of the Polish space
\[
X = W^{-k,2}(\Td) \times W^{-k,2}(\Td; R^d) \times R \times R \times C^{\ell -1}(\Td; R^d),\ k > \frac{d}{2}.
\]
Indeed,
\[
D = \cup_{M \geq 1} D_M,
\]
where
\begin{align}
	D_M = \Big\{ [\vr, \vm, \mu, \lambda, \vc{g} ] \Big|	\ &\vr \in L^1(\Td), \ \vm \in L^1(\Td; R^d),
\intTd{ \vr } \geq \underline{R} > 0,\ \intTd{ E(\vr, \vm) } \leq M, \br
	&M \geq \mu \geq \underline{\mu} > 0,\ M \geq \lambda \geq 0,\ \vc{g} \in C^{\ell}(\Td; R^d),\ \| \vc{g} \|_{C^{\ell}(\Td)} \leq M , \ell \geq 3 \Big\}
	\label{M3}
\end{align}
are compact subsets of $X$.

To avoid the problem of well--posedness, we consider a suitable semiflow selection.
The following statement can be proved exactly as in \cite[Proposition 5.6]{FanFei}:

\begin{mdframed}[style=MyFrame]

\begin{Proposition}[{\bf Measurable semigroup selection}] \label{PM1}
Let $\gamma > \frac{d}{2}$.	
	There exists a mapping
	\[
	\mathcal{S} : D \times [0, \infty) \to D
	\]
	enjoying the following properties:
	
	\begin{itemize}
		\item
		
		\begin{equation} \label{M4}
			\mathcal{S}\Big([\vr_0, \vm_0, \mu, \lambda, \vc{g}]; t\Big) =
			\Big[\vr(t,\cdot), \vm(t, \cdot), \mu, \lambda, \vc{g} \Big] \in D,
			\end{equation}
where $(\vr, \vm = \vr \vu)$ is a weak solution of the Navier--Stokes system \eqref{i1}--\eqref{i4} with the
initial data $(\vr_0, \vm_0)$ specified in Definition \ref{Dw1}.

\item The mapping
\begin{equation} \label{M5}
\mathcal{S}: D \times [0, \infty) \to D
\end{equation}
is jointly Borel measurable, where $D$ is endowed with the topology of the space $X$, in particular,
\[
\mathcal{S} [ \cdot, t ]: D \to D
\]	
is Borel measurable for any $t \geq 0$.
\item For any $[\vr_0, \vm_0, \mu, \lambda, \vc{g}] \in D$ there is a set of times $\mathcal{R} \subset [0, \infty)$
of full measure, $0 \in \mathcal{R}$, such that
\begin{equation} \label{M6}
\mathcal{S}\Big([\vr_0, \vm_0, \mu, \lambda, \vc{g}]; s + t \Big) =
\mathcal{S}\Big( [\vr(s, \cdot), \vm(s, \cdot), \mu, \lambda, \vc{g}] ; t \Big)
\end{equation}
for any $s \in \mathcal{R}$ and any $t \geq 0$.	

\item For any $M_0$ and $t \geq 0$, there exists $M(t)$ such that
\begin{equation} \label{M1}
	\mathcal{S} ( D_{M_0}; t ) \subset D_{M(t)}.
	\end{equation}

		\end{itemize}

	\end{Proposition}

\end{mdframed}

\begin{Remark} \label{MRR1}
The exceptional set $\mathcal{R}$ in \eqref{M6} is the set of time where the total energy is not left-continuous.

\end{Remark}

In view of the weak strong uniqueness property (see e.g. \cite[Chapter 6, Section 6.3]{FeLMMiSh}),
\[
\mathcal{S}\Big[\vr_0, \vm_0, \mu, \lambda, \vc{g}; t \Big] = \Big[ \tvr(t, \cdot), \tvm(t, \cdot), \mu, \lambda, \vc{g} \Big] \ \mbox{for any}\ t \in [0, \tau)
\]
whenever the Navier--Stokes system admits a regular solution $(\tvr, \tvm)$ on a time interval $[0, \tau)$.

\subsection{Strong law of large numbers}
\label{S}

We suppose that the data are \emph{random} variables in $D$. More precisely, there is a complete probability space
\[
\Big[ \Omega, \mathfrak{B}, \mathcal{P} \Big]
\]
and a measurable mapping
\[
\vc{U}_0 : \omega \in \Omega \mapsto [\vr_0 (\omega), \vm_0(\omega), \mu(\omega), \lambda(\omega), \vc{g}(\omega) ] \in D \ \mbox{for a.a.}\ \omega \in \Omega.
\]
We set
\[
\vc{U}(t, \omega) = \mathcal{S} \left[ \vc{U}_0 (\omega) ; t \right],\ t \in [0, \infty),\ \omega \in \Omega,
\]
where $\mathcal{S}$ is a semigroup selection specified in Proposition \ref{PM1}.
{As $\mathcal{S}$ is (Borel) measurable}, $\vc{U}$ is a random process with continuous paths in $X$ and as such can be interpreted as a \emph{statistical solution}
of the Navier--Stokes system. Moreover, we have the implications
\begin{align} \label{S1}
\vc{U}^i_0,\ i \in I, \ \mbox{independent}\ &\Rightarrow \ \vc{U}^i(t) = \mathcal{S} \left[ \vc{U}^i_0  ; t \right],\ i \in I,
 \ \mbox{independent for any} \ t \geq 0, \ \br
\vc{U}^1_0 \sim \vc{U}^2_0 \ &\Rightarrow \  \mathcal{S} \left[ \vc{U}^2_0  ; t \right] \sim
\mathcal{S} \left[ \vc{U}^2_0  ; t \right],	
	\end{align}
where the symbol $\sim$ stands for equivalence in law.


Our next aim is to derive  statistical error estimates. We start by showing that weak statistical solutions of the Navier--Stokes system \eqref{i1}--\eqref{i5} are bounded in expectation under the following assumption
\begin{equation*}
	\expe{ \left( \intTd{ E(\vr_0, \vm_0)  } \right)^2 }< \infty,\ \ \| \vc{g} \|_{C(\Td; R^d)} \leq \Ov{g} \ \ \mbox{a.s.},\ \ \Ov{g} - \mbox{a deterministic constant.}
\end{equation*}
Boundedness of the second moment is needed because of the velocity controlled by means of
	\eqref{velocity_est}.

{Using} \eqref{bounds1} and \eqref{velocity_est} we obtain
\begin{equation}
\expe{\| \vr (t, \cdot) \|_{L^\gamma(\Td)}^\gamma } + \expe{\| \vm(t, \cdot) \|_{L^{\frac{2\gamma}{\gamma + 1}}(\Td; R^d)}^{\frac{2\gamma}{\gamma + 1}} }   {\leq} c\left(T, \overline{g} \right) \left( 1 + \expe{ \intTd{ E \Big( \vr_0, \vm_0 \Big)   } }\right),
\end{equation}

\begin{align}
\label{est4}
\expe{ \intT{  \norm{\vu}_{L^q(\Td; R^d)}^2 }  }  {\leq}
c\left(T, {\underline{\mu}}, \overline{g}  \right)\left(1 +  \expe{\left( \intTd{ E ( \vr_0, \vm_0 )   } \right)^2 }\right),
\end{align}
\begin{align}
\label{est5}
\expe{   \norm{\vu}_{L^2(0,T; W^{1,2}(\Td;R^d))}^2   }  \leq
c\left(T,{\underline{\mu}},  \overline{g} \right)\left(1 +  \expe{ \left(\intTd{ E ( \vr_0, \vm_0 )   }\right)^2 }\right),
\end{align}
where {$q = 6$ if $d=3$ and $q \geq 1$ arbitrary finite if $d=2.$}
With the above estimates we are ready to show the boundedness of the zero mean of  random solutions applying  Jensen's inequality
\begin{align}
&\expe{\Big\|  \vr (t, \cdot) - \expe{\vr (t, \cdot)} \Big\|_{L^\gamma(\Td)}^\gamma }
\aleq {\expe{\Big\|  \vr (t, \cdot) \Big\|_{L^\gamma(\Td)}^\gamma }}
\br
& {\leq} c\left(T,  \overline{g} \right) \left( 1 + \expe{ \intTd{ E \Big( \vr_0, \vm_0 \Big)   } }\right).
	 \label{est1}
\end{align}
Analogously, we have
\begin{align}
&\expe{\Big\|  \vm (t, \cdot) - \expe{\vm (t, \cdot)} \Big\|_{L^{\frac{2\gamma}{\gamma +1}}(\Td;R^d))}^{\frac{2\gamma}{\gamma +1}} }\aleq   c\left(T, \overline{g} \right) \left( 1 + \expe{ \intTd{ E \Big( \vr_0, \vm_0 \Big)   } }\right)
\label{est1a}
\end{align}
%
%
%
and
\begin{align}
& \expe{\Big\|  \vu  - \expe{\vu } \Big\|_{L^2(0,T; W^{1,2}( \Td;R^d))}^2 }
\aleq   c\left(T,  \underline{\mu}, \overline{g} \right) \left( 1 + \expe{ \left(\intTd{ E \Big( \vr_0, \vm_0 \Big)   } \right)^2 }\right).
 \label{est1b}
\end{align}

The space $L^2(0,T; W^{1,2}( \Td ))$ is a Hilbert space and the spaces
$L^\gamma(\Td)$, $L^{\frac{2 \gamma}{\gamma + 1}}(\Td; R^d)$ are separable reflexive Banach spaces, in particular, the Borel sets generated by the $W^{-k,2}$-topology and the strong topology are the same
on $D$. As a direct consequence of the Strong law of large numbers for random variables ranging in a separable Banach space, we obtain the following result, see Ledoux, Talagrand \cite[Corollary 7.10]{LedTal}.

\begin{Proposition}[Strong law of large numbers] \label{PS1}
	Suppose that $\vc{U}_0^n$, $n=1,2,\dots$ are i.i.d. (independent, identically distributed) copies of  random data
	\[
	\vc{U}_0 = \Big[ \vr_0, \vm_0, \mu , \lambda, \vc{g} \Big] \in D
	\]
	such that
\begin{align}
&	\expe{ \left( \intTd{ E(\vr_0, \vm_0)  } \right)^2 }< \infty,\ \ \| \vc{g} \|_{C(\Td; R^d)} \leq \Ov{g} \ \ \mbox{a.s.},\ \ \Ov{g} - \mbox{a deterministic constant},
\label{C1}
\end{align}
		Then for
	\[
	\vc{U}^n = \mathcal{S} \left( \vc{U}^n_0; t \right) = [\vr^n(t, \cdot), \vm^n(t, \cdot), \mu^n , \lambda^n, \vc{g}^n ]
	\]
	there {hold}
	\[\begin{aligned}
	\frac{1}{N} \sum_{n=1}^N \vr^n (t, \cdot) &\to \expe{\vr (t, \cdot)} \ \mbox{in}\ L^\gamma(\Td),
	\\
		\frac{1}{N} \sum_{n=1}^N \vm^n (t, \cdot) &\to \expe{\vm (t, \cdot)} \ \mbox{in}\ L^{\frac{2 \gamma}{\gamma + 1}} (\Td; R^d)),
\\
		\frac{1}{N} \sum_{n=1}^N \vu^n  &\to \expe{\vu} \ \mbox{in}\ L^2(0,T; W^{1,2} (\Td; R^d))
	\end{aligned}
	\]
	as $N \to \infty$ $\prst-$a.s.,
	where $\vr$, $\vm\,(= \vr \vu)$ are determined as
	\[
	[ \vr(t, \cdot) , \vm (t, \cdot) , \mu, \lambda, \vc{g} ]  =	\mathcal{S} \Big[ \vr_0,\vm_0, \mu, \lambda, \vc{g} ; t \Big] \
\mbox{ for } \ t \in [0,T].
	\]
	
	\end{Proposition}

\noindent Further, applying \cite[Proposition~9.11]{LedTal} we obtain
\begin{align} \label{S5}
\expe{ \left\| \frac{1}{N}  \sum_{n=1}^N \Big( \vr^n (t, \cdot) - \expe{\vr (t, \cdot)} \Big) \right\|_{L^\gamma(\Td)}^r } &\leq C N^{1 -r} \left( \frac 1 N  \sum_{n=1}^N \expe{ \left\| \Big( \vr^n (t, \cdot) - \expe{\vr (t, \cdot)} \Big) \right\|_{L^\gamma(\Td)}^r } \right) \br
&\ \aleq N^{1-r} \qquad r=\min\{2,\gamma\},
\end{align}
where we have used \eqref{est1} in the last inequality.
Similarly, as the expected value of the momentum and velocity is bounded by the expected value of the (initial) energy, we again obtain by applying \cite[Proposition~9.11]{LedTal}
\begin{equation} \label{S6}
\expe{ \left\| \frac{1}{N}  \sum_{n=1}^N \Big( \vm^n (t, \cdot) - \expe{\vm (t, \cdot)} \Big) \right\|_{L^{\frac{2 \gamma}{\gamma +1}}(\Td;R^d)}^{\frac{2 \gamma}{\gamma +1}} }
\aleq N^{\frac{1-\gamma}{\gamma +1}},
\end{equation}

\begin{equation} \label{S6b}
\expe{ \left\| \frac{1}{N}  \sum_{n=1}^N \Big( \vu^n  - \expe{\vu } \Big) \right\|_{L^2(0,T;W^{1,2} (\Td;R^d))}^2}
\aleq N^{-1}.
\end{equation}
Note that the {estimates \eqref{S5} and \eqref{S6}}
 are uniform in $t \in [0,T].$

In addition, as the $L^q$ spaces are of the type $q=\min(q,2)$ in the sense of \cite[Chapter 9]{LedTal}, we may use
\cite[Theorem 9.21]{LedTal} to strengthen the conclusion of Proposition \ref{PS1} to
\begin{align}
	N^r &\left\| \frac 1 N \sum_{n=1}^N \Big( \vr^n (t, \cdot) - \expe{\vr (t, \cdot)} \Big) \right\|_{L^\gamma(\Td)} \  \to 0 \  &\mbox{as}\ N \to \infty, \ \ t \in [0,T], \label{S2a}\\
	N^{\frac{\gamma - 1}{2\gamma}}&\left\| \frac 1 N \sum_{n=1}^N \Big( \vm^n (t, \cdot) - \expe{\vm (t, \cdot)} \Big) \right\|_{L^{\frac{2 \gamma}{\gamma + 1}}(\Td; R^d)}  \  \to 0 \  &\mbox{as}\ N \to \infty, \ \ t \in [0,T],
\label{S2} \\
N^s &  \left\| \frac 1 N \sum_{n=1}^N \Big( \vu^n  - \expe{\vu } \Big) \right\|_{L^2(0,T;W^{1,2}(\Td; R^d))} \   \to 0 \  &\mbox{as}\ N \to \infty,
\label{S2b}
	\end{align}
$\prst$-a.s., where $r= {\frac{\gamma - 1}{\gamma}} \mbox{  if } \gamma <  2 \ \ \mbox{ or }\ \  r < \frac 1 2  \mbox{  if } \gamma \geq  2, $
$s < \frac  1 2$ and $t \in [0,T].$

Note that  \eqref{S2a} -- \eqref{S2} hold
for any $t \in [0,T]$. The estimates \eqref{S5} -- \eqref{S6b} and \eqref{S2a} -- \eqref{S2b} represent the \emph{statistical errors} of the Monte Carlo method.

\subsection{Central limit theorem}
\label{C}

For completeness, we state a variant of the Central limit theorem.
This result requires the Hilbert topology and holds on condition that the second moments are bounded.

\begin{Proposition}[Central limit theorem]
	Suppose that $\vc{U}_0^n$, $n=1,2,\dots$ are i.i.d.  copies of  random data
	\[
	\vc{U}_0 = \Big[ \vr_0, \vm_0, \mu , \lambda, \vc{g} \Big] \in D
	\]
	such that assumption \eqref{C1} holds.
	
{Then for}
\[
\vc{U}^n = \mathcal{S} \left( \vc{U}^n_0; t \right) = [\vr^n(t, \cdot), \vm^n(t, \cdot), \mu^n , \lambda^n, \vc{g}^n ]
\]
{there holds:}
%
\begin{align}
\frac{1}{\sqrt{N}} \sum_{n=1}^N \left( \vr^n(t, \cdot) - \expe{ \vr(t, \cdot) } \right) &\to \mathfrak{R} \ \mbox{in law in}\ W^{-k,2}(\Td) \ \mbox{as}\ N \to \infty, \br
\frac{1}{\sqrt{N}} \sum_{n=1}^N \left( \vm^n(t, \cdot) - \expe{ \vm(t, \cdot) } \right) &\to \mathfrak{M} \ \mbox{in law in}\ W^{-k,2}(\Td; R^d) \ \mbox{as}\ N \to \infty, \ \mbox{ for all } t \in [0,T] ,\br
\frac{1}{\sqrt{N}} \sum_{n=1}^N \left( \vu^n  - \expe{ \vu } \right) &  \to \mathfrak{U} \ \mbox{in law in}\ L^2(0,T; W^{1,2}(\Td; R^d)) \ \mbox{as}\ N \to \infty.
\label{C3}	
		\end{align}	
{In particular, we improve the convergence rate:}
\begin{align}
&	N^{1/2} \left\| \frac 1 N \sum_{n=1}^N \Big( \vr^n (t, \cdot) - \expe{\vr (t, \cdot)} \Big) \right\|_{W^{-k,2}(\Td)} \aleq 1\ \  \mbox{as}\ N \ \to \infty, \ k > \frac d 2 \\
&	N^{1/2} \left\| \frac 1 N \sum_{n=1}^N \Big( \vm^n (t, \cdot) - \expe{\vm (t, \cdot)} \Big) \right\|_{W^{-k,2}(\Td; R^d)} \aleq 1 \  \  \mbox{as}\ N \ \to \infty, \  k > \frac d 2\\
&  	N^{1/2}\left\| \frac 1 N \sum_{n=1}^N \Big( \vu^n  - \expe{\vu } \Big) \right\|_{L^2(0,T;W^{1,2}(\Td; R^d))} \aleq 1 \  \  \mbox{as}\ N \ \to \infty,
	\end{align}
$\prst$-a.s.

\end{Proposition}

\begin{proof}
Recall that  $\expe{ \left( \intTd{ E(\vr_0, \vm_0)  } \right)^2 }< \infty.$  Then applying the embedding
$$
L^1(\Td) \hookrightarrow \hookrightarrow W^{-k,2}(\Td), \qquad k > d/2
$$
and using similar estimates as in
 \eqref{est1} we {obtain}
$$ \expe{ \Big \| \vr (t, \cdot) - \expe{\vr (t, \cdot) } \Big \|^2_{W^{-k,2}(\Td)} } < \infty. $$
Consequently, \cite[Theorem~10.5]{LedTal} yields {the} desired result for the density. Analogous result holds
for the momentum, too.  We note that since the second moment of the velocity is bounded in the Hilbert topology the {Central} limit theorem applies directly.
\end{proof}

\subsection{The $k$-th central statistical moments} \label{high}

{We extend the statistical
convergence to the $k$-th moments} under the assumption
\begin{equation}
\expe{ \left( \intTd{ E(\vr_0, \vm_0)  } \right)^{2k} }< \infty.
\end{equation}
We start by introducing suitable notation.
For $k \in \mathbb{N}$ and a separable Banach space $X$ we denote by
\begin{equation}
X^{(k)} = \underbrace{X \otimes \cdots \otimes X}_{k  \mbox{ times}}
\end{equation}
the $k$-fold tensor product of $k$ copies of $X$, which is equipped with a cross norm $\| {\cred \otimes}  \|_{X^{(k)} }$
\begin{equation}
\|  f_1  \otimes \cdots \otimes  f_k \|_{X^{(k)} } = \|  f_1 \|_{X} \cdots  \|  f_k \|_{X}.
\end{equation}
Below we will use $X = L^{\gamma}(\Td),\,  X= L^{\frac{2 \gamma}{\gamma + 1}}(\Td; R^d)$ and $X=L^2(0,T; W^{1,2}(\Td;R^d))$ for the density, momentum, and velocity, respectively.


Applying \eqref{bounds1} and \eqref{velocity_est} we get the following estimates
on the higher order moments:
\begin{align*}
& \expe{\Big\|  \vr (t, \cdot)  \Big\|_{L^\gamma(\Td)}^{k \gamma } } + \expe{\Big\|  \vm (t, \cdot)  \Big\|_{L^{ \frac{2\gamma}{\gamma +1}}(\Td;R^d)}^{k (\frac{2\gamma}{\gamma +1})  } }
\aleq c\left(T, \overline{g}\right)\left(1 +  \expe{\left(\intTd{ E \Big( \vr_0, \vm_0 \Big)   }\right)^k }\right),
\br &\quad \phantom{mmmmmmmmmmmmmmmmmmmmmmmmmmmmmmmmmmmmmmmm} t\in[0,T],
\end{align*}
\begin{equation*}
\expe{\Big\|  \vu  \Big\|_{L^2(0,T; W^{1,2}(\Td; R^d))}^{2k }}
\aleq  c\left(T, { \underline{\mu}}, \overline{g}\right)\left(1 +  \expe{\left(\intTd{ E \Big( \vr_0, \vm_0 \Big)   }\right)^{2k} }\right).
\end{equation*}

To derive estimates on the $k-$th statistical moment, it is convenient to identify the product
space $X^{(k)}$ of functions of the variable $x \in \Td$ with a subspace of functions
defined on $\mathbb{T}^{dk}$, specifically,
\[
f_1 \otimes \dots \otimes f_k \approx f_1(x_1) \dots f_k(x_k) ,\ (x_1, \dots, x_k) \in \mathbb{T}^{dk}.
\]
Keeping this convention in mind, we deduce the following
estimate on the $k$-th central moments:
\begin{align*}
&\expe{\Big\|  \vr^{(k)} (t, \cdot) - \expe{\vr^{(k)}  (t, \cdot)}\Big\|_{L^\gamma(\Tkd)}^\gamma } \aleq  \expe{\Big\|  \vr^{(k)} (t, \cdot) \Big\|_{L^\gamma(\Tkd)}^\gamma + \Big\|  \expe{\vr^{(k)}  (t, \cdot)} \Big\|_{L^\gamma(\Tkd)}^\gamma }
\br
\aleq  &\ \expe{\Big\|  \vr^{(k)} (t, \cdot) \Big\|_{L^\gamma(\Tkd)}^\gamma + \expe{ \Big\|  \vr^{(k)}  (t, \cdot)} \Big\|_{L^\gamma(\Tkd)}^\gamma } = 2\expe{\Big\|  \vr^{(k)} (t, \cdot) \Big\|_{L^\gamma(\Tkd)}^\gamma  } =  2\expe{\Big\|  \vr (t, \cdot) \Big\|_{L^\gamma(\Td)}^{k\gamma}  }
\br
\aleq &\  c\left(T, \overline{g}
 \right)\left(1 +  \expe{ \left(\intTd{ E \Big( \vr_0, \vm_0 \Big)   }\right)^k }\right), \quad t\in[0,T],
	\end{align*}
and similarly
\begin{align*}
&\expe{\Big\|  \vm^{(k)} (t, \cdot) - \expe{\vm^{(k)}  (t, \cdot)} \Big\|_{L^{ \frac{2\gamma}{\gamma +1}}(\Tkd;R^d)}^{\frac{2\gamma}{\gamma +1}} }
\aleq  c\left(T, \overline{g} \right)\left(1 +  \expe{ \left(\intTd{ E \Big( \vr_0, \vm_0 \Big)   }\right)^k }\right)  \quad t\in[0,T],
\br
& \expe{\Big\|  \vu^{(k)} - \expe{\vu^{(k)}} \Big\|_{L^2(0,T, W^{1,2} (\Tkd;R^d))}^2}
\aleq  c\left(T,\underline{\mu}, \overline{g} \right)\left(1 +  \expe{ \left(\intTd{ E \Big( \vr_0, \vm_0 \Big)   }\right)^{2k} }\right).
\end{align*}

Further, applying \cite[Proposition~9.11]{LedTal} to the Banach spaces $L^\gamma(\Tkd)$ and $L^{\frac{2 \gamma}{\gamma +1} }(\Tkd; R^d) $ we
obtain
\begin{align}
&\expe{ \left\| \frac{1}{N}  \sum_{n=1}^N \Big[ \Big(\vr^{(k)} \Big)^n (t, \cdot) - \expe{\vr^{(k)}  (t, \cdot)}\Big] \right\|_{L^\gamma(\Tkd)}^r } \aleq N^{1-r} \qquad r=\min\{2,\gamma\}, \\
&\expe{ \left\| \frac{1}{N}  \sum_{n=1}^N \Big[ \Big(\vm^{(k)} \Big)^n (t, \cdot) - \expe{\vm^{(k)}  (t, \cdot)} \Big] \right\|_{L^{\frac{2 \gamma}{\gamma +1}}(\Tkd;R^d)}^{\frac{2 \gamma}{\gamma +1}} }
\aleq N^{\frac{1-\gamma}{\gamma +1}} \quad t\in[0,T].
\end{align}
Analogous results follow for the velocity, i.e.
$$
\expe{ \left\| \frac{1}{N}  \sum_{n=1}^N \Big[ \Big(\vu^{(k)} \Big)^n - \expe{\vu^{(k)}  } \Big] \right\|^2_{L^2(0,T; W^{1,2}(\Tkd;R^d))}}
 \aleq N^{-1}.
$$
Applying
\cite[Theorem 9.21]{LedTal} we also derive the convergence of the $k$-th central moments
\begin{align}
	N^r & \left\| \frac{1}{N}  \sum_{n=1}^N \Big[ \Big(\vr^{(k)} \Big)^n (t, \cdot) - \expe{\vr^{(k)}  (t, \cdot)}\Big] \right\|_{L^\gamma(\Tkd)}  \to 0 \  \mbox{as}\ N \to \infty,
\\
& \qquad \qquad \qquad \qquad \qquad r= {\frac{\gamma - 1}{\gamma}} \mbox{  if } \gamma <  2 \ \ \mbox{ or }\ \  r < \frac 1 2  \mbox{  if } \gamma \geq  2, \quad \quad t\in[0,T],
\br
	N^{\frac{\gamma - 1}{2\gamma}}&\left\| \frac{1}{N}  \sum_{n=1}^N \Big[ \Big(\vm^{(k)} \Big)^n (t, \cdot) - \expe{\vm^{(k)}  (t, \cdot)} \Big] \right\|_{L^{\frac{2 \gamma}{\gamma +1}}(\Tkd;R^d)}  \to 0 \  \mbox{as}\ N \to \infty,
\quad t\in[0,T],
\nonumber
\br
& N^s \left\| \frac{1}{N}  \sum_{n=1}^N \Big[ \Big(\vu^{(k)} \Big)^n - \expe{\vu^{(k)}  } \Big] \right\|_{L^2(0,T; W^{1,2}(\Tkd;R^d))}
 \to 0 \  \mbox{as}\ N \to \infty, \quad s < \frac 1 2,
	\end{align}
$\prst$-a.s.

\subsection{Deviation and variance} \label{dev_var}
After showing the convergence of the $k$-th central  moments we proceed to study the  first deviation of the density and momentum as well as the variance of the velocity
\[
\mbox{Dev}({\vr}) \equiv \expe{\Big| \vr - \expe{\vr} \Big|}
\quad
\mbox{Dev}({\vm}) \equiv \expe{\Big| \vm -  \expe{\vm}  \Big|},
\quad
\mbox{Var}({\vu}) \equiv  \expe{ \left| \vu - \expe{\vu} \right|^2 }  .
\]
{Note carefully that these are deterministic functions of $(t,x)$.}

Our aim is to study the convergence of  MC estimators of the deviation and  variance, i.e. we investigate the behaviour of
\begin{eqnarray}
&&\left\| \frac 1 N \sum_{n=1}^N \Big| \vr^n (t, \cdot) - \frac{1}{N}  \sum_{m=1}^N \vr^m (t, \cdot) \Big| \ -  \mbox{Dev}(\vr(t, \cdot))  \right\|_{L^\gamma(\Td)} ,
\br
&&\left\| \frac 1 N \sum_{n=1}^N \Big| \vm^n (t, \cdot) - \frac{1}{N}  \sum_{m=1}^N \vm^m (t, \cdot) \Big| \ -  \mbox{Dev}(\vm(t, \cdot))  \right\|_{L^{\frac{2\gamma}{\gamma + 1}}(\Td; R^d)}, \qquad \mbox{ for all } \ t \in [0,T]
\br
&& \qquad \mbox{  and  }
\br
&& \left \| \frac{1}{N-1}  \sum_{n=1}^N  \left| \vu^n   - \frac{1}{N}  \sum_{m=1}^N \vu^m \right|^2  -  \mbox{Var}(\vu) \right \|_{L^1(0,T; L^3(\Td;R^d))} \qquad \mbox{ as }  N \to \infty.
\end{eqnarray}

\noindent First, let us consider the following i.i.d. random variables
\begin{equation*}
Y^n := \Big| \vr^n (t, \cdot) - \expe{\vr (t, \cdot)} \Big| \ - \mbox{Dev}(\vr),\
\expe{Y^n} = 0.
\end{equation*}
Applying \eqref{est1} and Jensen's inequality we obtain
$\expe{ \| Y^n (t, \cdot) \|^\gamma_{L^\gamma(\Td)}} < \infty .$
%
Thus, we can use \cite[Theorem 9.21]{LedTal} which yields for all $t \in [0,T]$
\begin{align*}
	N^r & \left\| \frac 1 N \sum_{n=1}^N Y^n (t, \cdot) \right\|_{L^\gamma(\Td)} \to 0 \ \  \mbox{ as }\ N \to \infty \quad \prst-\mbox{a.s.},
\quad  r= {\frac{\gamma - 1}{\gamma}} \mbox{  if } \gamma <  2 \ \ \mbox{ or }\ \  r < \frac 1 2  \mbox{  if } \gamma \geq  2.
\end{align*}

This result together with \eqref{S2a} leads to the convergence  of the MC estimator for the first deviation of the density
\begin{eqnarray}
&&N^r \left\| \frac 1 N \sum_{n=1}^N \Big| \vr^n (t, \cdot) - \frac{1}{N}  \sum_{n=1}^N \vr^n (t, \cdot) \Big| \ - \ \mbox{Dev}(\vr)
 \right\|_{L^\gamma(\Td)} \aleq
 \br
&& N^r \left\| \frac 1 N \sum_{n=1}^N Y^n (t, \cdot) \right\|_{L^\gamma(\Td)}  +
	N^r \left\| \frac 1 N \sum_{n=1}^N \Big( \vr^n (t, \cdot) - \expe{\vr (t, \cdot)} \Big) \right\|_{L^\gamma(\Td)}
\to 0 \  \mbox{ as }\ N \to \infty, \nonumber
\\ \label{dev1}
&& r= {\frac{\gamma - 1}{\gamma}} \mbox{  if } \gamma <  2 \ \ \mbox{ or }\ \  r < \frac 1 2  \mbox{  if } \gamma \geq  2  \qquad \quad \prst-\mbox{a.s.}
\end{eqnarray}

\noindent Analogous analysis  yields the convergence of the  MC estimator of the deviation for the moment
\begin{align}
&	N^{\frac{\gamma - 1}{2\gamma}}  \left\| \frac 1 N \sum_{n=1}^N \Big| \vm^n (t, \cdot) - \frac{1}{N}  \sum_{n=1}^N \vm^n (t, \cdot)\Big| \ - \ \expe{ \Big| \vm (t, \cdot) -  \expe{\vm (t, \cdot)}  \Big|} \right\|_{L^{\frac{2 \gamma}{\gamma + 1}}(\Td; R^d)}  \to 0 \  \mbox{as}\ N \to \infty. \label{dev2}
\end{align}
$\prst$-a.s.

\medskip
We close this subsection by analyzing the behaviour of the unbiased MC estimator of the variance of velocity. Let us consider
following i.i.d. random variables
\[
Z^n=   \big| \vu^n   - \expe{\vu} \Big|^2  - \mbox{Var}(\vu)
\]
with
$
\expe{Z^n} = 0.
$
Next, using \eqref{est4} we obtain by straightforward calculations
\[
 \expe{\norm{ Z^n}_{L^1(0,T; L^3(\Td; R^d))} } \aleq  \expe{\norm{\vu}_{L^2(0,T; L^6(\Td; R^d))}^2 } < \infty.
\]
Thus, by the Strong law of large numbers~\cite[Theorem~9.21]{LedTal}  we obtain
$$
\norm{ \frac 1 N \sum_{n=1}^N  Z^n }_{L^1(0,T; L^3(\Td; R^d))}  \to 0 \qquad \mbox{ as } \ N \to \infty \qquad \prst-\mbox{ a.s. }
$$

Further, applying the triangular inequality,  \eqref{S2b} and boundedness of $\mbox{Var}(\vu)$, cf.~\eqref{est4}, we derive
\begin{eqnarray}
&&\left \| \frac{1}{N-1}  \sum_{n=1}^N  \left( \vu^n   - \frac{1}{N}  \sum_{m=1}^N \vu^m \right)^2  -  \mbox{Var}(\vu) \right \|_{L^1(0,T; L^3(\Td;R^d))} \aleq \frac{ N} {N-1}\norm{ \frac 1 N \sum_{n=1}^N  Z^n }_{L^1(0,T; L^3(\Td; R^d))}
\br
&& +\norm{ \left(\frac 1 N \sum_{n=1}^N  \vu^n  - \expe{\vu} \right)^2 }_{L^1(0,T; L^3(\Td; R^d))} + \frac{1}{N-1} \norm{ \mbox{Var}(\vu)}_{L^1(0,T; L^3(\Td; R^d))} \qquad \to 0  \
\nonumber \\ \label{var}
&& \hspace{9cm}   \mbox{ as } \ N\to \infty \qquad \prst-\mbox{ a.s. }
\end{eqnarray}
If $d=2$,  \eqref{var}  holds in $L^1(0,T; L^q(\Td)),$ $1\leq q < \infty.$

\section{Finite volume approximations}
\label{FVA}

Exact solutions of the Navier--Stokes system \eqref{i1}--\eqref{i5} will be approximated by a suitable structure preserving numerical method. To illustrate  the ideas, we concentrate on the upwind finite volume method but any consistent approximation satisfying {the} below mentioned structure preserving properties can be applied as well. In particular, results presented in what follows also apply to the Marker-and-Cell (MAC) finite difference method, see \cite{FeLMMiSh,BS_1,BS_2}.

\subsection{Finite volume method}
\label{FV}

A physical domain $\Td$ is decomposed into finite volumes (cuboids for simplicity)
\[
 \Td  = \bigcup_{K \in \grid_h} K.
\]
Here $h  \in (0,1)$ is a mesh parameter which means that
$|K| \approx h^d.$
The set of all faces $\sigma \in \partial K,$  $K \in \grid_h$ is denoted by $\Sigma,$
$
|\sigma| \approx h^{d-1}.
$

We will work with a piecewise constant approximation in space and denote by ${Q}_h$
the space of  functions constant on each element $K \in \grid_h.$  The associated
projection reads
\[
\Pi_h: L^1(\Td) \to {Q}_h,\ \Pi_h v = \sum_{K \in \grid_h} \mathds{1}_K \frac{1}{|K|} \int_K
v \dx.
\]
To approximate differential operators {in \eqref{i1}, \eqref{i2}},  we need  to define corresponding discrete differential operators. To this end we first introduce
the average and jump operators on any face $\sigma \in \Sigma$
\[
\avs{v} = \frac{v^{\rm in} + v^{\rm out} }{2},\ \ \
\jump{ v }  = v^{\rm out} - v^{\rm in}, \qquad v \in Q_h,
\]
where $v^{\rm out}, v^{\rm in}$ are respectively the outward, inward limits with respect to a given normal $\vn$ to $\sigma \in \Sigma.$
We now proceed by introducing discrete differential operators
for piecewise constant functions $r_h \in Q_h,$ $\vvh \in \vQh \equiv( Q_h)^d$:
\begin{equation*}
\begin{aligned}
& \gradd r_h =  \sum_{\sigma\in\pd K} \left(\gradd r_h \right)_{\sigma} 1_\sigma,  \quad \left(\gradd r_h\right) _{\sigma} =\frac{\jump{r_h} }{h} \vc{n}
\\
&\Divh \vvh  = \sum_{K \in \grid_h}  (\Divh  \vvh)_K 1_K, \quad
(\Divh \vvh)_K = \sum_{\sigma\in \pd K} \frac{|\sigma|}{|K|} \avs{\vvh} \cdot \vn.
\\
 \end{aligned}
\end{equation*}

In our finite volume method we approximate convective terms  by a dissipative upwind numerical flux denoted by $F_h;$ specifically
\begin{eqnarray*}
F_h(r_h,\vvh)
 = \avs{r_h} \ \avs{\vv_h} \cdot \vc{n}
- \left( \muh + \frac{1}{2} |\avs{\vv_h} \cdot \vc{n}| \right)\jump{ r_h },
\ -1 < \varepsilon .
\end{eqnarray*}
Analogously, we define the vector-valued numerical flux ${\bf F}_h({\bf r}_h,\vvh)$ componentwisely.

Further, time evolution is approximated by the implicit Euler method. Let $\Delta t > 0,$ $\Delta t \approx h,$  be a time step and time instances be denoted as
$
t_k = k \Delta t,\ k=1,2,\dots, N_T.
$
We set
\begin{align*}
v^k(x) = v(t^k,x)  \ \mbox{ for all } \  x\in \Td,\ t^k=k\,\Delta t \ \mbox{ for } k=0,1, \ldots, N_T
\end{align*}
and approximated time derivative $\frac{\partial {v} }{\partial t}$  by the backward Euler finite difference
\[
\frac{\partial {v} }{\partial t} \approx D_t {v}^k \equiv \frac{ {v}^k - {v}^{k-1} }{\Delta t}.
\]
We introduce a piecewise constant interpolation in time of the discrete values $v^k$,
\begin{align}\label{NM_TD}
v_h(t,\cdot) = v_0 \mbox{ for } t<\Delta t,\ &
v_h(t,\cdot)=v^k \mbox{ for } t\in [k \Delta t,(k+1) \Delta t),\ k=1,2,\ldots,N_T.
\end{align}

We are ready to introduce the upwind finite volume method that will be used to approximate the Navier--Stokes system \eqref{i}.
\begin{Definition}[{\bf FV method}] \label{DD1}
\phantom{mm}

Given initial data $(\vr_{0}, \vm_{0}) \in L^1(\Td; R^{d+1})$ are approximated by the piecewise constant projection
\[
\vr_{h}^0 = \Pi_h \vr_0,\ \vm_{h}^0 = \Pi_h \vm_0, \ \vr_{h}^0 \vu_{h}^0 = \vm_{h}^0.
\]	
\medskip
\noindent A pair $(\vr_{h}, \vm_{h} = \vr_{h} \vu_{h})$ of piecewise constant functions (in space and time) is a numerical approximation of the Navier--Stokes system \eqref{i1}--\eqref{i5}  if the following system of discrete equations holds:
\begin{subequations}\label{scheme}
\begin{align}
&\intTd{ D_t \vr_{h} \varphi_h } - \sum_{ \sigma \in \Sigma } \intSh{  F_h(\vr_{h},\vu_{h})
\jump{\varphi_h}   } = 0 \quad \mbox{for all } \varphi_h \in {Q}_h,\label{scheme_den}\\
&\intTd{ D_t  (\vr_{h} \vu_{h}) \cdot \bfphi_h } - \sum_{ \sigma \in \Sigma } \intSh{ {\bf F}_h(\vr_{h} \vu_{h},\vu_{h})
\cdot \jump{\bfphi_h}   }- \sum_{ \sigma \in \Sigma } \intSh{  \avs{p(\vr_{h})} \vc{n} \cdot \jump{ \bfphi_h }  } \nonumber \\
&= - \mu \frac{1}{h} \sum_{ \sigma \in \Sigma } \intSh{ \jump{\vu_{h}}  \cdot
\jump{\bfphi_h}  } - \eta \intTd{ \Divh \vu_{h} \Divh \bfphi_h } + \intTd{ \vr_h (\Pi_h \vc{g}) \bfphi}
\quad \mbox{for all }
\bfphi_h \in \vQh . \label{scheme_mom}
\end{align}
\end{subequations}
where $\eta =  \frac {d-2} d \mu +\lambda$.
\end{Definition}

As reported in~\cite{FeiLukMizShe,FeLMMiSh} the FV method is structure preserving in the following sense.
\begin{itemize}
\item {\bf Positivity  of the discrete density}
\begin{align}
\label{positivity}
\vr_{h} (t) >  0  \ \mbox{ for any } t > 0 \mbox{ provided } \vr_{h}^{0} > 0
\end{align}
\item {\bf Discrete total energy dissipation}
\begin{multline} \label{dei}
\intTd{ E( \vrh, \vmh ) (\tau, \cdot)  }
+  \int_0^\tau   \left(  \mu   \norm{\gradd  \vu_{h} }_{L^2(\Td; R^{d\times d})}^2  + \eta \norm{\Divh \vu_{h} }_{L^2(\Td; R^{d\times d})}^2 \right) \dt
\\
\aleq c\left(T, \| \vc{g} \|_{C(\Td, R^d)} \right)\left(1 +   \intTd{ E \Big( \vr_0, \vm_0 \Big)   } \right).
\end{multline}
\end{itemize}

In the above estimate we have used  the convexity of $E(\vr, \vm)$  and  Jensen's inequality which lead to
$$
\intTd{ E(\vr_h^0, \vm_h^0 )} \leq \intTd{E(\vr_0, \vm_0)}.
$$
Application of  Gronwall's lemma implies the above discrete total energy dissipation~\eqref{dei}.

\subsection{Measurability}

The principal difficulty associated with a time implicit method, such as \eqref{scheme} is possible non--uniqueness that may occur even at the
level of approximate solutions. We denote by $\Vrmh$,
\[
\Big[ \vr_0, \vm_0 , \mu, \lambda, \vc{g} \Big] \in D \mapsto
\Vrmh \in 2^{ (Q_h \times \vc{Q}_h )^M },\ M \approx \frac{T}{\Delta t},
\]
the possible multivalued map that associates to the data $\Big[ \vr_0, \vm_0 , \mu, \lambda, \vc{g} \Big]$
the set of all FV approximations at the level $h$. Note carefully that the range of this mapping is isomorphic to a finite dimensional Euclidean space for any fixed $\Delta t$, $h$. Moreover, the following properties are easy to check:
\begin{itemize}
	\item for each fixed data $\Big[ \vr_0, \vm_0 , \mu, \lambda, \vc{g} \Big] \in D$, the set
	$\Vrmh$ is non--empty and compact;
	\item if
	\begin{eqnarray*}
	\Big[ \vr^n_0, \vm^n_0 , \mu^n, \lambda^n, \vc{g}^n \Big]_{n=1}^\infty \in D_L \ \mbox{for some}\ L > 0, \phantom{mmmmmm}\br
	\Big[ \vr^n_0, \vm^n_0 , \mu^n, \lambda^n, \vc{g}^n \Big] \to
	\Big[ \vr_0, \vm_0 , \mu, \lambda, \vc{g} \Big] \mbox{ in }\ D, \mbox{ as }  n\to \infty \ \
	\end{eqnarray*}
	and
	\[
	(\vr^n_h, \vm^n_h) \in \{ \vrh^n,  \vmh^n \}
	\]
	is a FV numerical solution corresponding to the data $\Big[ \vr^n_0, \vm^n_0 , \mu^n, \lambda^n, \vc{g}^n \Big]$,
then there is a subsequence $n_k$ such that
	\[
	(\vrh^{n_k}, \vm_h^{n_k})  \to (\vr_h, \vm_h) \in \Vrmh,
	\]
	where $\Vrmh$ is the set of FV solutions corresponding to the data $\Big[ \vr_0, \vm_0 , \mu, \lambda, \vc{g} \Big]$.
	\end{itemize}

In particular, there is a measurable selection, specifically a Borel mapping
\[
\Big[ \vr_0, \vm_0 , \mu, \lambda, \vc{g} \Big] \in D \mapsto
(\vrh, \vmh) \in \Vrmh,
\]
see e.g. Bensoussan and Temam \cite[Theorem A.1]{BenTem}.
{Accordingly,
here and hereafter,} we consider only FV solutions $(\vrh, \vmh)$ which are Borel measurable functions of data.

\subsection{Convergence and error estimates of the FV method}

We consider regular initial data, specifically
\begin{equation} \label{F4}
	\vr_0 \in W^{3,2}(\Td),\ 0 < \underline{\vr} \leq \min_{\Td} \vr_0,\
\vm_0 \in W^{3,2}(\Td; R^d).	
	\end{equation}

\subsubsection{Deterministic data}

We report the following result on the convergence of the finite volume method \eqref{scheme} for deterministic
data, see~\cite[Theorem~11.3, Theorem~7.12]{FeLMMiSh}.

\begin{mdframed}[style=MyFrame]

\begin{Proposition} \label{PF1}
	Let the initial data $(\vr_0, \vm_0)$ belong to the class  \eqref{F4}, $\vc{g} \in C^\ell(\Td; R^d)$, $\ell \geq 3$, $\mu > 0$,
	$\lambda \geq 0$.
\begin{itemize}
\item[i)]  Consider FV solutions  $\left( \vrh, \vmh(= \vrh \vuh)\right)_{h \searrow 0}$ obtained by \eqref{scheme} satisfying
	\begin{equation} \label{F4bis}
	{ \sup_{h} \left(\| \vrh \|_{L^\infty((0,T)\times \Td)} + \|\vu_h\|_{L^\infty ((0,T)\times \Td; R^d)}  \right)< \infty}
		\end{equation}
	Then
	\begin{align}
 \left\| \vr_h - \vr \right\|_{L^r( (0,T) \times \Td)} +
\left\| \vm_h - \vm \right\|_{L^r( (0,T) \times \Td;R^d)} + \left\| \vu_h - \vu \right\|_{L^r((0,T) \times \Td; R^d)}
 \to 0 \ \mbox{as}\ h \to 0
	\nonumber
	\end{align}
for any $1 \leq r < \infty$, where $(\vr, \vu(\vm=\vr \vu))$ is a classical solution of the Navier--Stokes system \eqref{i1}--\eqref{i5}, specifically,
\[
\vr \in C([0,T]; W^{3,2}(\Td))\cap C^1([0,T]\times \Td),\ \vu \in C([0,T]; W^{3,2}(\Td; R^d))\cap C^1([0,T]\times \Td; R^d).
\]	
\item[ii)] {Suppose the} classical solution $(\vr, \vu)$
of the Navier--Stokes system \eqref{i1}--\eqref{i5} {emanating from the initial data $(\vr_0, \vm_0)$ exists on the time interval $[0,T]$.}
\\
Then the FV solutions $\left(\vrh, \vmh(= \vrh \vuh)\right)_{h \searrow 0}$  converge strongly  to the classical solution $(\vr, \vu)$
\begin{align}
 \hspace{-0.5cm}\left\| \vr_h - \vr \right\|_{L^r(0,T; L^\gamma (\Td))} +
\left\| \vm_h - \vm \right\|_{L^r(0,T; L^{ \frac{2\gamma}{\gamma + 1} }(\Td; R^d))} +
\left\| \vu_h - \vu \right\|_{L^2( (0,T) \times  \Td; R^d)}
 \to 0 \mbox{ as}\ h \to 0
\end{align}
{for any $1 \leq r < \infty$.}
\end{itemize}
	\end{Proposition}

\end{mdframed}

\medskip
In order to estimate the convergence rate of the FV method \eqref{scheme} we {use the concept of relative energy representing
a ``distance" between two solutions:}
\begin{eqnarray*}
&&\mathcal{E}\left(\vrh(t), \vmh(t)| \vr(t), \vu(t)\right)
\br
&&= \int_{\Td} \left( \frac 1 2 \vrh(t) |\vuh(t) - \vu(t) |^2 +  P(\vrh(t)) - P'(\vr(t)) (\vrh(t) -\vr(t) ) - P(\vr(t))  \right) \mbox{d}x.
\end{eqnarray*}
Here  $(\vr(t), \vu(t))$ is a classical solution of the Navier--Stokes system \eqref{i1}--\eqref{i5},
{and $(\vr_h, \vu_h)$ its numerical approximation.}
 For \emph{more regular} data
\begin{equation}
\vr_0 \in W^{6,2}(\Td),\ \vm_0 \in W^{6,2}(\Td;R^d),\  0 < \underline{\vr} \leq \min_{\Td} \vr_0, \ \vc{g} \in C^\ell(\Td; R^d), \ell \geq 6
\label{reg_init}
\end{equation}
the strong solution $(\vr, \vu) \in C([0,\tau]; W^{6,2}(\Td)) \times C([0,\tau]; W^{6,2}(\Td;R^d)),$ $0 < \tau \leq T$ exists
(at least locally in time), see \cite[Theorem~2.7]{BrFeHo2016}. The following result on the error estimates was proved in \cite{FLS_error}.\\

\noindent {\bf Error estimates}
%
\begin{eqnarray}
&&\hspace{-0.5cm}\sup_{0\leq t \leq \tau}\mathcal{E}(\vr_h,\vmh | \vr, \vu)
+\mu     \inttaO{|\gradd \vuh-  \Grad \vu|^2 }
+ \eta   \inttaO{|\Divh \vuh - \Div \vu|^2 }
\nonumber \\
&&\leq c ( h^\alpha + \Delta t^\beta), \label{error rates1}
\end{eqnarray}
{for certain exponents $\alpha, \beta$ specified below,}
where
\[
c=c( T, \| \vc{g} \|_{C^{\ell}(\Td; R^d)}, \| (\vr_0, \vm_0) \|_{W^{6,2}(\Td; R^{d+1})}, \underline{\vr}, \| (\vr, \vu) \|_{C([0,\tau] \times \Td; R^{d+1})}).
\]
It is remarkable that the estimates depend only on the data and the $L^\infty-$norm of the
	strong solution. In agreement with the conditional regularity result of Sun, Wang, and Zhang \cite{SuWaZh},
the strong solution exists as long as its $L^\infty-$norm remains bounded.

The convergence rate in space $\mathcal{O}(h^\alpha)$ depends in general on $\gamma.$ For example, it reduces to $\alpha=\frac 1 3 $ if $\gamma \searrow 1$ and $\alpha=1$ for $\gamma \geq 3,$ see \cite{FLS_error} for a precise formula for $\alpha$. The convergence rate for time discretization is {$\beta = \frac 1 2.$}

Moreover, we can also derive the estimates for the velocity applying the Sobolev-Poincar\'e inequality, see~Appendix~\ref{APDA} and
\cite[Lemma~A.2]{FLS_error}
\begin{eqnarray}
&&\norm{ \vuh -\vu  }_{L^2(0,\tau; L^q(\Td;R^d))} \aleq \norm{ \gradd \vuh - \Grad \vu  }_{L^2((0,\tau)\times \Td;R^{d\times d})}
\nonumber \\
&&\hspace{5cm}+ \sup_{t \in [0,\tau]}(\mathcal{E}(\vrh(t), \vmh(t)| \vr(t), \vu(t)))^{\frac 1 2 } + \mathcal{O}(h)
\label{error_q}
\end{eqnarray}
{with $q = 6$ if $d=3$ and $q \geq 1$ arbitrary finite if $d=2,$ \ $0 \leq \tau \leq T.$}

Further, assuming {uniform boundedness of the numerical solutions,}
$$
\| (\vr_h, \vm_h) \|_{L^\infty((0,T)\times \Td; R^{d+1})} \aleq 1,
$$
the global classical solution exists and {we have} the first order convergence rate in \eqref{error rates1}, i.e.~$\alpha = 1 = \beta.$
{As the relative energy is a strictly convex function of $(\vr, \vm)$, we get}
\begin{eqnarray} \label{est3}
&&\left\| \vr_h(t,\cdot) - \vr(t, \cdot) \right\|_{L^2(\Td)} + \left\| \vm_h(t, \cdot) - \vm(t, \cdot) \right\|_{L^2(\Td; R^d)}\aleq (\mathcal{E}(\vrh(t), \vmh(t)| \vr(t), \vu(t)))^{\frac 1 2 } \br
&&\phantom{mmmmmmmmmmmmmmmmmmmmmmm} \qquad \qquad \qquad \mbox{ for all } t \in [0,T].
\end{eqnarray}

In summary,  we have the following error estimates for uniformly bounded {FV numerical solutions}, see~\cite{FLS_error}.

\begin{mdframed}[style=MyFrame]

\begin{Proposition} \label{PF2}
 Let the initial data $(\vr_0, \vm_0)$ belong to the regularity class \eqref{reg_init}
and  $\vc{g} \in C^\ell(\Td; R^d)$, $\ell \geq 6$, $\mu > 0$,
	$\lambda \geq 0$.
Suppose that the Navier--Stokes system admits a classical solution $(\vr, \vu)$
in the class
\begin{equation} \label{Regular}
\vr \in  C([0,T]; W^{6,2}(\Td)) \cap C^1([0,T] \times \Omega) \quad  \vu \in C([0,T]; W^{6,2}(\Td;R^d)) \cap C^1([0,T] \times \Omega; R^d).
\end{equation}
Let $(\vrh, \vuh)$, $\vm_h = \vr_h \vu_h$ be the numerical solutions resulting from
the FV method \eqref{scheme}.

Then the following estimates hold:
\begin{eqnarray}
&&\hspace{-0.5cm}\sup_{0\leq t \leq T}\mathcal{E}(\vr_h,\vmh | \vr, \vu)(t,\cdot)
+\mu     \intTO{|\gradd \vuh-  \Grad \vu|^2 }
 \leq C\, ( h + \Delta t),
\br
&&\hspace{-0.5cm}\left\| \vr_h(t,\cdot) - \vr(t, \cdot) \right\|_{L^2(\Td)} + \left\| \vm_h(t, \cdot) - \vm(t, \cdot) \right\|_{L^2(\Td; R^d)}
\leq C\, \left(\sqrt h + \sqrt{ \Delta t}\right), \qquad  t \in [0,T],
\br
&& \left\| \vuh  - \vu \right\|_{L^2(0,T; L^q( \Td;R^d))}  \leq C\, \left(\sqrt h + \sqrt{ \Delta t}\right),
\br
&& \hspace{5cm}1 \leq q < \infty \ \mbox{ if }\ d= 2\ \ \mbox{ and } \ \ 1 \leq q \leq 6 \ \mbox{ if } \ d= 3,
\label{error rates3}
\end{eqnarray}
whenever $h \in (0,1)$, $\Delta t \in (0,1)$, where
\[
\hspace{-0.2cm}C=C( T, \| \vc{g} \|_{C^\ell(\Td; R^d)},  \| (\vr_0, \vu_0)\|_{W^{6, 2}(\Td; R^{d+1})}, \underline{\vr},
\| (\vr , \vu) \|_{C([0,T]\times \Td; R^{d+1})}
, \| (\vrh , \vuh) \|_{L^\infty((0,T)\times \Td; R^{d+1})})
\]
is a bounded function of bounded arguments.
	\end{Proposition}

\end{mdframed}

As observed in \cite{FLS_error}, classical solutions enjoy better regularity stated in \eqref{Regular}
as long as their $C-$norm remains bounded.

\subsubsection{Random data}
\label{rnd_data}

Now, consider random data belonging to the class \eqref{F4} a.s. for some deterministic constant $\underline{\vr} > 0.$
A relevant analogue of the boundedness hypothesis \eqref{F4bis} proposed in \cite{FeiLuk2021} is
boundedness in probability.

\begin{mdframed}[style=MyFrame]

\begin{Definition} [{\bf Boundedness in probability of FV solutions}] \label{DFF1}
	
	We say that a  sequence {$\left(\vrh, \vuh \right)_{h\searrow 0}$} is \emph{bounded in probability} if
		\begin{align}
		&\mbox{for any}\ \ep > 0, \ \mbox{there exists}\ M= M(\ep) \ \mbox{such that}
		\mbox{ for all } h \in (0,1)
        \br & {\prst\left(\left[  \| \vrh \|_{L^\infty((0,T)\times \Td) }+ \|\vu_h\|_{L^\infty((0,T)\times \Td; R^d)}  > M
		\right]\right) \leq \ep.} \label{bip}
	\end{align}
	
	\end{Definition}

\end{mdframed}

Following  the arguments of \cite{FeiLuk2021} we show convergence of the FV solutions
provided $(\vrh, \vmh)_{h \searrow 0}$ is bounded in probability in the sense of Definition \ref{DFF1}.

\begin{enumerate}
	
	\item We consider numerical solutions $(\vrh, \vm_h(=\vrh\vuh))$
Borel measurable with respect to the data
	\[
	\Big[ \vr_0, \vm_0, \mu, \lambda, \vc{g} \Big] \in D
	\]
	such that
	\begin{equation} \label{hyp1}
	\expe{ \left(\intTd{E(\vr_0, \vm_0) }\right)^2} < \infty,\,\, 
	 \| \vc{g} \|_{C(\Td; R^d)} \leq \Ov{g} \ \ \prst-\mbox{ a.s.}.
	\end{equation}

	
	\item
	
	Taking a subsequence of FV solutions $(\vrhk, \vu_{h_k})_{h_k \searrow 0}$ we consider a family of random variables
	\[
	\Big[ \vr_0, \vm_0, \mu, \lambda, \vc{g}, \vrhk, \vuhk, \Lambda_{h_k} \Big]_{h_k \searrow 0},
	\]
	where
	\[
	{ \Lambda_{h_k} = \| (\vrhk, \vuhk) \|_{L^\infty ((0,T) \times \Td; R^{d+1})}, }
	\]
	ranging in the Polish space
	\begin{eqnarray*}
	&& Y = W^{3,2} (\Td) \times W^{3,2} (\Td; R^d) \times R \times R \times C^{\ell}(\Td; R^d)
	\times W^{-k,2}((0,T) \times \Td) \times \br
      &&\qquad \qquad \qquad \ \times W^{-k,2}((0,T) \times \Td; R^d) \times R, \quad  \ell \geq 3.
	\end{eqnarray*}
	In view of {hypothesis \eqref{bip}}, the family of laws associated to
	\[
	\Big[ \vr_0, \vm_0, \mu, \lambda, \vc{g}, \vrhk,  \vuhk, \Lambda_{h_k} \Big]_{h_k \searrow 0}
	\]
is tight in $Y$. Applying the Skorokhod representation theorem \cite{Jakub} we conclude, exactly as in \cite[Section 5.1]{FeiLuk2021}, that there is a new probability space and a new sequence of random variables
\[
\Big[ \tvr_{0,h_k}, \tvm_{0,h_k}, \widetilde\mu_{h_k}, \widetilde \lambda_{h_k}, \widetilde{ \vc{g}}_{h_k}, \widetilde{\vr}_{h_k},  \widetilde{\vu}_{h_k}, \widetilde{\Lambda}_{h_k} \Big]	
\sim \Big[ \vr_0, \vm_0, \mu, \lambda, \vc{g}, \vrhk,  \vuhk, \Lambda_{h_k} \Big]
\]
satisfying
\begin{align}
\widetilde{\Lambda}_{h_k} &= {\| (\widetilde{\vr}_{h_k} , \widetilde{\vu}_{h_k} ) \|_{L^\infty((0,T)\times \Td; R^{d+1})} < \infty,} \  \br
\tvr_{0,{h_k}} &\to \tvr_0 \ \mbox{in}\ W^{3,2}(\Td), \br
\tvm_{0,{h_k}} &\to \widetilde{\vm}_0 \ \mbox{in}\ W^{3,2}(\Td; R^d), \br
\widetilde{\mu}_{h_k} &\to \widetilde{\mu},\ \widetilde{\lambda}_h \to \widetilde{\lambda},\
\widetilde{\vc{g}}_{h_k} \to \widetilde{\vc{g}} \ \mbox{in}\ C(\Td; R^d) , \br
\tvr_{h_k} &\to \tvr \ \mbox{in}\ L^r((0,T) \times \Td),\  1 \leq r < \infty \ , \br
\widetilde{\vu}_{h_k} & \to \widetilde{\vu} \ \mbox{in}\ L^r((0,T) \times \Td; R^d),\ 1 \leq r < \infty
\end{align}
a.s.,
where $(\tvr, \widetilde{\vu})$ is the classical solution of the Navier--Stokes system \eqref{i1}--\eqref{i5} corresponding to the data
\[
\Big[ \tvr_0, \tvm_0, \widetilde{\mu}, \widetilde{\lambda}, \widetilde{\vc{g}} \Big]
\sim \Big[ \vr_0, \vm_0, {\mu}, {\lambda}, {\vc{g}} \Big] .
\]
Here, the symbol $\sim$ denotes equivalence in law of random variables.

\item

The convergence of numerical solutions in the preceding step is unconditional, meaning once the
convergence of the data is given, there is no need to extract a subsequence as the limit is unique. Consequently, by means of the Gy\" ongy--Krylov theorem \cite{Gkrylov}, exactly as in \cite[Theorem 2.6]{FeiLuk2021}, we recover unconditional convergence in the original
probability space,
\begin{align}
	\| \vr_h - \vr \|_{L^r((0,T) \times \Td)} &\to 0 \ \mbox{in probability} \br
	\| \vm_h - \vm \|_{L^r((0,T) \times \Td; R^d)} &\to 0\  \mbox{in probability}  \br
   \| \vu_h - \vu \|_{L^r((0,T) \times \Td; R^d)} &\to 0\  \mbox{in probability}
	\label{F6}
\end{align}
for any $1 \leq r < \infty$, where $(\vr, \vm)$ is a classical solution of the Navier--Stokes
system \eqref{i1}--\eqref{i5} with the initial data $(\vr_0, \vm_0)$. In particular,
\[
(\vr, \vm)(t, \cdot) = \mathcal{S} \left( [\vr_0, \vm_0, \mu, \lambda, \vc{g}; t ] \right).
\]

	\end{enumerate}

As a byproduct of the above arguments, we see that the Navier--Stokes system \eqref{i1}--\eqref{i5} admits global in time classical solution for the data $[\vr_0, \vm_0, \mu, \lambda, \vc{g}]$ $\prst$-a.s. Thus, applying Proposition~\ref{PF1} the convergence \eqref{F6} can be strengthened to
\begin{align}
&&\left( \| \vr_h - \vr \|_{L^r(0,T; L^\gamma(\Td))} + \| \vm_h - \vm \|_{L^r(0,T; L^{\frac{2 \gamma}{\gamma + 1}}( \Td; R^d))}
 + \| \vu_h - \vu \|_{L^2((0,T) \times \Td; R^d)}
\right) &\to 0 \quad  \prst-\mbox{ a.s.},
\br
&&\qquad 1 \leq r < \infty.
	\label{F6bis}
\end{align}

Finally, as the {second moments} of the initial energy and data are  bounded, cf. assumption \eqref{hyp1},
we get
\begin{align}
\expe{ 	\| \vr_h - \vr \|_{L^r(0,T; L^\gamma (\Td))}^q } &\to 0 \qquad  \mbox{ for } 1 \leq q < 2 \gamma, \ 1 \leq r < \infty, \br
\expe{	\| \vm_h - \vm \|_{L^r(0,T; L^{\frac{2 \gamma}{\gamma + 1}}(\Td; R^d))}^q } &\to 0\qquad   \mbox{for }\ 1 \leq q < \frac{4 \gamma}{\gamma + 1},\  1 \leq r < \infty,
\br
\expe{ 	\| \vu_h - \vu \|_{L^2((0,T) \times \Td; R^d)}^2 } &\to 0.
	\label{F7}
\end{align}

\section{Convergence of the Monte Carlo FV method}
\label{MOCA}

{As the statistical errors are controlled by \eqref{S2a}--\eqref{S2b}}, it is enough to control the discretization errors
\[
\frac{1}{N} \sum_{n = 1}^N (\vr^n_{h} - \vr^n),\ \
\frac{1}{N} \sum_{n=1}^N (\vm^n_{h} - \vm^n ),\ \ \frac{1}{N} \sum_{n=1}^N (\vu^n_h - \vu^n ).
\]
We have for all $(\vr_{h}, \vm_{h})$, $h \searrow 0$
%
\[
\left\| \frac{1}{N} \sum_{n = 1}^N (\vrh^n - \vr^n) \right\|_{L^r(0,T; L^\gamma(\Td))}^q
\aleq \frac{1}{N}  \sum_{n = 1}^N \left\| \vrh^n - \vr^n \right\|_{L^r(0,T; L^\gamma(\Td))}^q, \qquad 1 \leq q < \gamma, \ 1 \leq r <\infty.
\]
Passing to the expectations,
\begin{align}
&\expe{  \left\| \frac{1}{N} \sum_{n = 1}^N (\vr^n_h - \vr^n) \right\|_{L^r(0,T; L^\gamma(\Td))}^q }
\aleq \frac{1}{N} \sum_{n=1}^N  \expe{  \left\| \vr^n_h - \vr^n \right\|_{L^r(0,T; L^\gamma(\Td))}^q } \br
&\quad = \expe{ \left\| \vr_h - \vr \right\|_{L^r(0,T; L^\gamma(\Td))}^q } \to 0,\
\qquad 1 \leq q < \gamma, \ 1\leq r < \infty.
\label{F8}
\end{align}
Similarly, we can show
\begin{align}
	&\expe{ \left\| \frac{1}{N} \sum_{n = 1}^N (\vmh^n - \vm^n) \right\|_{L^r(0,T; L^{\frac{2 \gamma}{\gamma + 1}}(\Td; R^d))}^q }
	\to 0,\ 1 \leq q < \frac{2 \gamma}{\gamma + 1}, \ 1\leq r <\infty,
    \label{F9} \\
	\label{F9a}
	&\expe{ \left\| \frac{1}{N} \sum_{n = 1}^N (\vuh^n - \vu^n) \right\|^2_{L^2( (0,T) \times \Td; R^d)} }
	\to 0.
\end{align}

Combining \eqref{S2a}, \eqref{S2}, \eqref{S2b}  with \eqref{F8}, \eqref{F9}, \eqref{F9a} we get the desired convergence of the finite volume approximation.
Summarizing we state the first main result concerning the convergence of the Monte Carlo FV method.


\begin{mdframed}[style=MyFrame]
	
	\begin{Theorem} \label{FVT1}
		
		Let the data be random variables and satisfy
		\[
		\Big[ \vr_0, \vm_0, \mu, \lambda, \vc{g} \Big] \in D ,
		\]
		\begin{align}
		&\vr_0 \in W^{3,2}(\Td),\ \min_{\Td} \vr_0 = \underline{\vr} > 0,\
	\vm_0 \in W^{3,2}(\Td; R^d),\ \expe{ \left(\intTd{ E[ \vr_0, \vm_0 ] }\right)^2} < \infty, \br
	&\qquad 0 < \underline{\mu} \leq \mu,\ \lambda \geq 0,\ \| \vc{g} \|_{C(\Td; R^d)} \leq \Ov{g}
	\nonumber	
		\end{align}
$\prst$-a.s., where $\underline{\mu}$, $\underline{\vr},$ $\Ov{{g}}$ are deterministic constants.

Suppose that $[\vr_0^n, \vc{m}_0^n, \, \mu^n, \lambda^n, g^n]$, $n=1,2,\dots, N$ are i.i.d. copies of  random data.
Let $\left(\vrh^n, \vmh^n \right)_{h \searrow 0}$
be a sequence of FV solutions \eqref{scheme} corresponding to these data samples.
Assume that FV solutions $\left(\vrh^n, \vmh^n \right)_{h\searrow 0},$  $n=1,2,\dots$
are bounded in probability in the sense of Definition~\eqref{DFF1}.

Then		
\begin{align}
&\expe{ \left\| \frac{1}{N} \sum_{n = 1}^N \Big(\vr^n_h - \expe{ \vr  } \Big) \right\|_{L^r(0,T; L^\gamma(\Td))}^q } \to 0 \quad \mbox{ for }\ N \to \infty,\ h \to 0,\ 1 \leq q < 2 \gamma, \  \br
&\expe{ \left\| \frac{1}{N} \sum_{n = 1}^N \Big( \vm^n_h - \expe{\vm} \Big) \right\|_{L^r(0,T; L^{\frac{2 \gamma}{\gamma + 1}}(\Td; R^d))}^q }
\to 0\ \quad \mbox{for } N\to \infty, \ h \to 0,\ 1 \leq q < \frac{4 \gamma}{\gamma + 1},	
\nonumber \\
& \expe{ \left\| \frac{1}{N} \sum_{n = 1}^N \Big( \vu^n_h - \expe{\vu} \Big) \right\|^{2}_{L^2((0,T) \times \Td; R^d) }}
\to 0\ \quad \mbox{for } N\to \infty, h\to 0,
\label{fin}
	\end{align}
where $1\leq r < \infty$ and $(\vr, \vu)$ is a classical solution to the Navier--Stokes system corresponding to the data
$[\vr_0, \vm_0, \mu, \lambda, \vc{g}]$.

		\end{Theorem}

	\end{mdframed}

\section{Error estimates of the Monte Carlo FV method}
\label{sec_error}

In this section we estimate the errors of the Monte Carlo FV method. The error estimates are obtained under the assumption of more regular data \eqref{reg_init}, i.e.
$(\vr_0, \vm_0) \in W^{6,2}(\Td) \times W^{6,2}(\Td;R^d)\ \prst$-a.s.
Furthermore, under assumption \eqref{bip} that the FV solutions $(\vrh, \vmh)_{h \searrow 0}$ are bounded in probability, it follows from the arguments of Section \ref{rnd_data} that there exists
a random classical solution  $(\vr, \vu)$ of the Navier--Stokes system \eqref{i1}--\eqref{i5}, such that
\begin{eqnarray}
&&\vr \in  C([0,T]; W^{6,2}(\Td)) \cap C^1([0,T] \times \Td) \nonumber \\
&&\vu \in  C([0,T]; W^{6,2}(\Td;R^d)) \cap C^1([0,T] \times \Td; R^d) \qquad  \prst-\mbox{a.s.} \label{reg_class}
\end{eqnarray}


Now, we are in a position to apply Proposition \ref{PF2}. First observe that the assumed regularity of the
data implies that for any $\ep > 0$, there exists $M_1 = M_1(\ep)$ such that
\begin{equation} \label{kon1}
\mathcal{P} \Big\{ \| \vc{g} \|_{C^\ell(\Td; R^d)} \leq M_1(\ep),\
	\| (\vr_0 , \vu_0) \|_{W^{6,2}(\Td; R^{d+1})} \leq M_1(\ep),\
	\| (\vr , \vu) \|_{C([0,T] \times \Td; R^{d+1})} \leq M_1(\ep) \Big\} \geq 1 - \ep.
	\end{equation}

Next, as the numerical solutions are bounded in probability, for any $\ep$ there is $M_2(\ep) > 0$ such that
\begin{equation} \label{kon2}
\mathcal{P} \Big\{ \| (\vrh, \vuh) \|_{L^\infty((0,T) \times \Td; R^{d +1} )} \leq M_2(\ep)
\Big\} > 1-\ep \ \ \mbox{ for all }\ h \in (0,1).
\end{equation}	

Combining \eqref{kon1}, \eqref{kon2} with the error estimates stated in Proposition \ref{PF2},  notably
formula \eqref{error rates3}, we obtain
\emph{error estimates in probability}:

For any $\ep > 0$, there is $K = K(\ep)$ such that
\begin{eqnarray}
&&\hspace{-0.6cm}\prst \Bigg[  \left\| \vr_h^n(t,\cdot) - \vr^n(t, \cdot) \right\|_{L^2(\Td)}
+ \left\| \vm_h^n(t, \cdot) - \vm^n(t, \cdot) \right\|_{L^2 (\Td; R^d)}
\leq K(\ep)(\sqrt{h} + \sqrt{\Delta t}) \Bigg]
 \geq 1 - \ep,
\br&&\qquad \phantom{mmmmmmmmmmmmmmmmmmmmmmmmmm} \mbox{ for all } \ t \in [0,T],
\nonumber \\ \label{error5}
&&\hspace{-0.6cm}\prst \Bigg[  \left\| \vu^n_h - \vu^n \right\|_{L^2 (0,T; L^q(\Td; R^d))}
\leq K(\ep)(\sqrt{h} + \sqrt{\Delta t}) \Bigg]
\qquad  \geq 1 - \ep
\nonumber \\
&& \hspace{5cm}1 \leq q < \infty \ \mbox{ if }\ d= 2\ \ \mbox{ and } \ q = 6 \  \ \mbox{ if } \ d= 3.
\nonumber
	\end{eqnarray}
As all random variables are equally distributed, the above estimates are independent of $n$.

To summarize we have proved the  second main result of this paper concerning the error estimates of the Monte Carlo FV method for the Navier--Stokes system \eqref{i1}--\eqref{i5}.

\begin{mdframed}[style=MyFrame]
	
\begin{Theorem} [{\bf Error estimates}] \label{FVT2}
	Let the data are random variables and satisfy
		\[
		\Big[ \vr_0, \vm_0, \mu, \lambda, \vc{g} \Big] \in D ,
		\]
$(\vr_0, \vu_0) \in W^{6,2}(\Td) \times W^{6,2}(\Td;R^d),$ \ $\min_{\Td}  \vr_0 \equiv \underline{\vr} > 0,$ $\prst$-a.s. and
\[
\expe{ \left(\intTd{ E[ \vr_0, \vm_0 ] }\right)^2} < \infty.
\]
Further, $0 < \underline{\mu} \leq \mu,\ \lambda \geq 0$ and
$\vc{g} \in C^\ell(\Td; R^d),\, \ell \geq 6, \  \| \vc{g} \|_{C(\Td; R^d)} \leq \Ov{g} \ \ \prst-\mbox{a.s.}$
Let $(\vr , \vm = \vr \vu)$ be a classical solution to the Navier--Stokes system \eqref{i1}--\eqref{i5}
corresponding to the data
$[\vr_0, \vm_0, \mu, \lambda, \vc{g}]$ and
belonging to the regularity class \eqref{reg_class}.  \\

Suppose that $[\vr_0^n, \vc{m}_0^n, \, \mu^n, \lambda^n, g^n]$, $n=1,2,\dots$ are i.i.d. copies of  random data.
Let $\left(\vrh^n, \vmh^n \right)_{h \searrow 0}$
be  FV solutions \eqref{scheme} corresponding to these data samples.
Assume that FV solutions $\left( \vrh^n, \vmh^n \right)_{h\searrow 0},$  $n=1,2,\dots$
are bounded in probability in the sense of Definition~\eqref{DFF1}. 

Then the MC estimators $\sum_{n = 1}^N \vr^n_h,$ $\sum_{n = 1}^N \vm^n_h$, $\sum_{n = 1}^N \vu^n_h$ satisfy for every $N=1,2,\dots$  the following error estimates.
For the expectation of the statistical errors it holds
\begin{align}
& \expe{\left\| \frac{1}{N} \sum_{n = 1}^N \Big(\vr^n (t, \cdot)- \expe{ \vr (t, \cdot) } \Big) \right\|^r_{L^\gamma(\Td)} } \aleq N^{1-r}
\quad \mbox{ for }\   \ r=\min{(2, \gamma)}\br
& \expe{ \left\| \frac{1}{N} \sum_{n = 1}^N \Big( \vm^n(t, \cdot) - \expe{\vm(t, \cdot)} \Big) \right\|^{\frac{2 \gamma}{\gamma + 1}}_{ L^{\frac{2 \gamma}{\gamma + 1}}(\Td; R^d)}}\aleq N^{\frac{1-\gamma}{1+\gamma}  },\ \ \mbox{ for all } t \in [0,T];
\br
& \expe{ \left\| \frac{1}{N} \sum_{n = 1}^N \Big( \vu^n - \expe{\vu} \Big) \right\|^2_{L^2(0,T; W^{1,2}(\Td; R^d))}}\aleq N^{-1}, \qquad N =1,2,\dots.
\label{fin1}
\end{align}
The approximation errors are estimated in probability, meaning
for any $\ep > 0$, there exists $K(\ep)$ such that
\begin{eqnarray} \label{EE3a}
&&\hspace{-1cm}\prst \Bigg[  \left(  \left\| \frac{1}{N} \sum_{n = 1}^N (\vr^n_h(t, \cdot) - \vr^n(t, \cdot) ) \right\|_{L^2(\Td))}
+ \left\| \frac{1}{N} \sum_{n = 1}^N (\vm^n_h(t, \cdot) - \vm^n(t, \cdot)) \right\|_{L^2 (\Td; R^d))} \right)
\br
&& \qquad \qquad \leq K(\ep)( \sqrt{h} + \sqrt{\Delta t} ) \Bigg] \geq 1 - \ep  \qquad \mbox{ for all } t\in [0,T],
\br
&&\prst \Bigg[  \left\| \frac{1}{N} \sum_{n = 1}^N ( \vu^n_h - \vu^n )\right\|_{L^2 (0,T; L^q(\Td; R^d))}
\leq K(\ep)(\sqrt{h} + \sqrt{\Delta t}) \Bigg]
\  \geq 1 - \ep, \
\br
&& \hspace{5cm}1 \leq q < \infty \ \mbox{ if }\ d= 2\ \ \mbox{ and } \ q = 6 \  \ \mbox{ if } \ d= 3, \br
&&\phantom{mmmmmmmmmmmmmmm} h \in (0,1),\ \ \Delta t \in (0,1), \ \ N = 1,2,\dots.
\end{eqnarray}

\end{Theorem}

\end{mdframed}

%
\medskip	

Further, we can combine the results on statistical errors of the deviation and variance, cf.~\eqref{dev1}, \eqref{dev2}, \eqref{var} with
Proposition~\ref{PF2} and the error estimates \eqref{error rates1}  to derive the following result on the convergence of  deviation and variance.

\begin{mdframed}

\begin{Corollary} \label{coro}
Let assumptions of Theorem~\ref{FVT2} hold. Then the deviation of the density and momentum as well as the variance of the velocity computed by the FV solutions $\left( \vr_h^n, \vm_h^n \right)_{h \searrow 0},$ \ $n= 1,2, \dots$ converge $\prst$-a.s.

\begin{eqnarray}
&&N^r \left\| \frac 1 N \sum_{n=1}^N \Big| \vr_h^n  - \frac{1}{N}  \sum_{m=1}^N \vr_h^m  \Big| \ -  \mbox{Dev}(\vr)  \right\|_{L^s(0,T; L^\gamma(\Td))} \to 0,  \mbox{ as }  N \to \infty, \ h \to 0,
\label{dev_rho} \\ \nonumber
&& \phantom{mmmmmmmmmmmmmmmmmmmmmmmmmmm} \quad \mbox{ where } \ 1 \leq s < \infty ,
\\ \label{dev_m}
&&\hspace{-0.3cm}N^{\frac{\gamma -1}{2 \gamma}}\left\| \frac 1 N \sum_{n=1}^N \Big| \vm_h^n  - \frac{1}{N}  \sum_{m=1}^N \vm_h^m  \Big| \ -  \mbox{Dev}(\vm )  \right\|_{L^s(0,T; L^{\frac{2\gamma}{\gamma + 1}}(\Td; R^d))}\to 0, \quad \mbox{ as } N \to \infty, \ h \to 0,
 \nonumber \\
&& \phantom{mmmmmmmmmmmmmmmmmmmmmmmmmmm}\quad  \mbox{ where } 1 \leq s < \infty,\
\\ \label{var_u}
&& \hspace{-0.3cm}\left\| \frac{1}{N-1}  \sum_{n=1}^N  \left( \vu_h^n   - \frac{1}{N}  \sum_{m=1}^N \vu_h^m \right)^2  -  \mbox{Var}(\vu) \right\|_{L^1(0,T; L^3(\Td;R^d))} \to 0 \quad \mbox{ as }  N \to \infty, \ h \to 0. 
\end{eqnarray}

\end{Corollary}

\end{mdframed}

\begin{proof}
The convergence results \eqref{dev1} and \eqref{dev2}  together with the convergence of FV solutions $(\vr_h, \vm_h)_{h \searrow 0},$ cf.~Proposition~\ref{PF1}
 \begin{align*}
 \hspace{-0.5cm}\left\| \vr_h^n - \vr^n \right\|_{L^s(0,T; L^\gamma (\Td))} +
\left\| \vm_h^n - \vm^n \right\|_{L^s(0,T; L^{ \frac{2\gamma}{\gamma + 1} }(\Td; R^d))}
 \to 0 \ \mbox{ as }\ h \to 0,\ n=1,2,\dots
\end{align*}
 yield, by a direct calculation, the convergence of the estimator of  deviation of the density and momentum, cf.~\eqref{dev_rho} and \eqref{dev_m}.

In order to show the convergence of the variance of velocity we can apply \eqref{var}  and   \eqref{error_q}.
Indeed, since the global classical solution $(\vr, \vu)$ exists, we use \eqref{error rates1} $\prst$-a.s. to derive
$$
\mathcal{E}(\vr_h^n, \vm_h^n| \vr^n, \vu^n) \to 0 \ \mbox{ as } h \to 0 \qquad \prst{\mbox-a.s.}, \qquad n =1,2, \dots
$$
and obtain \eqref{var_u}.
For $d=2$ the convergence of velocity  variance holds in $L^1(0,T; L^q(\Td;R^d)),$ $1\leq q < \infty.$

\end{proof}

\section{Numerical experiment} \label{num}
In this section we present numerical results obtained by the  FV method \eqref{scheme}, as well as the MAC finite difference method.
The parameter $\varepsilon$ in the diffusive upwind flux of the FV method is set to $0.6$.
In order to experimentally investigate the convergence of MC method, we sample $N$ i.i.d  initial data $(\vr_0^n, \vm_0^n)$, $n=1, \dots, N,$
from the random field $(\vr_0, \vm_0)$ with $N = 5\times 2^{n}, \, n = 0,\dots, 4$ and approximate them by piecewise constant projection on a computational domain $[-1,1] \times [-1,1]$ with a mesh parameter $h=2/512$. 
The finial time is set to $T = 0.1$.
In what follows we concentrate on the experimental analysis of the following statistical errors.
\begin{itemize}
\item Error of mean value:
\begin{align}\label{E1}
E_1(U) =  \frac1{M} \sum_{m=1}^M \left(  \left\| \frac{1}{N} \sum_{n = 1}^N \Big(U^{n,m} (T, \cdot)- \expe{ U (T, \cdot) } \Big) \right\|_{L^p(\Td)}  \right)
\end{align}
and
\begin{align}\label{E2}
E_2(U) = \left[ \frac1{M} \sum_{m=1}^M \left( \left\| \frac{1}{N} \sum_{n = 1}^N \Big( U^{n,m}(T, \cdot) - \expe{U(T, \cdot)} \Big) \right\|^{p}_{L^{p}(\Td)}\right) \right]^{1/p}
\end{align}
with
\begin{equation*}
\expe{U}(t,x) =  \frac1{S} \sum_{s=1}^{S}  U^{s}(t,x),
\end{equation*}
where $U$ stays for $\vr_h,\, \vm_h$ or $\vu_h$ and $S$ denotes the size of ensamble. Results are averaged over $M$ realisations.
In the numerical simulations presented below we set $S, M$ to $1000, 40$, respectively.

\item Error of deviation or variance:
\begin{align}\label{E3}
E_3(U ) =
\begin{cases}
\frac1{M} \sum\limits_{m=1}^M \left(   \left\| \frac 1 N \sum\limits_{n=1}^N \Big| U^{n,m}  - \frac{1}{N}  \sum\limits_{l=1}^N U^{l,m}  \Big| \ -  \mbox{Dev}(U)  \right\|_{L^p(\Td)}  \right), & \mbox{if} \ U = \vr_h, \vm_h, \\
\frac1{M} \sum\limits_{m=1}^M \left(  \left\| \frac{1}{N-1}  \sum\limits_{n=1}^N  \left( U^{n,m}   - \frac{1}{N}  \sum\limits_{l=1}^N U^{l,m} \right)^2  -  \mbox{Var}(U) \right\|_{L^p( \Td)} \right),  &  \mbox{if} \ U = \vu_h
\end{cases}
\end{align}
and
\begin{align}\label{E4}
E_4(U) =
\begin{cases}
\left[ \frac1{M} \sum\limits_{m=1}^M \left( \left\| \frac 1 N \sum\limits_{n=1}^N \Big| U^{n,m}  - \frac{1}{N}  \sum\limits_{l=1}^N U^{l,m}  \Big| \ -  \mbox{Dev}(U)  \right\|_{L^p(\Td)}^p  \right) \right]^{1/p}, & \mbox{if} \ U = \vr_h, \vm_h, \\
\left[ \frac1{M} \sum\limits_{m=1}^M \left(  \left\| \frac{1}{N-1}  \sum\limits_{n=1}^N  \left( U^{n,m}   - \frac{1}{N}  \sum\limits_{l=1}^N U^{l,m} \right)^2  -  \mbox{Var}(U) \right\|_{L^p( \Td)}^p \right)\right]^{1/p},  &  \mbox{if} \ U = \vu_h
\end{cases}
\end{align}
with
\begin{align*}
&\mbox{Dev}(U) =  \frac1{S} \sum_{s=1}^{S}  \Bigg|U^{s} - \frac1{S} \sum_{s=1}^{S}  U^{s}(t,x) \Bigg|, \quad
\mbox{Var}(U) =  \frac1{S-1} \sum_{s=1}^{S}   \Bigg|U^{s} - \frac1{S} \sum_{s=1}^{S}  U^{s}(t,x) \Bigg|^2.
\end{align*}
\end{itemize}
Further, $p$ is $\gamma, 2\gamma/(\gamma+1), 2$ for $U$ being the density $\vr_h$, momentum $\vm_h$ and velocity $\vu_h$, respectively.

The parameters arising in the Navier--Stokes system are taken as
\begin{equation*}
\mu = 0.1, \ \eta = 0, \ \vc{g} = 0,\ \mbox{and} \ a = 0, \ \gamma = 1.4 \ \mbox{in} \  \eqref{w6}.
\end{equation*}

\subsection{Experiment 1: Random perturbation of a steady state.} \label{example1}
In this experiment we consider a random perturbation of a steady state $(\vr, \vu)(t,x) = (1, 0)$ with the initial data
\begin{subequations}
\begin{align}
& \vr_0(x) = 1 + Y_1(\omega) \cos(2\pi (x_1+x_2)), \\
& \vu_0(x) = (Y_2(\omega) , Y_3(\omega) )^t,
\end{align}
where $Y_j(\omega)$ are uniformly distributed, i.e.
\begin{equation}
Y_j(\omega) \ \sim \ \mathcal{U}\left( -0.1, 0.1 \right),  \quad j =1,2,3.
\end{equation}
\end{subequations}
Figure~\ref{ex1} displays the mean and the deviation/variance of numerical solutions $\vr, \vu$, as well as them along $x = y$, obtained by the FV method.
We omit the numerical results obtained by the MAC method since they look similar.
Figure~\ref{ex1-dev} shows the errors $E_i$, $i=1,\dots,4$, cf. \eqref{E1}--\eqref{E4}, obtained by both methods.

The numerical results show that errors of the mean of $\vr, \vm, \vu$, the deviation of $\vr, \vm$, as well as the variance of $\vu$, converge with a statistical rate $-1/2$, which is the optimal rate of the MC method.
We recall that our convergence analysis in Section~\ref{StSo} indicates lower convergence rates that may arise in a general (less regular) situation.

\begin{figure}[htbp]
	\setlength{\abovecaptionskip}{0.cm}
	\setlength{\belowcaptionskip}{-0.cm}
	\centering
	\begin{subfigure}{0.32\textwidth}
		\includegraphics[width=\textwidth]{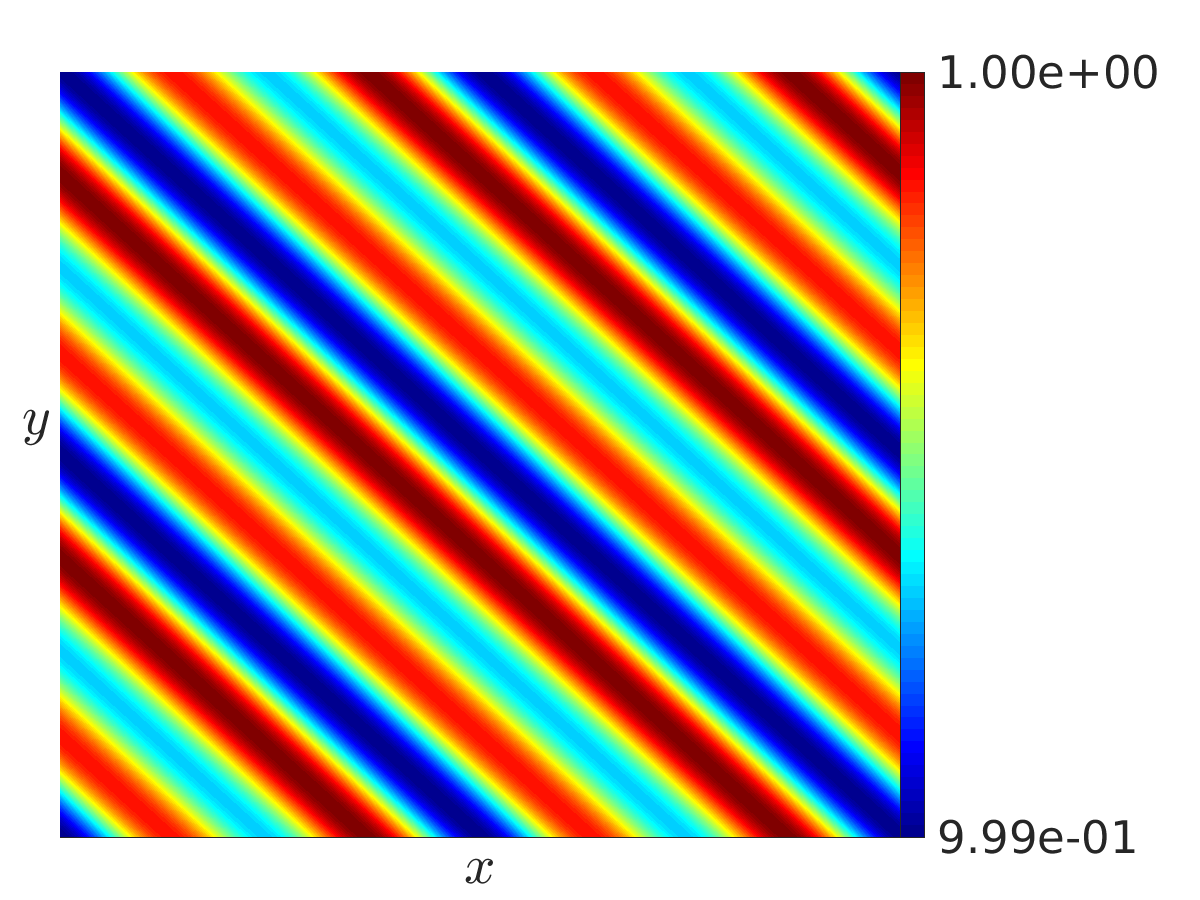}
		\caption{ \bf $\vr$ - Mean}
	\end{subfigure}	
	\begin{subfigure}{0.32\textwidth}
		\includegraphics[width=\textwidth]{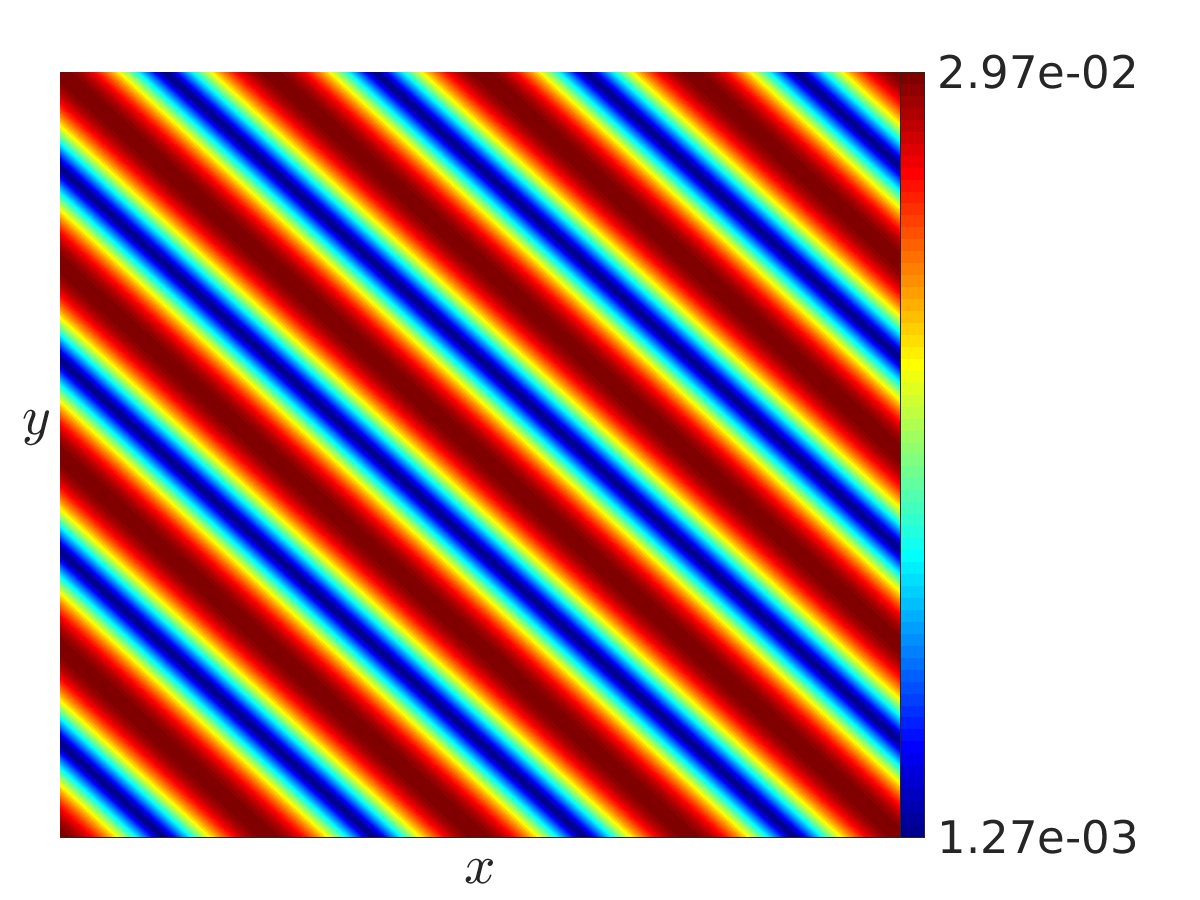}
		\caption{ \bf $\vr$ - Deviation }
	\end{subfigure}	
	\begin{subfigure}{0.32\textwidth}
		\includegraphics[width=\textwidth]{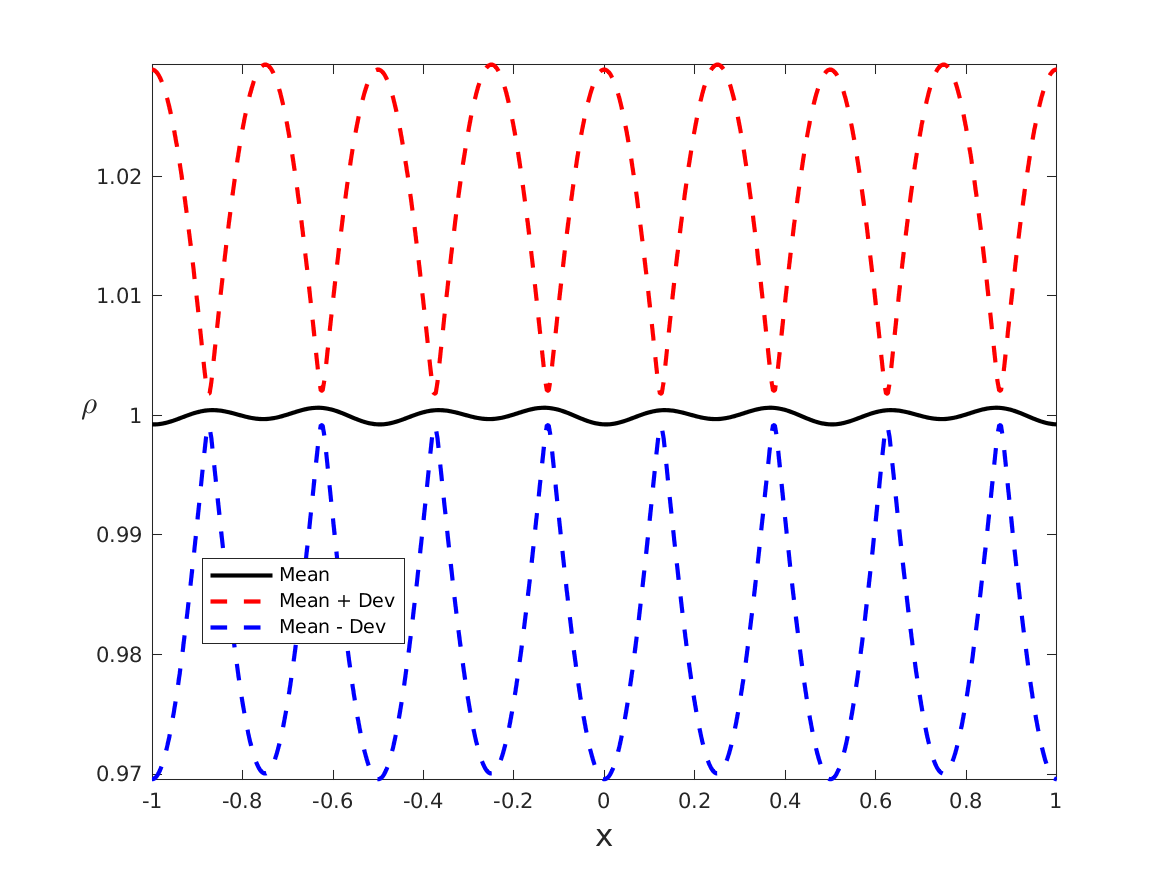}
		\caption{ \bf $\vr$ }
	\end{subfigure}	\\		
	\begin{subfigure}{0.32\textwidth}
		\includegraphics[width=\textwidth]{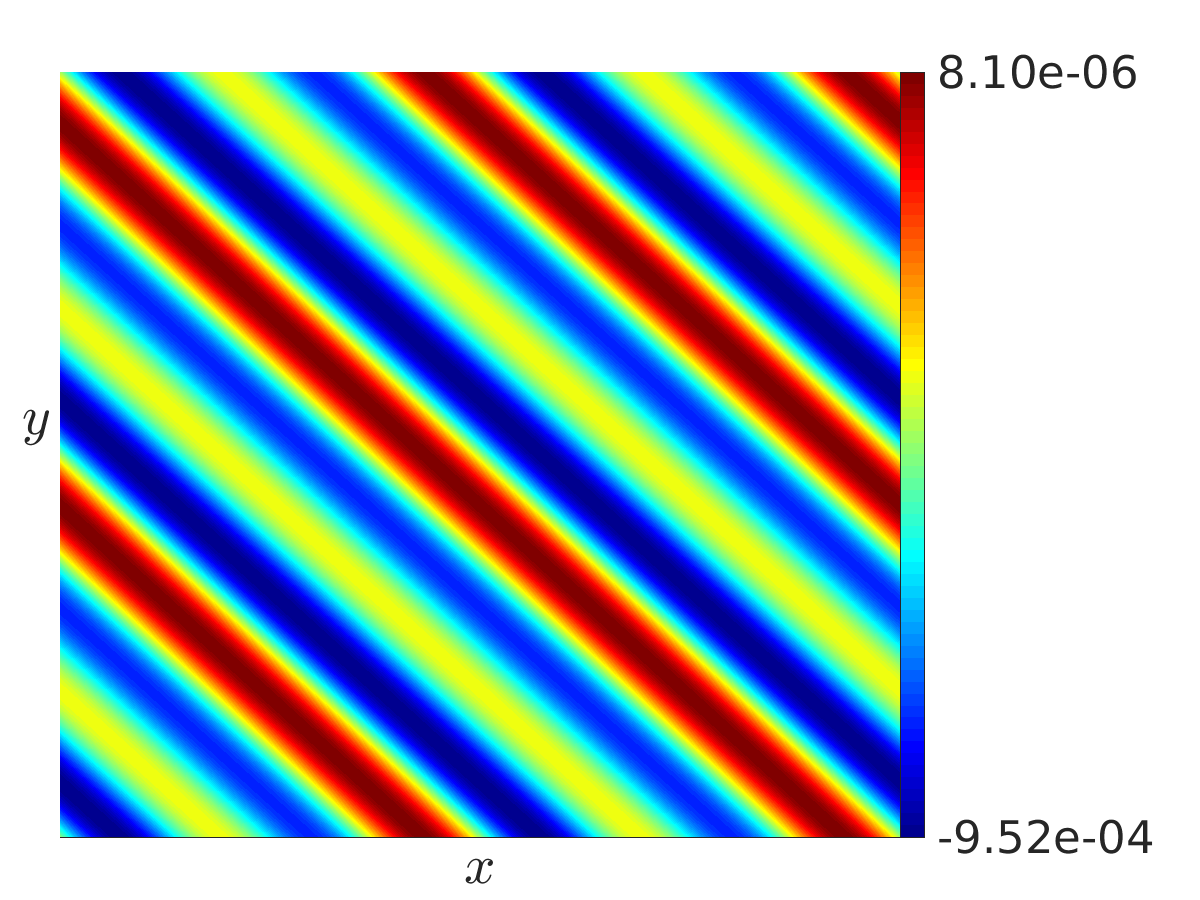}
		\caption{ \bf $u_1$ - Mean}
	\end{subfigure}	
	\begin{subfigure}{0.32\textwidth}
		\includegraphics[width=\textwidth]{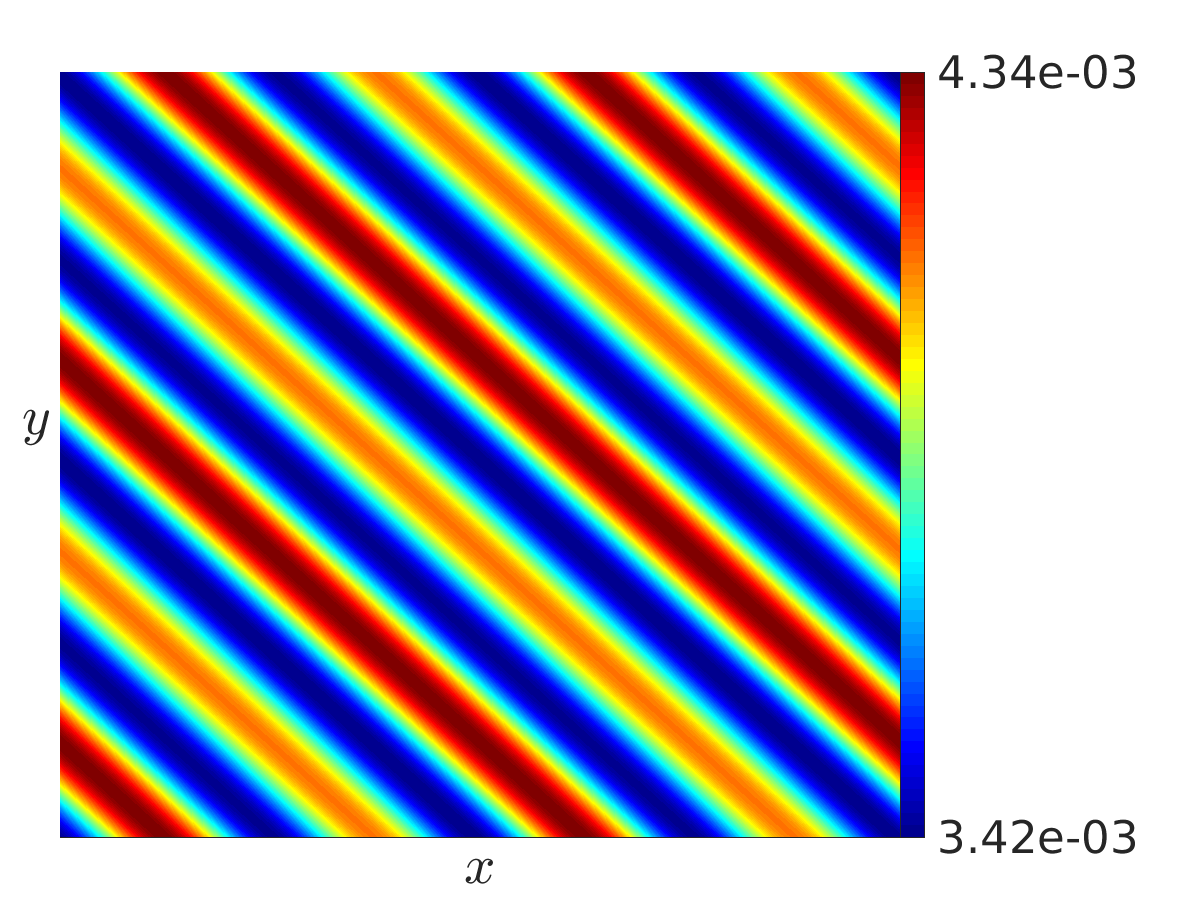}
		\caption{ \bf $u_1$ - Variance}
	\end{subfigure}	
	\begin{subfigure}{0.32\textwidth}
		\includegraphics[width=\textwidth]{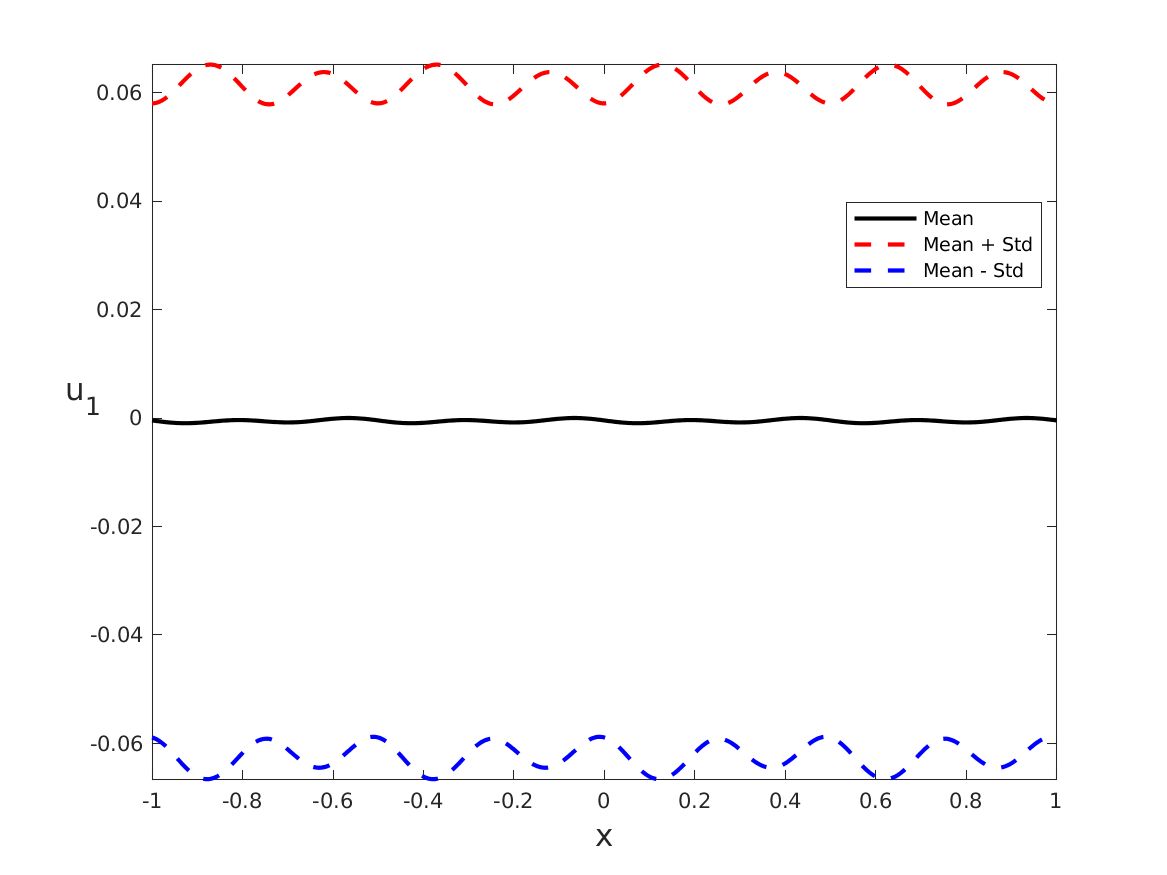}
		\caption{ \bf $u_1$}
	\end{subfigure}	\\
	\begin{subfigure}{0.32\textwidth}
		\includegraphics[width=\textwidth]{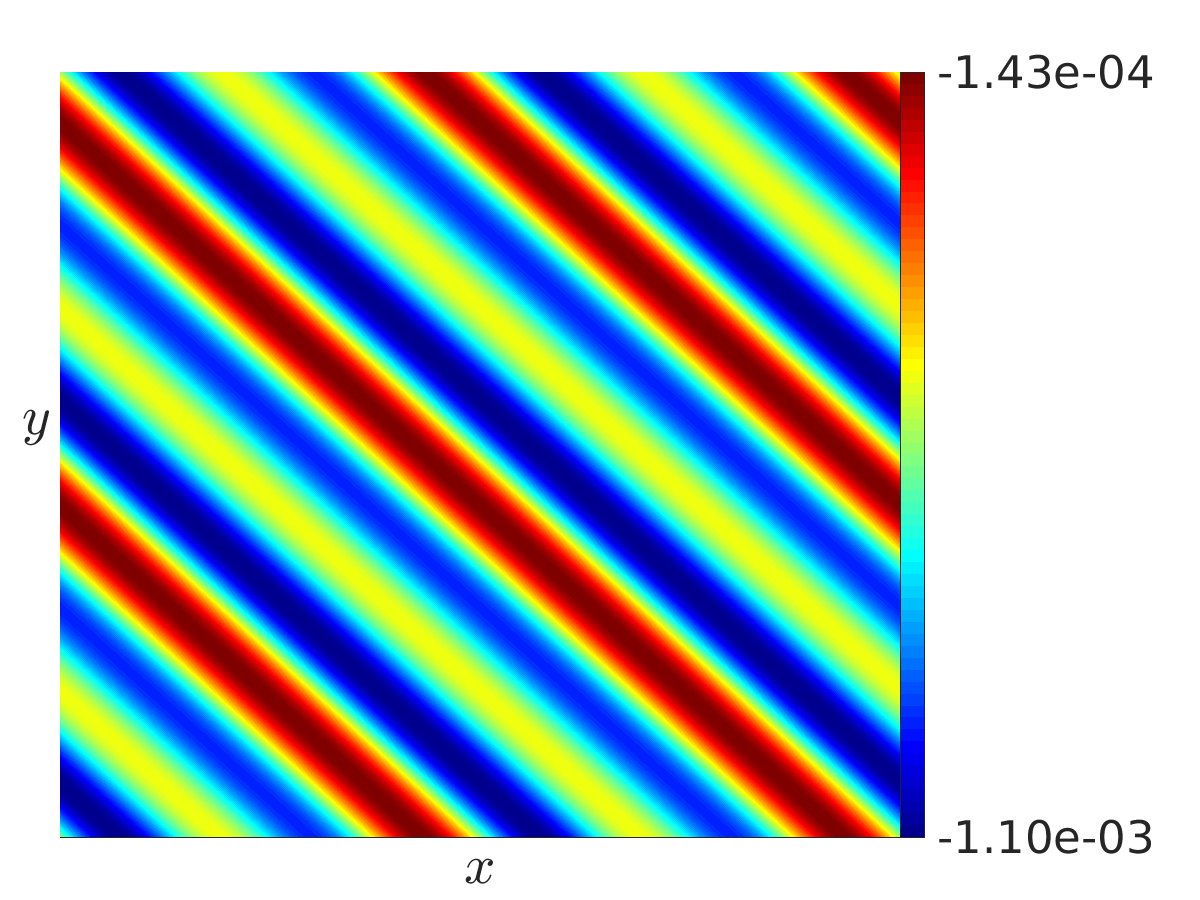}
		\caption{ \bf $u_2$ - Mean}
	\end{subfigure}	
	\begin{subfigure}{0.32\textwidth}
		\includegraphics[width=\textwidth]{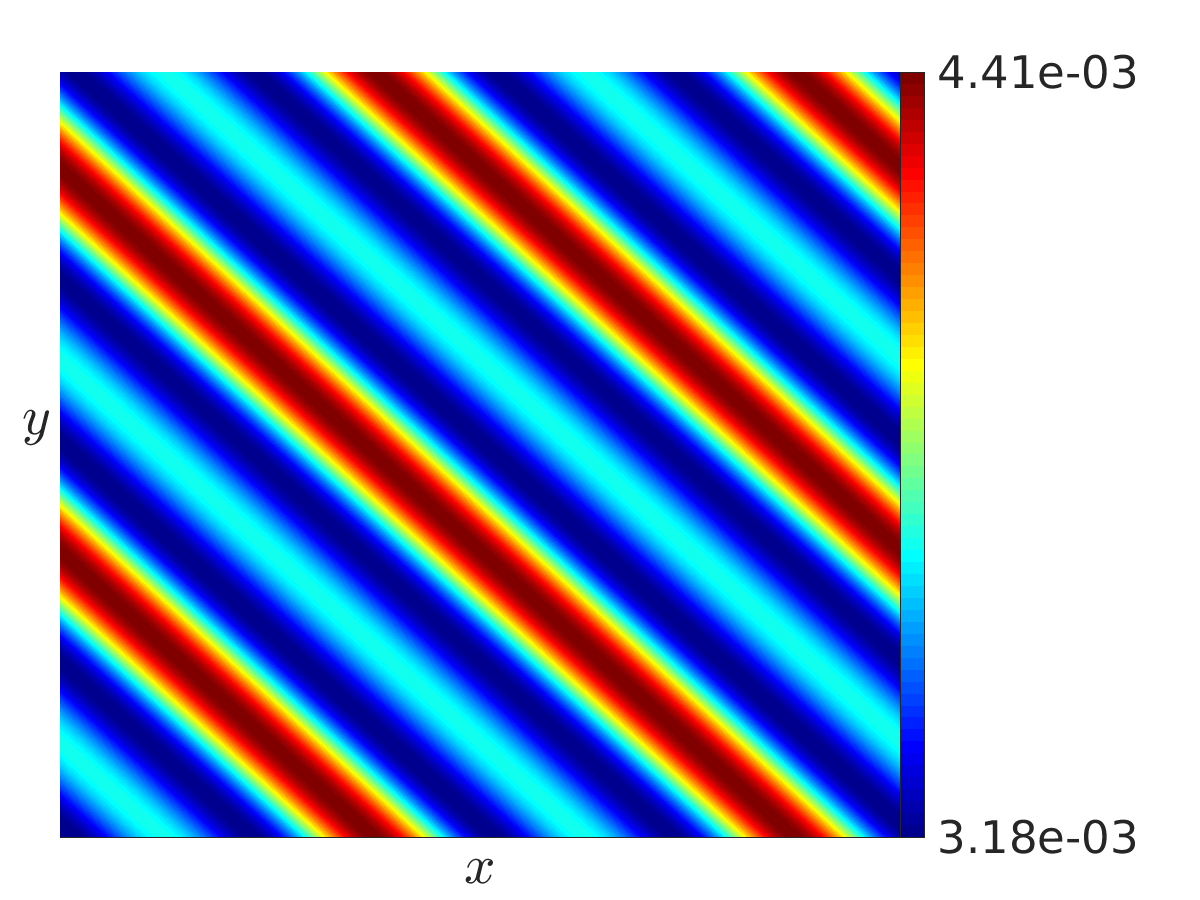}
		\caption{ \bf $u_2$ - Variance}
	\end{subfigure}	
	\begin{subfigure}{0.32\textwidth}
		\includegraphics[width=\textwidth]{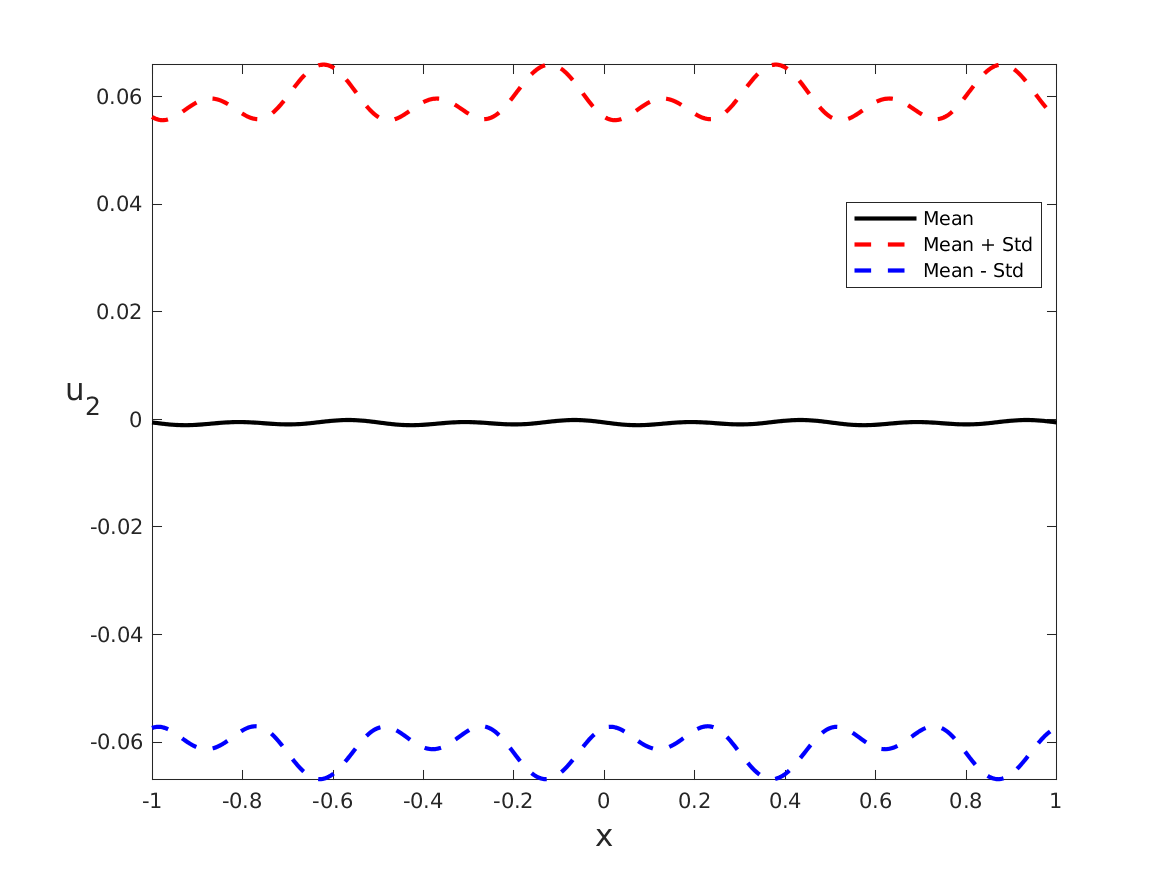}
		\caption{ \bf $u_2$}
	\end{subfigure}	
	\caption{  \small{Example \ref{example1}: Numerical solutions obtained by the FV method. Left: Mean-value of $\vr, \vu$; Middle: Deviation/variance of $\vr, \vu$; Right:  $\vr, \vu$ at the line $x = y$.}}\label{ex1}
\end{figure}

\begin{figure}[htbp]
	\setlength{\abovecaptionskip}{0.cm}
	\setlength{\belowcaptionskip}{-0.cm}
	\centering
		\includegraphics[width=\textwidth]{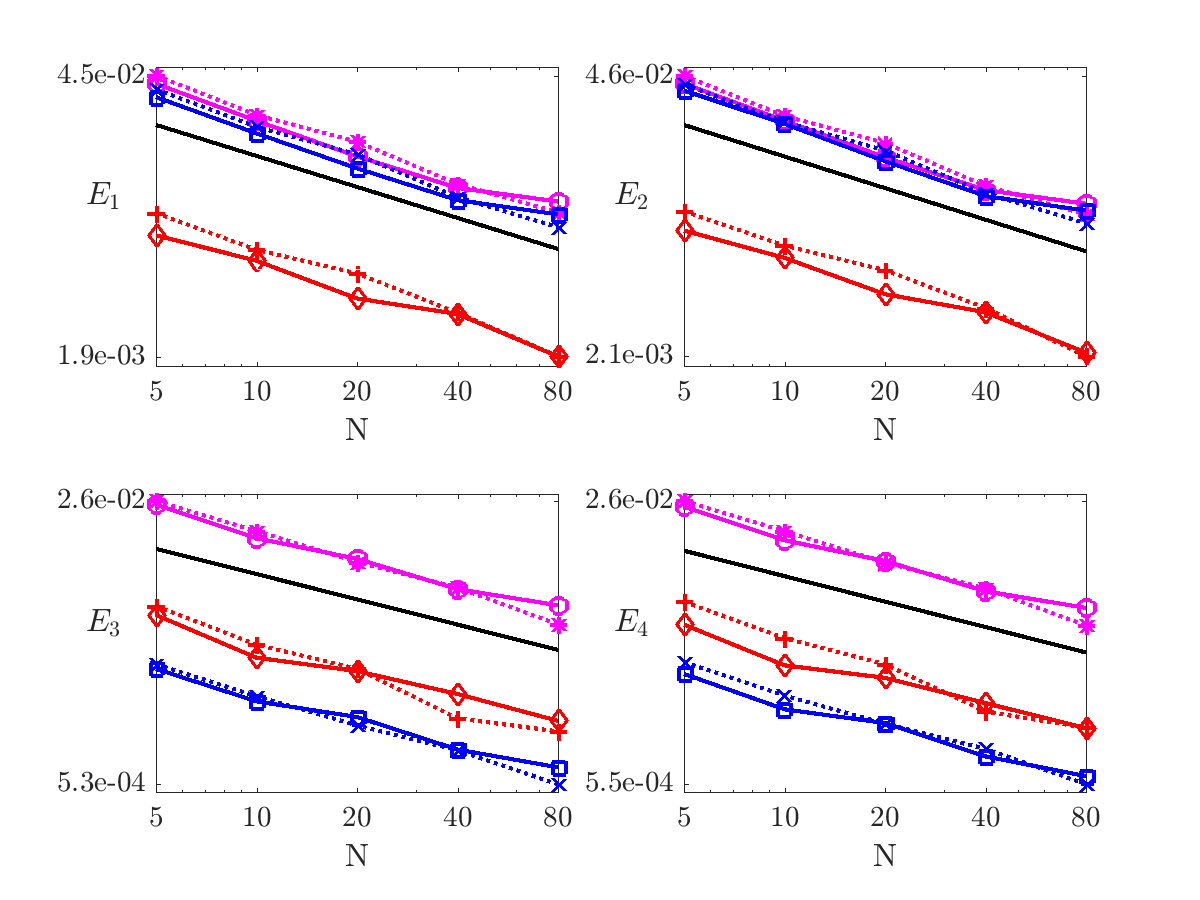}\\
		\includegraphics[width=0.6\textwidth]{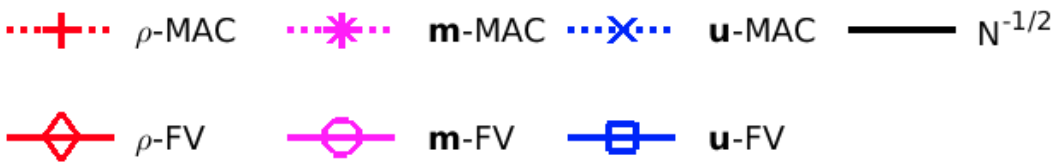}		
	\caption{  \small{Example \ref{example1}:  Statistical errors of the mean ($E_1, E_2$) and the deviation/variance ($E_3, E_4$).
	The black solid lines without any marker denote the reference slope of $N^{-1/2}$.  }}\label{ex1-dev}
\end{figure}

%

\subsection{Experiment 2: Random perturbation of a vortex.} \label{example2}
In the second experiment we simulate a more complicated physical structure -- vortex with a random noise.
Initial data are given by
\begin{subequations}
\begin{align}
& \vr_0(x) = 1 + Y_1(\omega) \cos(2\pi (x_1+x_2)), \\
& \vu_0(x) = (Y_2(\omega) , Y_3(\omega) )^t +  \begin{cases}
\left(\frac{ [1-\cos(4\pi |x|)] x_2}{|x|},\ -\frac{ [1-\cos(4\pi |x|)]x_1}{|x|} \right)^t, & \mbox{if}\  |x| < 0.5,\\
(0,0)^t, & \mbox{otherwise}
\end{cases}
\end{align}
where
\begin{equation}
Y_j(\omega) \ \sim \ \mathcal{U}\left( -0.1, 0.1\right),  \quad j =1,2,3.
\end{equation}
\end{subequations}
The mean and  deviation/variance of numerical solutions $\vr, \vu$ obtained by the FV method are showed in Figure~\ref{ex2}.
Analogously as above we investigate the errors $E_1, \dots, E_4$ obtained by both methods and present them in Figure~\ref{ex2-dev}.

The numerical results indicate that the mean and deviation/variance of $\vr, \vu, \vm$ converge with a rate $-1/2$.

\begin{figure}[htbp]
	\setlength{\abovecaptionskip}{0.cm}
	\setlength{\belowcaptionskip}{-0.cm}
	\centering
	\begin{subfigure}{0.32\textwidth}
		\includegraphics[width=\textwidth]{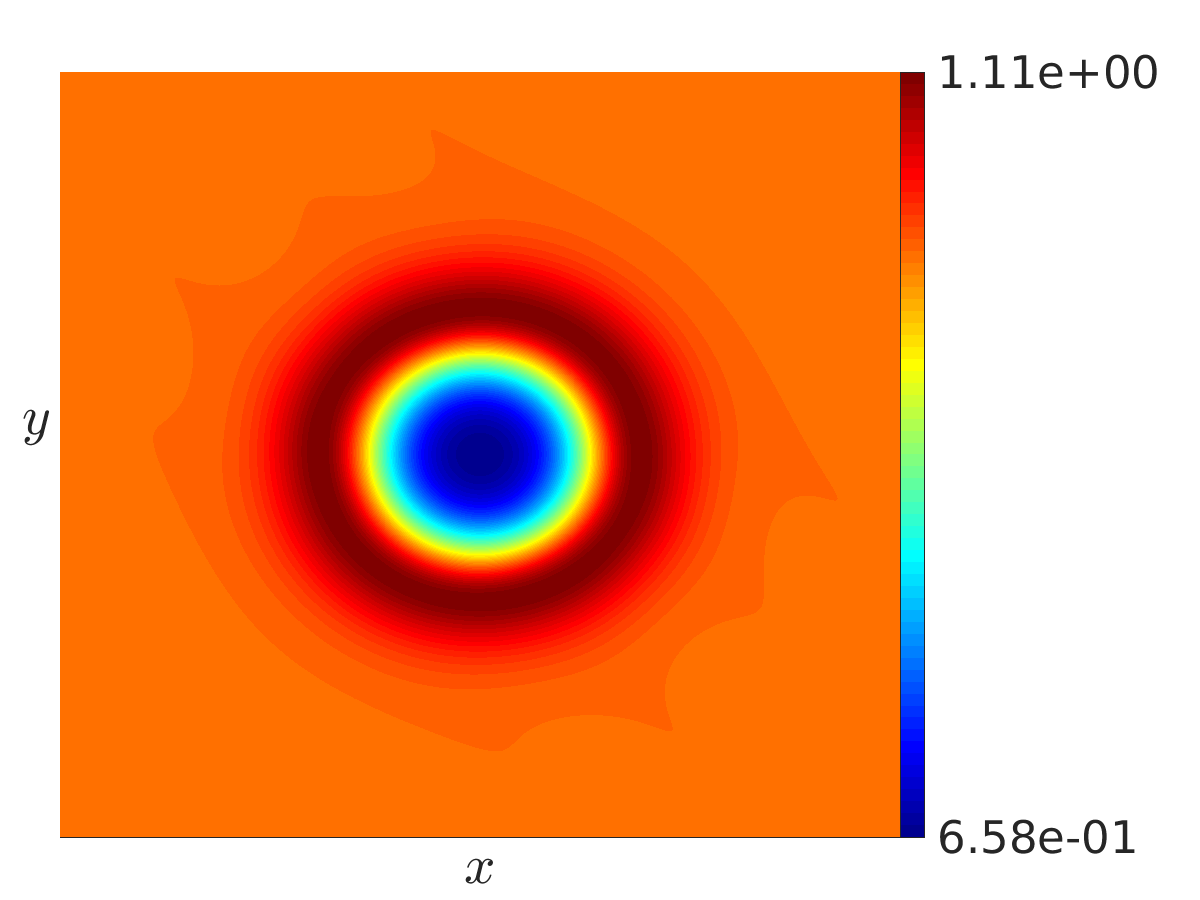}
		\caption{ \bf $\vr$ - Mean}
	\end{subfigure}	
	\begin{subfigure}{0.32\textwidth}
		\includegraphics[width=\textwidth]{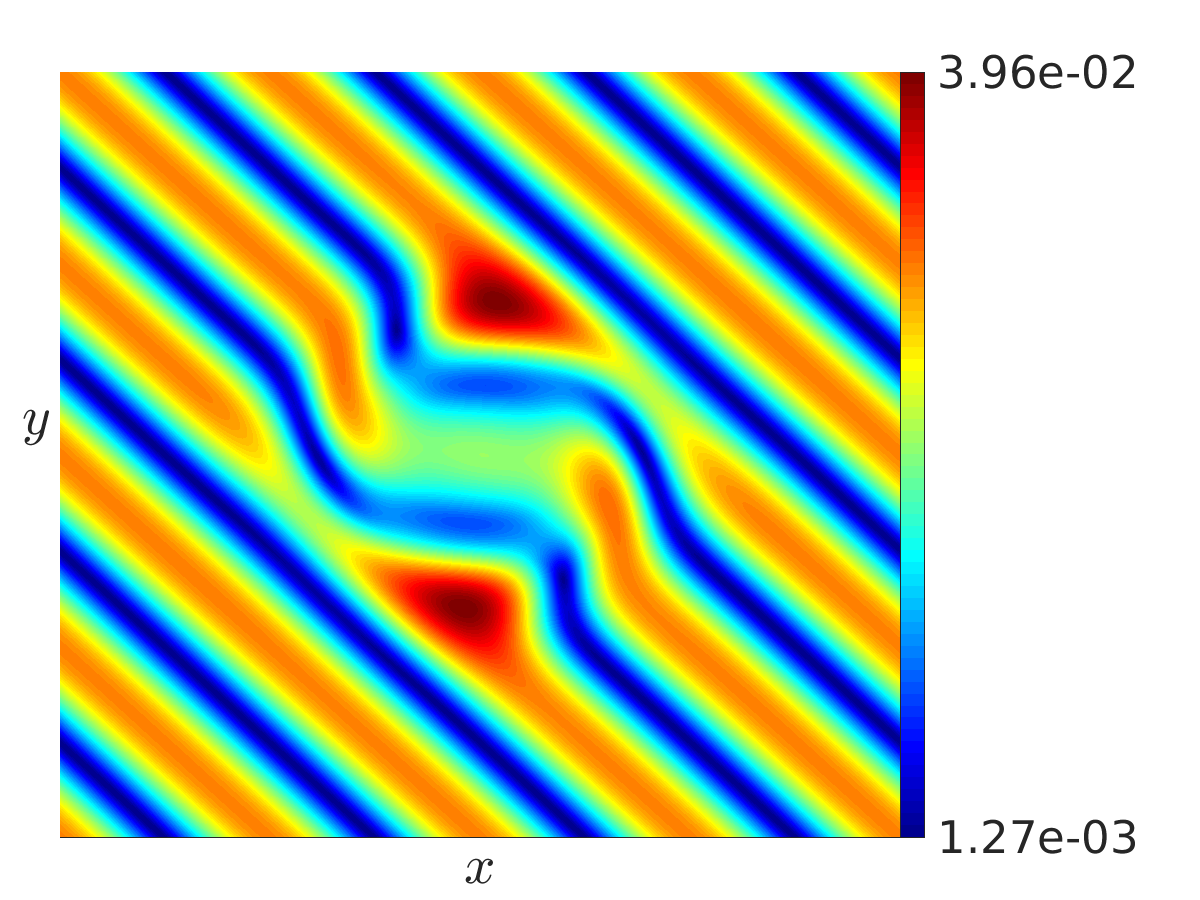}
		\caption{ \bf $\vr$ - Deviation }
	\end{subfigure}	
	\begin{subfigure}{0.32\textwidth}
		\includegraphics[width=\textwidth]{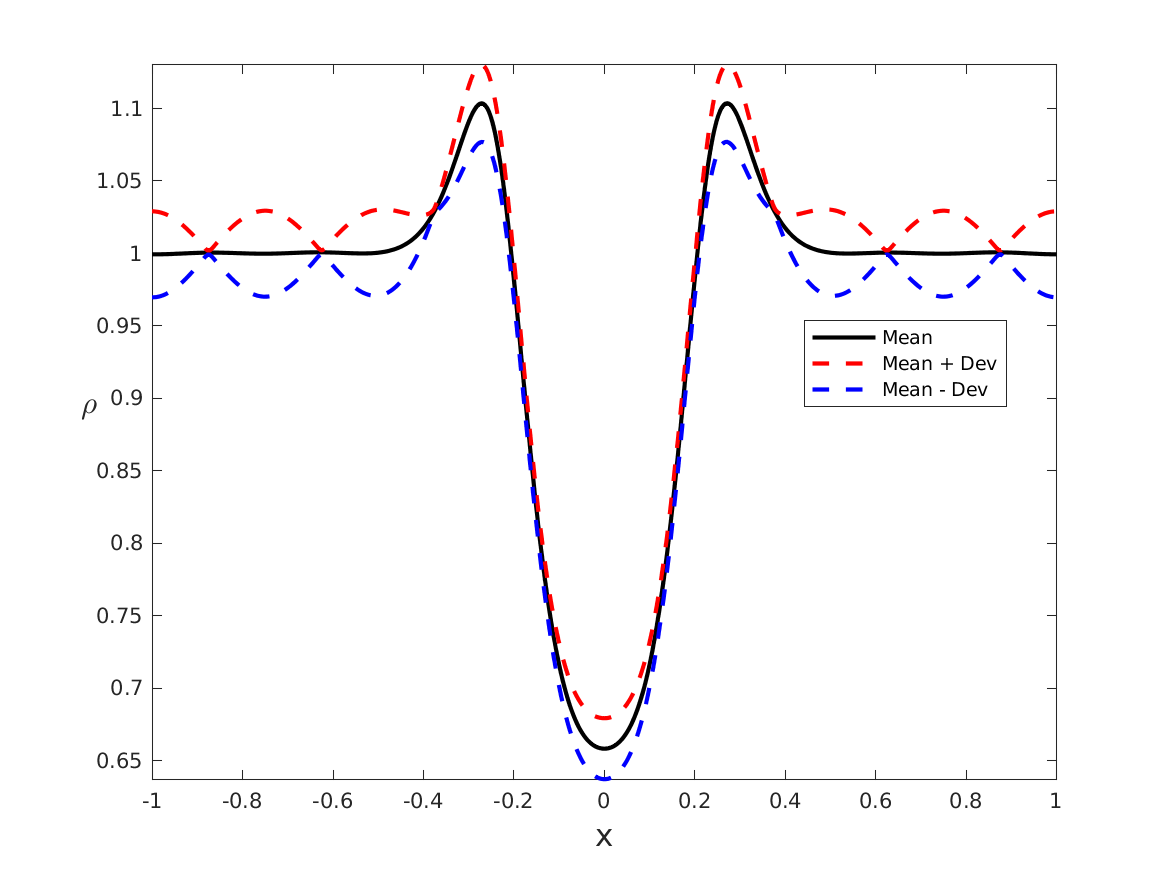}
		\caption{ \bf $\vr$ }
	\end{subfigure}	\\		
	\begin{subfigure}{0.32\textwidth}
		\includegraphics[width=\textwidth]{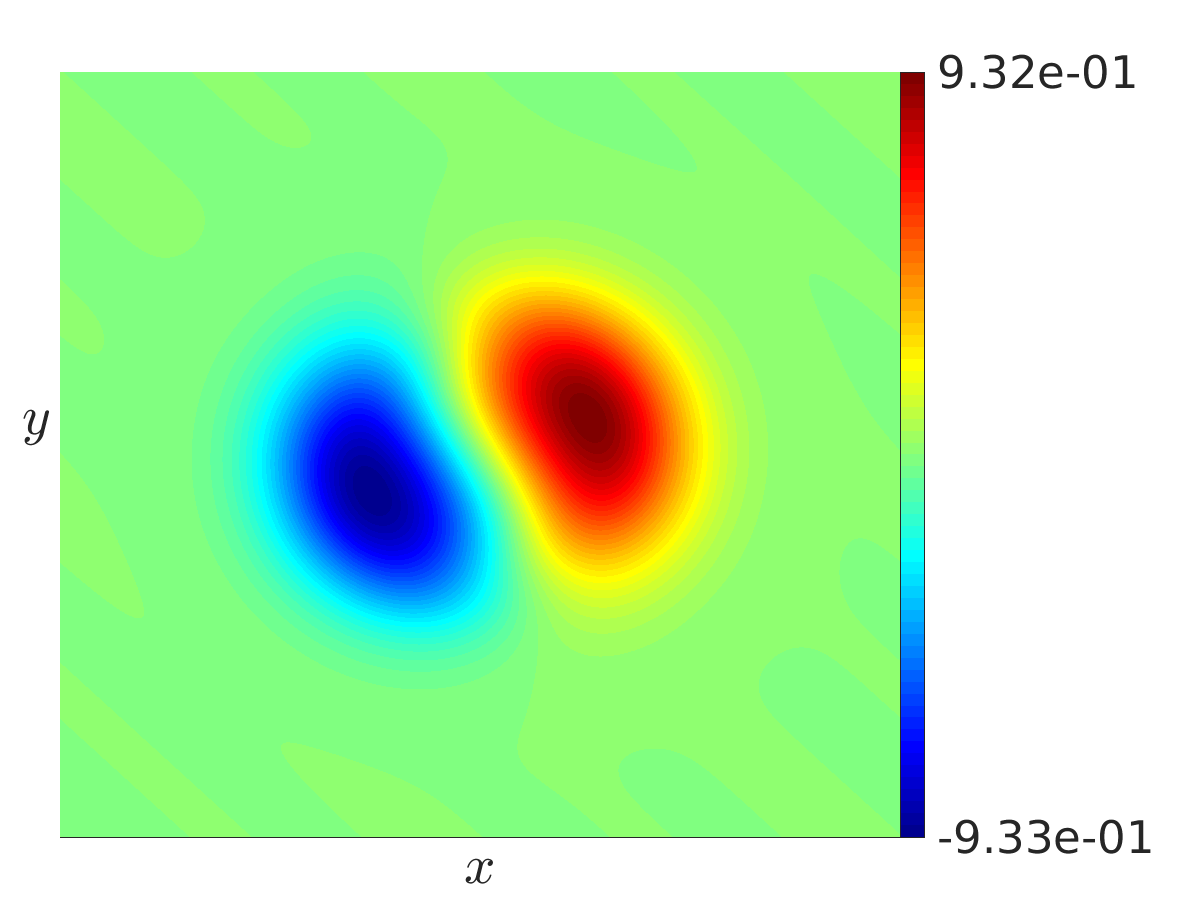}
		\caption{ \bf $u_1$ - Mean}
	\end{subfigure}	
	\begin{subfigure}{0.32\textwidth}
		\includegraphics[width=\textwidth]{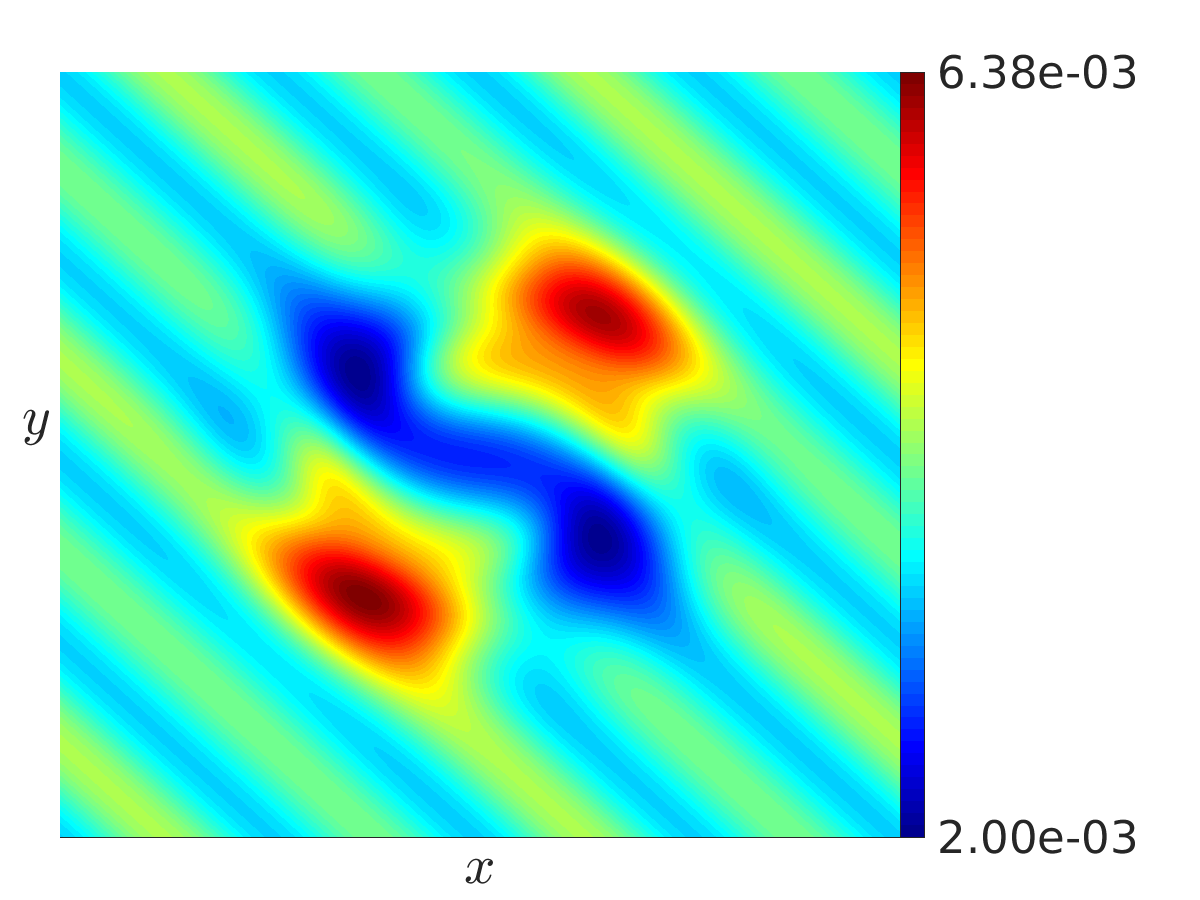}
		\caption{ \bf $u_1$ - Variance}
	\end{subfigure}	
	\begin{subfigure}{0.32\textwidth}
		\includegraphics[width=\textwidth]{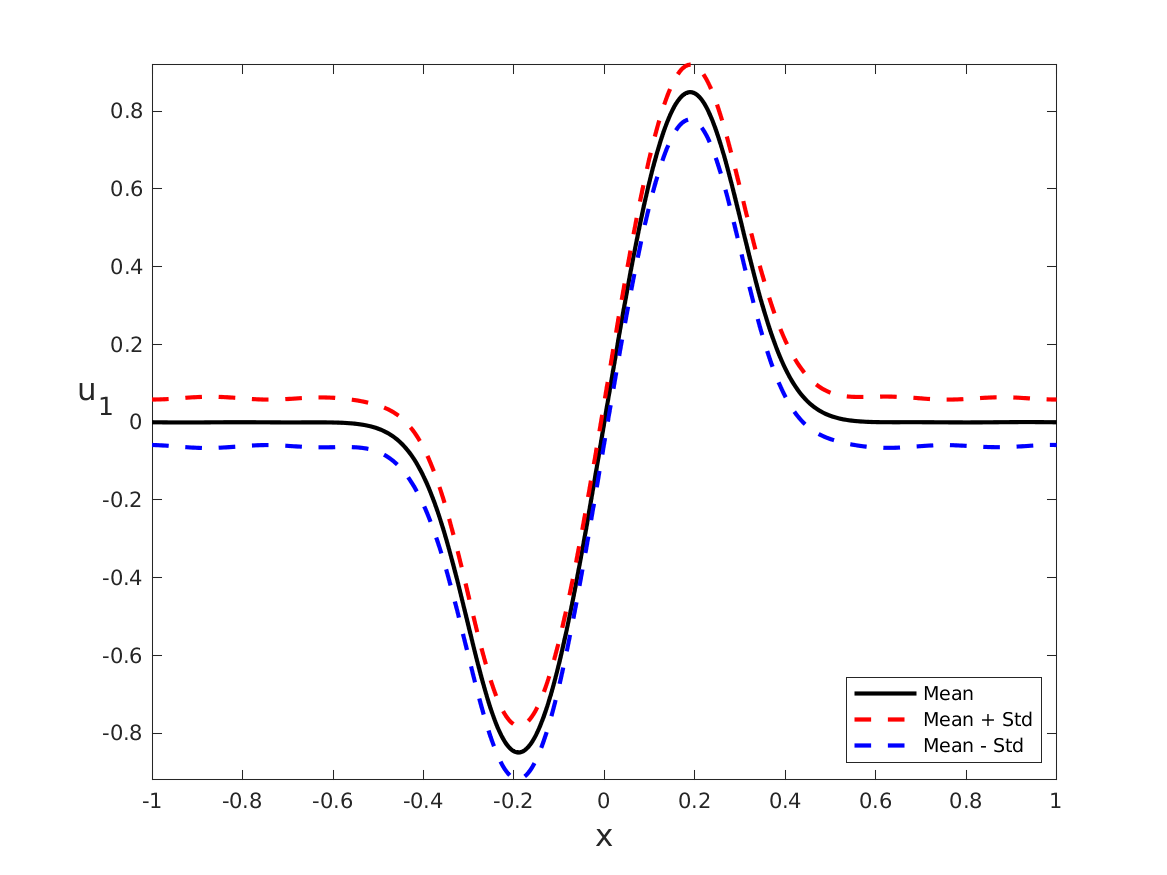}
		\caption{ \bf $u_1$}
	\end{subfigure}	\\
	\begin{subfigure}{0.32\textwidth}
		\includegraphics[width=\textwidth]{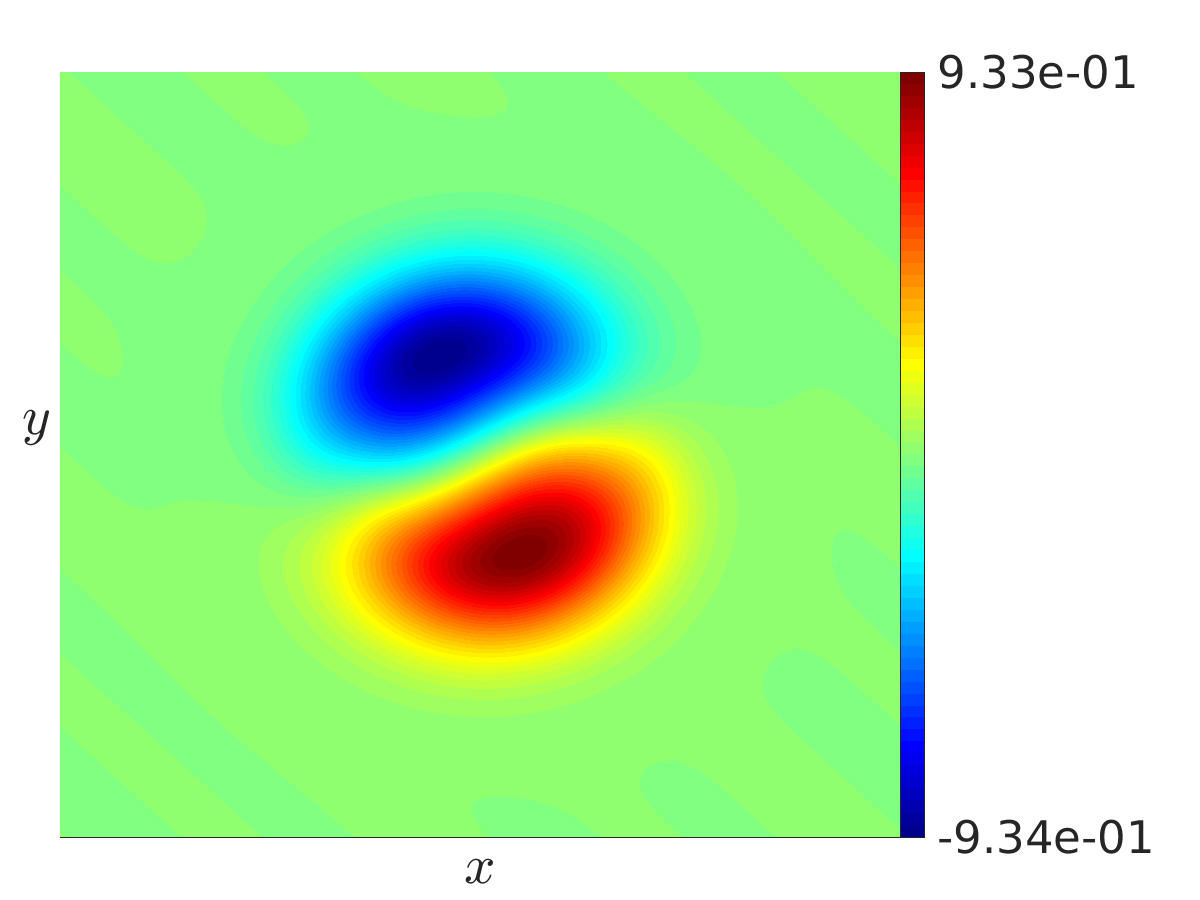}
		\caption{ \bf $u_2$ - Mean}
	\end{subfigure}	
	\begin{subfigure}{0.32\textwidth}
		\includegraphics[width=\textwidth]{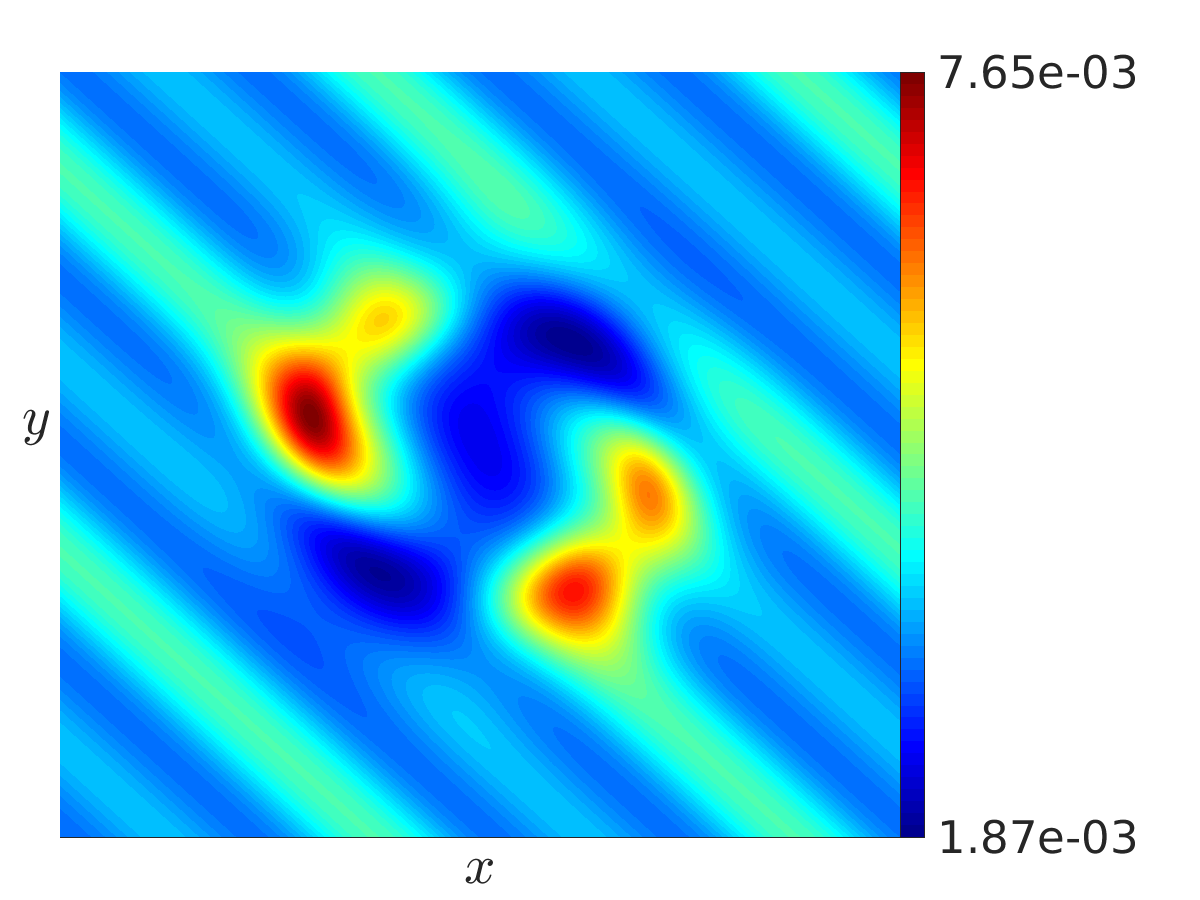}
		\caption{ \bf $u_2$ - Variance}
	\end{subfigure}	
	\begin{subfigure}{0.32\textwidth}
		\includegraphics[width=\textwidth]{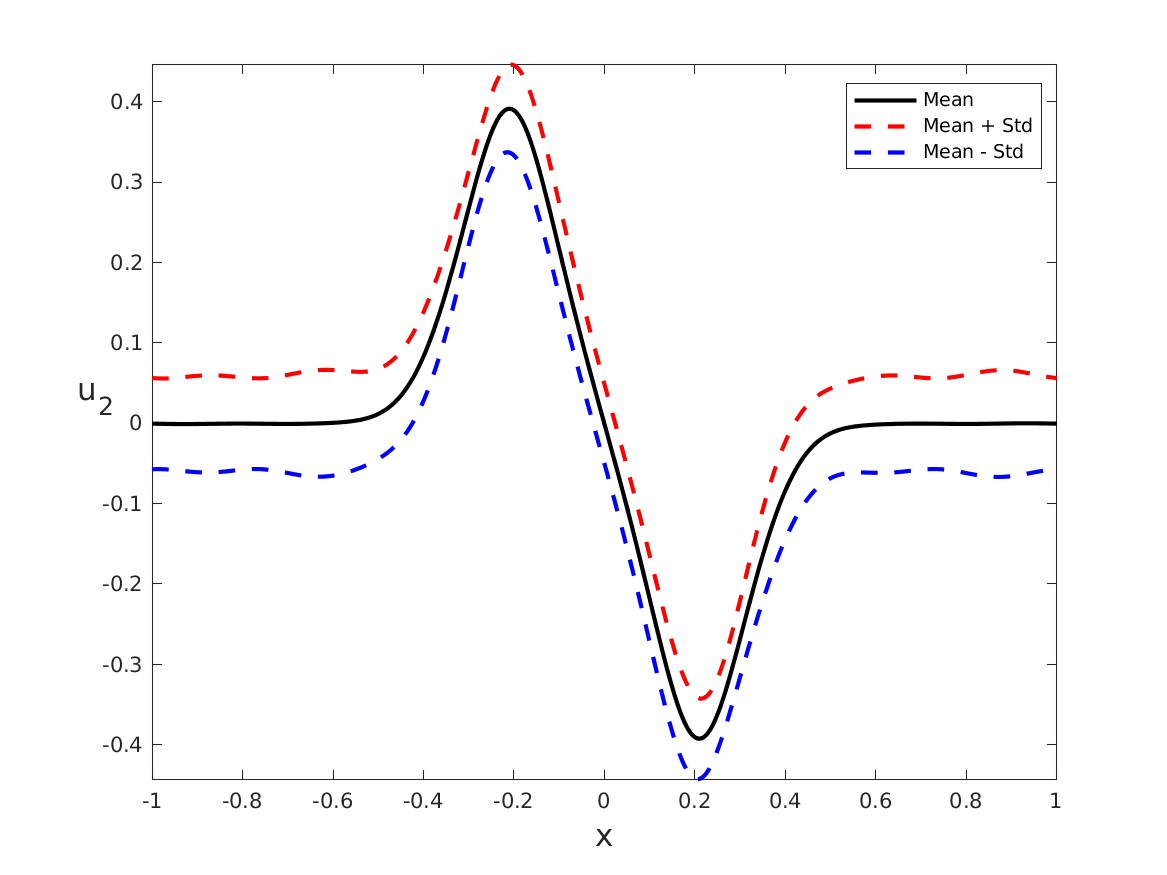}
		\caption{ \bf $u_2$}
	\end{subfigure}	
	\caption{  \small{Example \ref{example2}: Numerical solutions obtained by the FV method. Left: Mean-value of $\vr, \vu$; Middle: Deviation/variance of $\vr, \vu$; Right:  $\vr, \vu$ at the line $x = y$.  }}\label{ex2}
\end{figure}

\begin{figure}[htbp]
	\setlength{\abovecaptionskip}{0.cm}
	\setlength{\belowcaptionskip}{-0.cm}
	\centering
		\includegraphics[width=\textwidth]{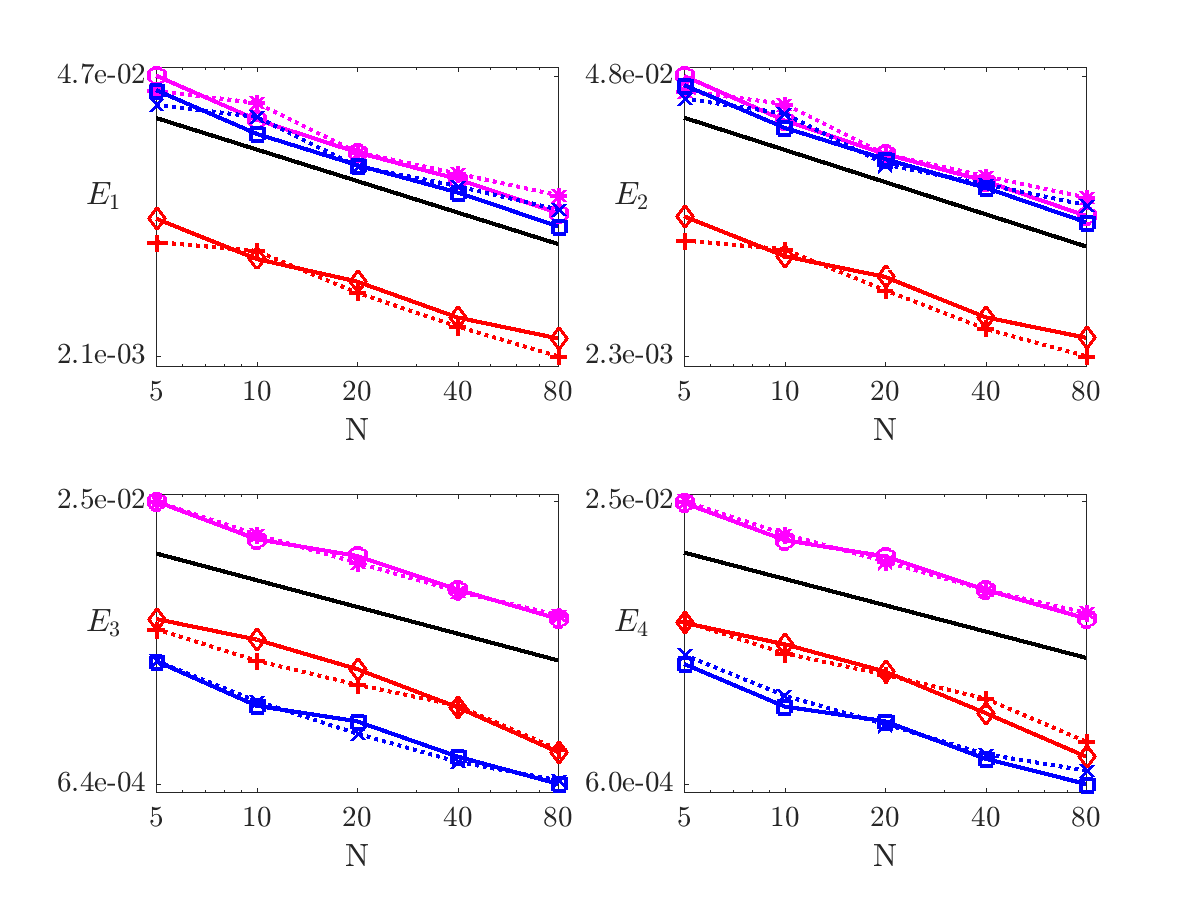}\\
		\includegraphics[width=0.6\textwidth]{graphs/legend}	
	\caption{  \small{Example \ref{example2}:  Statistical errors of the mean ($E_1, E_2$) and the deviation/variance ($E_3, E_4$).
	The black solid lines without any marker denote the reference slope of $N^{-1/2}$.  }}\label{ex2-dev}
\end{figure}


\subsection{Experiment 3: Random perturbation of the vortex interface.} \label{example3}
In the last experiment we simulate a vortex with a random perturbation of velocity interface.
In order to compare with Experiment 2, we construct this experiment in such a way that the mean of the initial data $(\vr, \vu)$ are the same as in Experiment 2.
Specifically, the initial data are set as
\begin{subequations}
\begin{align}
& \vr_0(x) = 1, \\
& \vu_0(x) = \begin{cases}
 \left(\frac{ [1-\cos( 2 \pi |x| / I)] x_2}{|x|},\ -\frac{ [1-\cos( 2 \pi |x| / I)]  /I) x_1}{|x|} \right)^t, & \mbox{if}\  |x| < I,\\
(0,0)^t, & \mbox{otherwise}
\end{cases}
\end{align}
where
\begin{equation}
I = 0.5+ Y(\omega), \quad Y \ \sim \ \mathcal{U}\left( -0.1, 0.1\right),  \quad j =1,2,3.
\end{equation}
\end{subequations}
Figure~\ref{ex3} shows the mean and the deviation/variance of numerical solutions $\vr, \vu$ obtained  by the FV method.
The errors $E_1, \dots, E_4$ obtained by the FV and MAC  methods are showed in Figure~\ref{ex2-dev}.

Compared with Figure~\ref{ex2}, Figure~\ref{ex3} indicates  that the solution has now some noise on the interface.
On the other hand, we observe again a convergence rate of $-1/2$.

Note that in our experiments we fixed a mesh resolution and analysed experimentally only the statistical errors with respect to a number of samples $N$. In our previous works \cite{FeiLukMizShe,FeLMMiSh,BS_2} we have already investigated thoroughly the discretization error of the FV and MAC methods and therefore we have concentrated here on a novel experimental part, the statistical errors.

\begin{figure}[htbp]
	\setlength{\abovecaptionskip}{0.cm}
	\setlength{\belowcaptionskip}{-0.cm}
	\centering
	\begin{subfigure}{0.32\textwidth}
		\includegraphics[width=\textwidth]{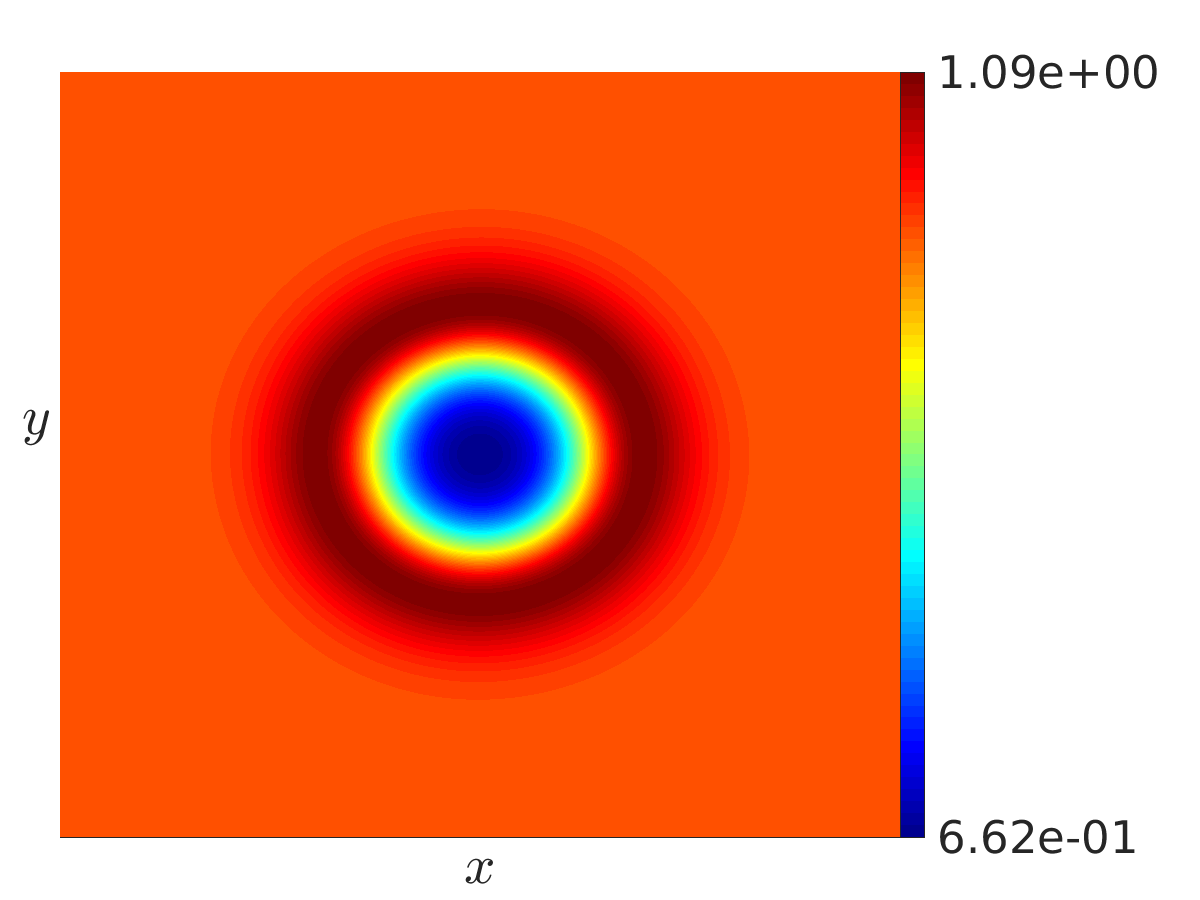}
		\caption{ \bf $\vr$ - Mean}
	\end{subfigure}	
	\begin{subfigure}{0.32\textwidth}
		\includegraphics[width=\textwidth]{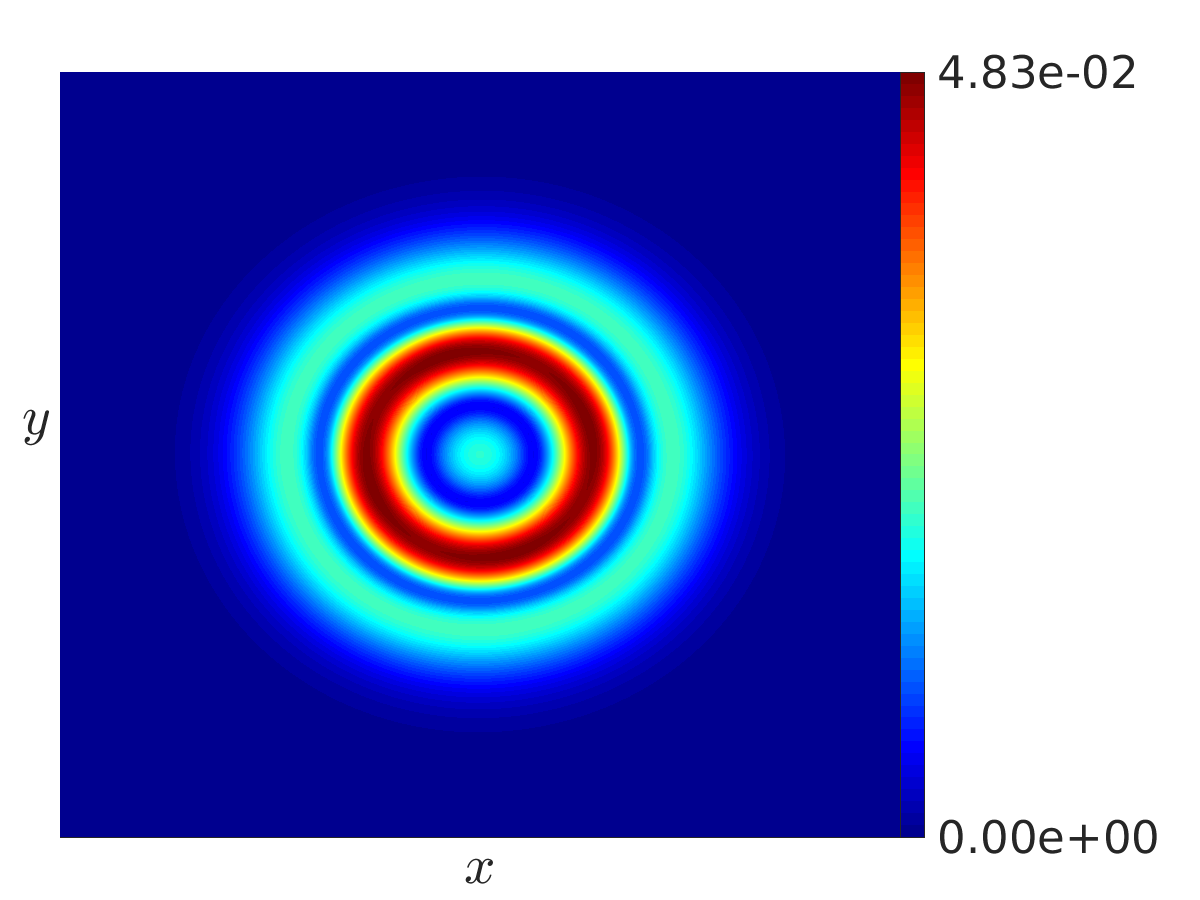}
		\caption{ \bf $\vr$ - Deviation }
	\end{subfigure}	
	\begin{subfigure}{0.32\textwidth}
		\includegraphics[width=\textwidth]{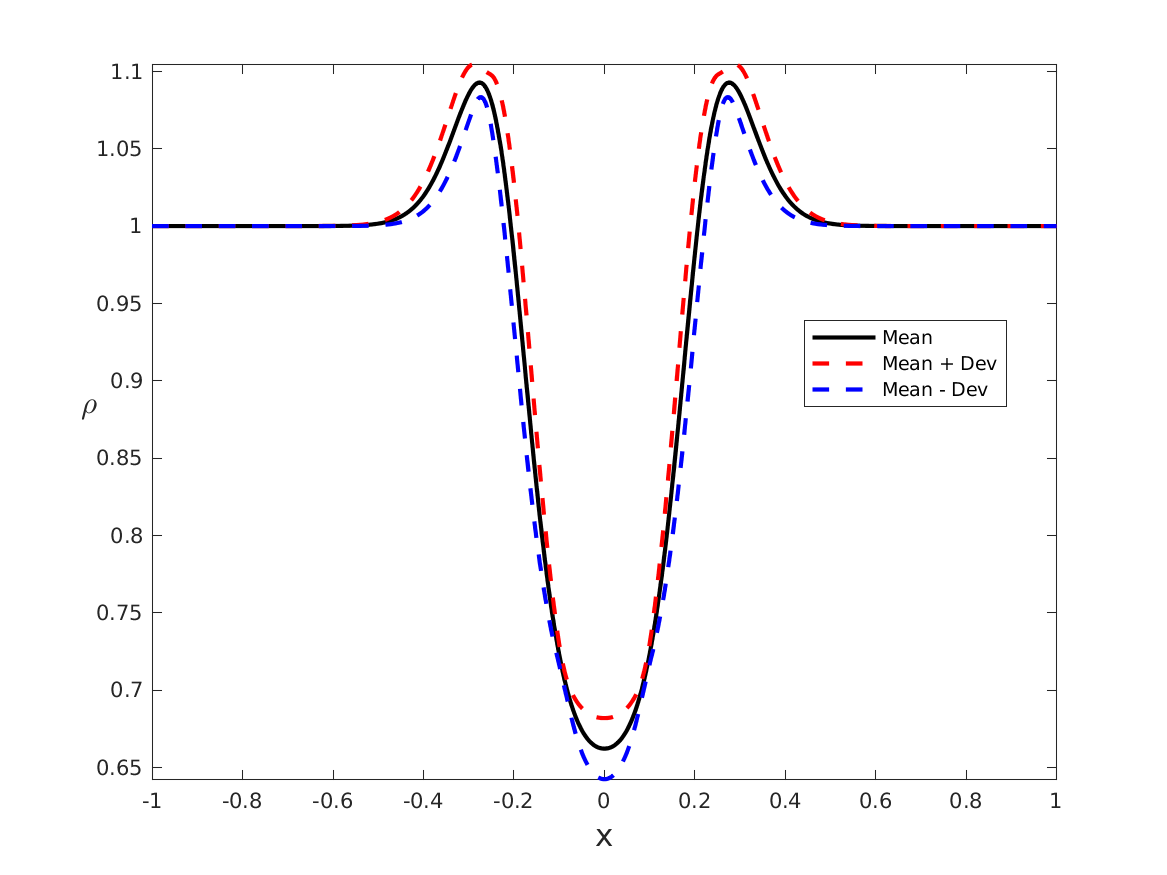}
		\caption{ \bf $\vr$  }
	\end{subfigure}	\\		
	\begin{subfigure}{0.32\textwidth}
		\includegraphics[width=\textwidth]{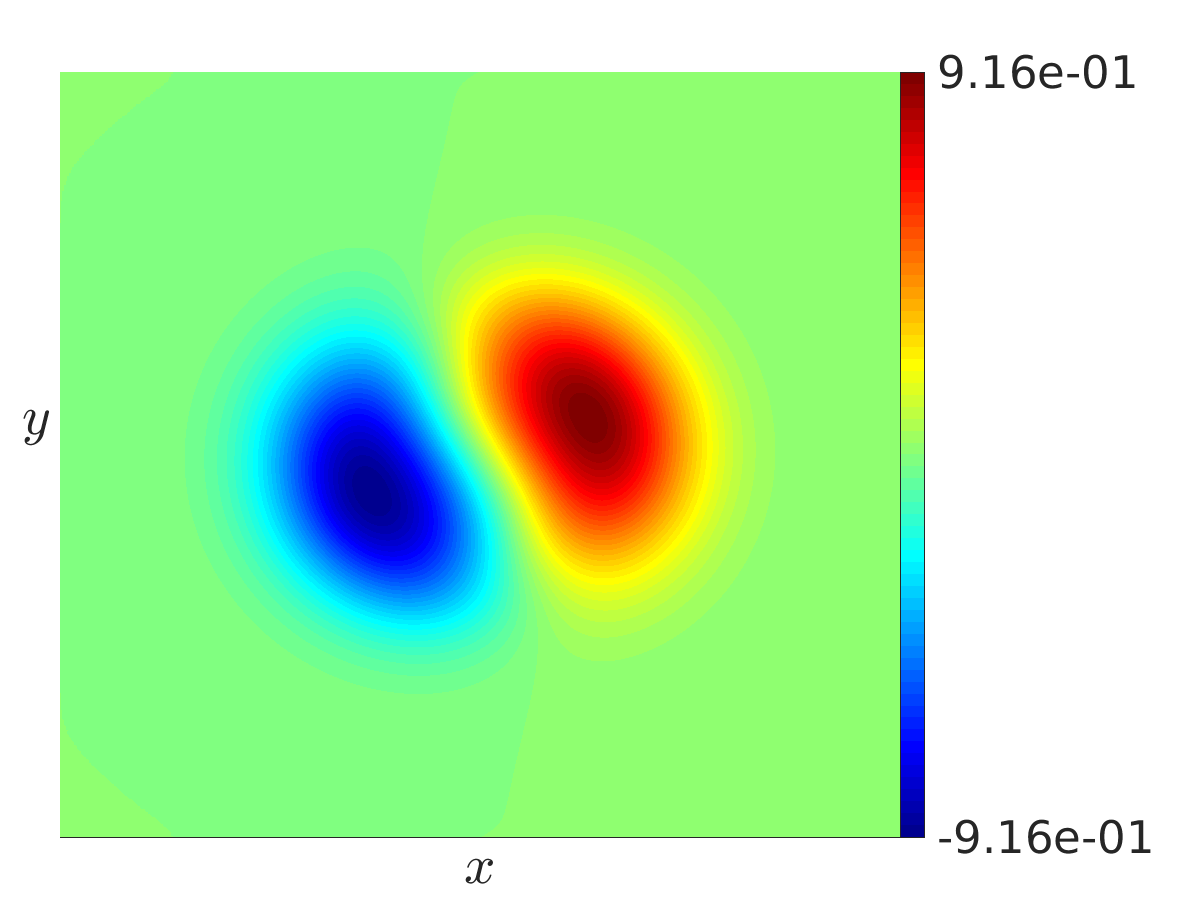}
		\caption{ \bf $u_1$ - Mean}
	\end{subfigure}	
	\begin{subfigure}{0.32\textwidth}
		\includegraphics[width=\textwidth]{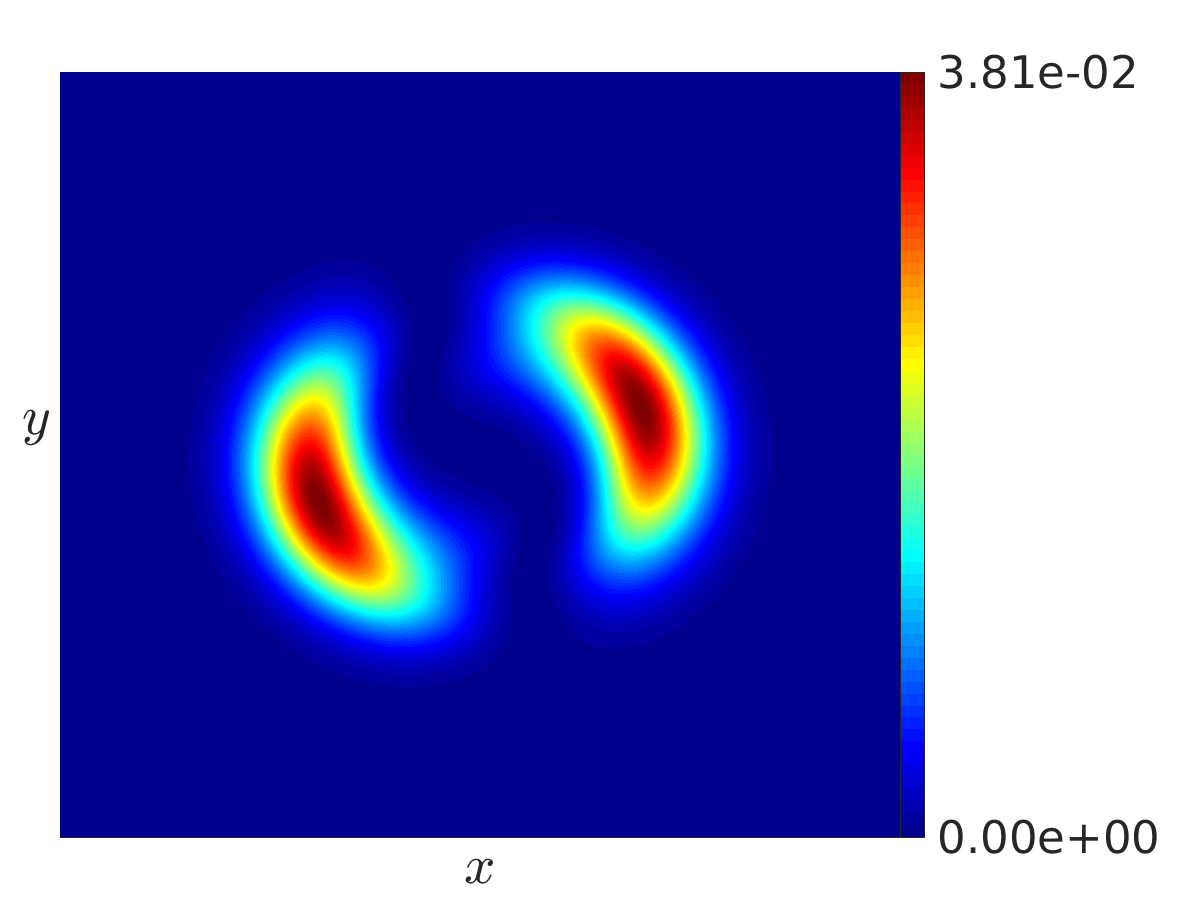}
		\caption{ \bf $u_1$ - Variance}
	\end{subfigure}	
	\begin{subfigure}{0.32\textwidth}
		\includegraphics[width=\textwidth]{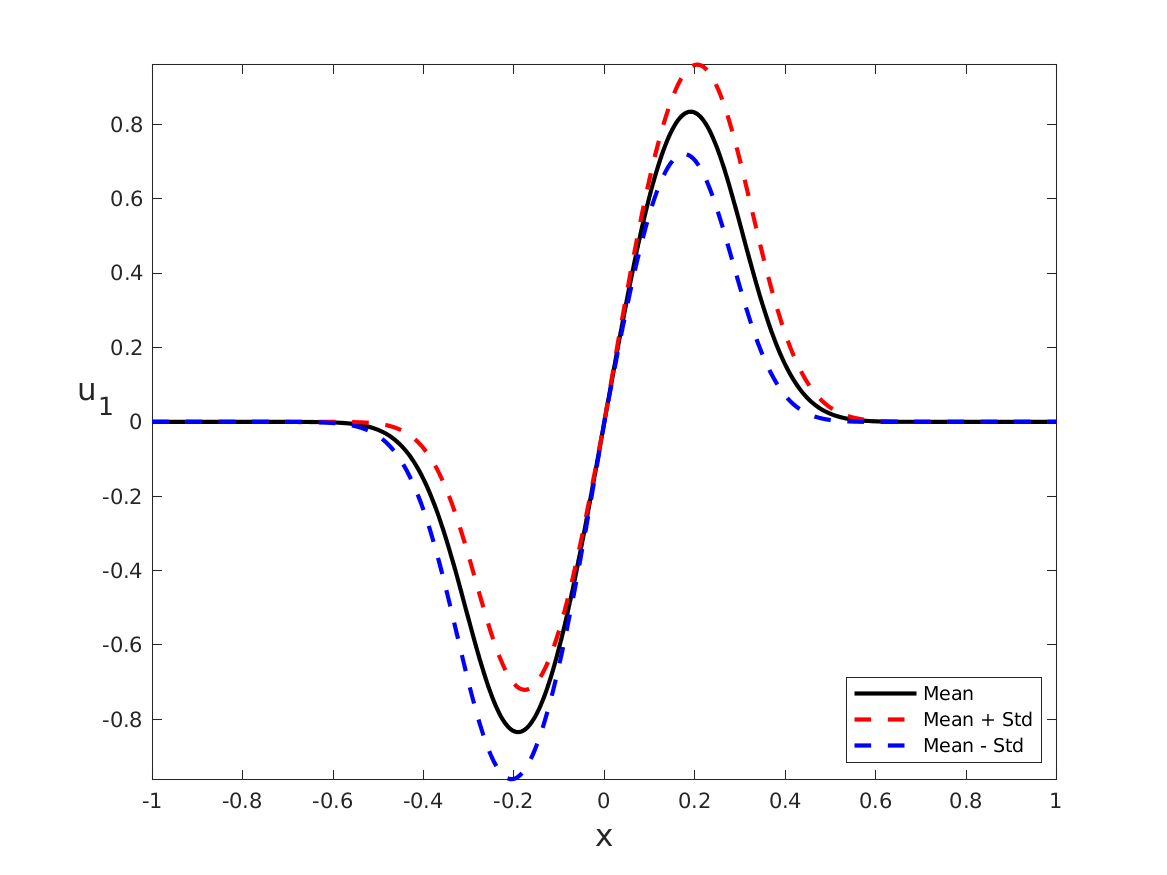}
		\caption{ \bf $u_1$}
	\end{subfigure}	\\
	\begin{subfigure}{0.32\textwidth}
		\includegraphics[width=\textwidth]{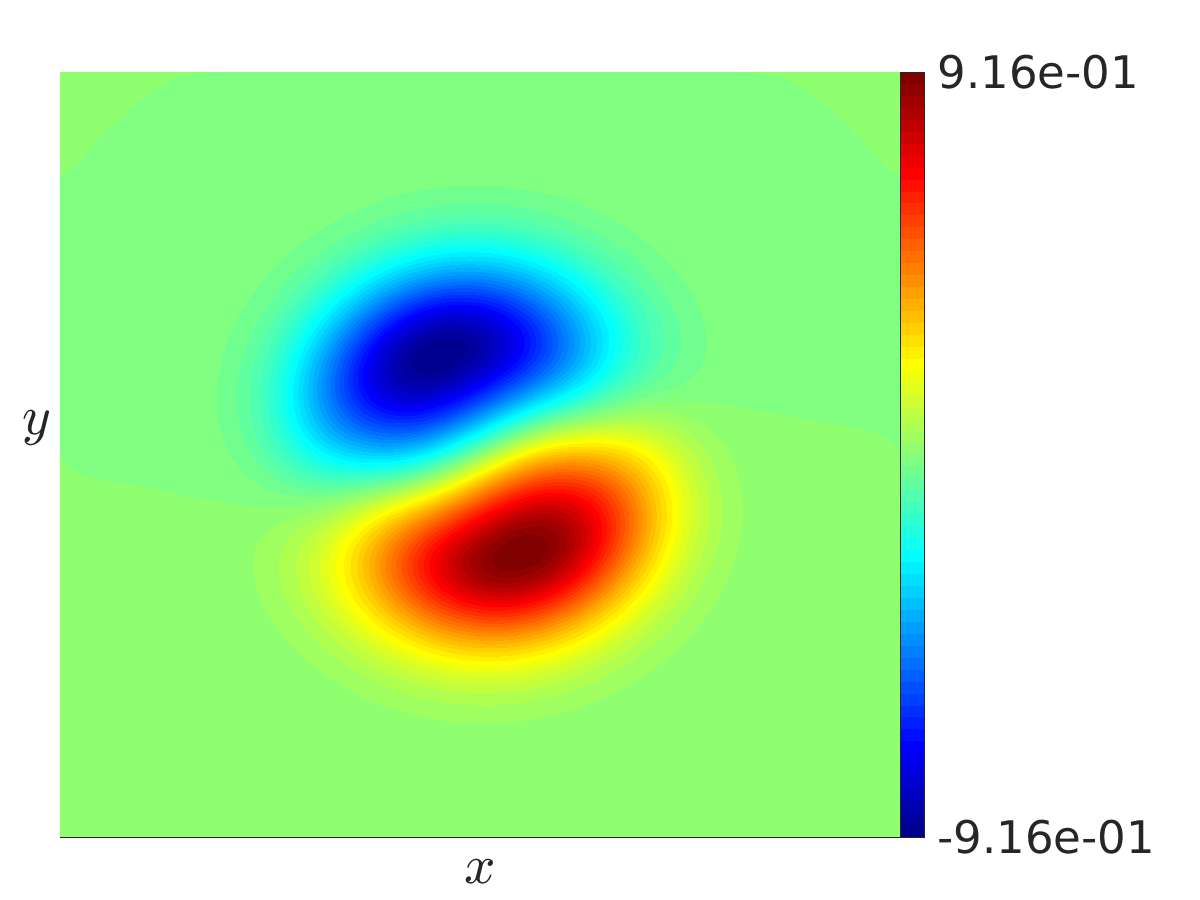}
		\caption{ \bf $u_2$ - Mean}
	\end{subfigure}	
	\begin{subfigure}{0.32\textwidth}
		\includegraphics[width=\textwidth]{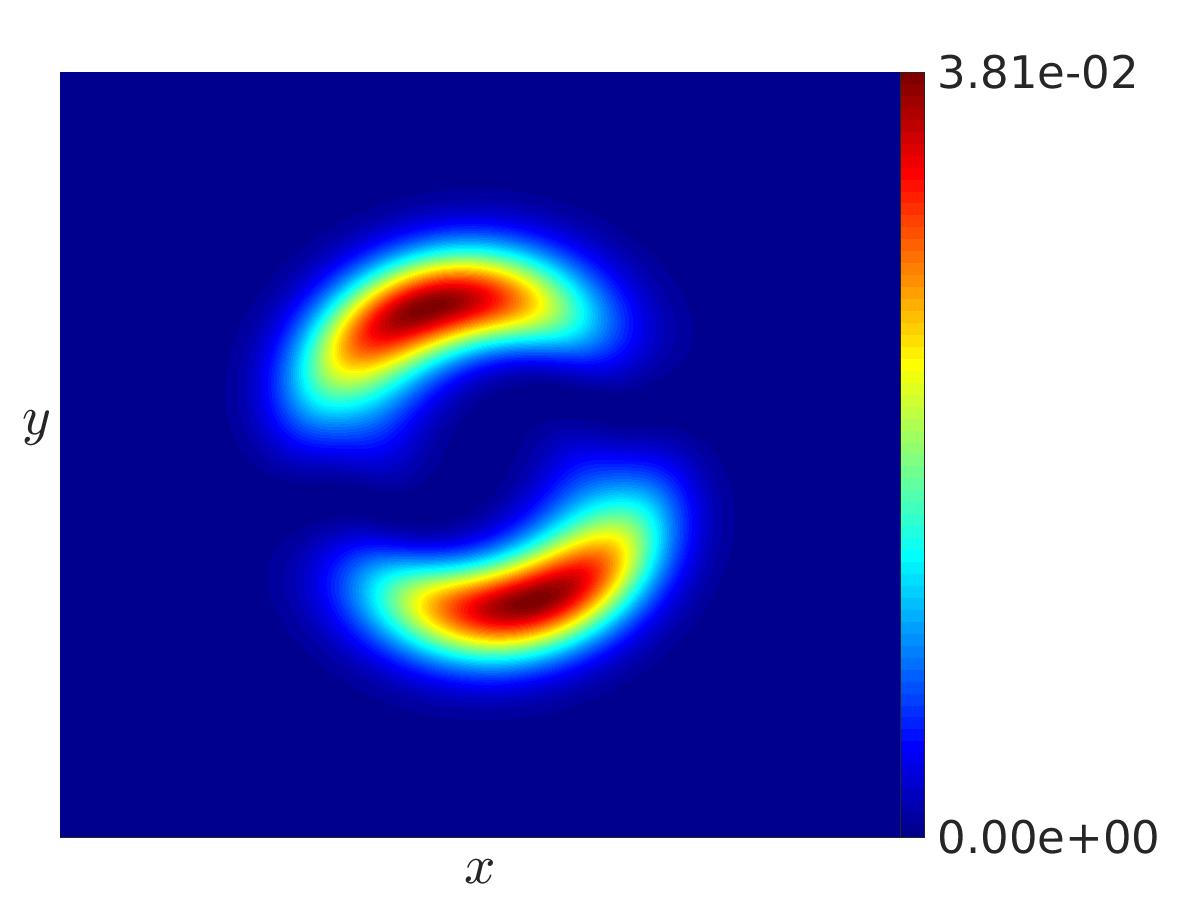}
		\caption{ \bf $u_2$ - Variance}
	\end{subfigure}	
	\begin{subfigure}{0.32\textwidth}
		\includegraphics[width=\textwidth]{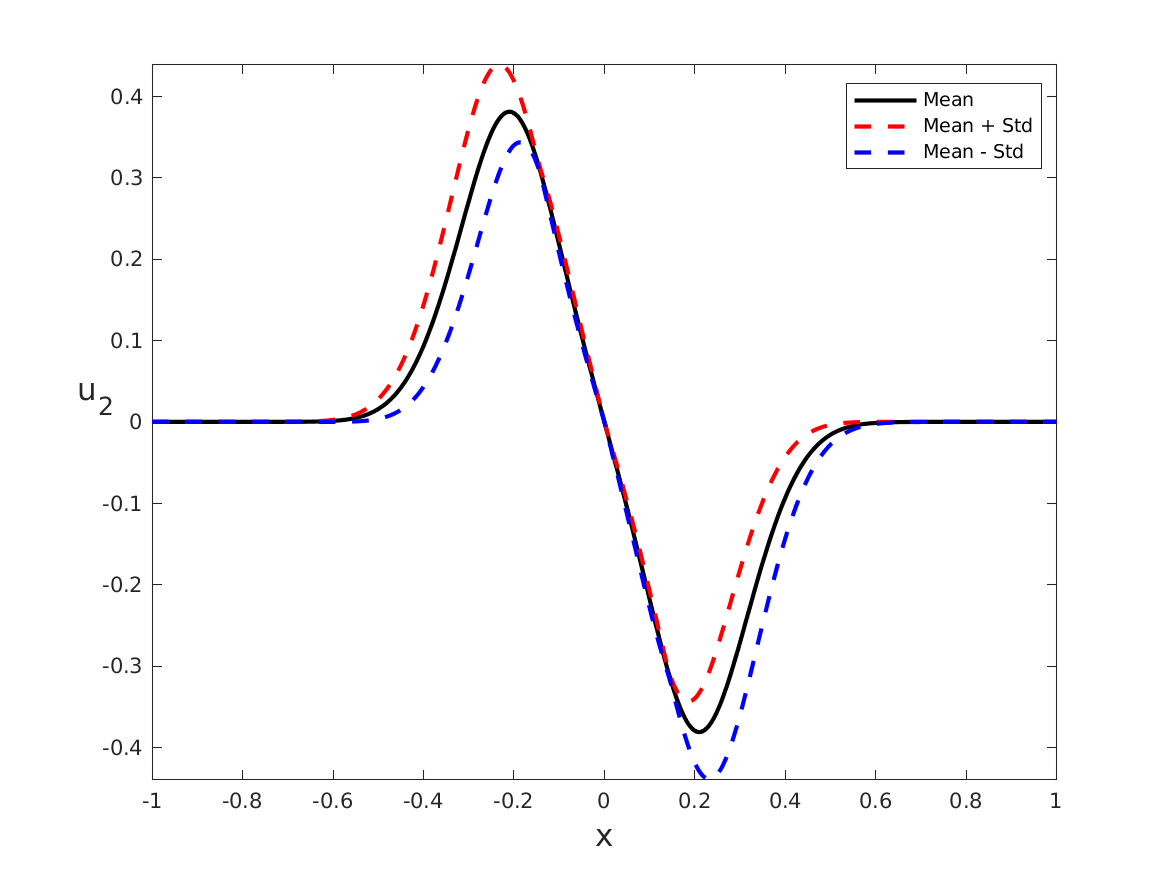}
		\caption{ \bf $u_2$}
	\end{subfigure}	
	\caption{  \small{Example \ref{example3}: Numerical solutions obtained by the FV method. Left: Mean-value of $\vr, \vu$; Middle: Deviation/variance of $\vr, \vu$; Right:  $\vr, \vu$ at the line $x = y$.}}\label{ex3}
\end{figure}

\begin{figure}[htbp]
	\setlength{\abovecaptionskip}{0.cm}
	\setlength{\belowcaptionskip}{-0.cm}
	\centering
		\includegraphics[width=\textwidth]{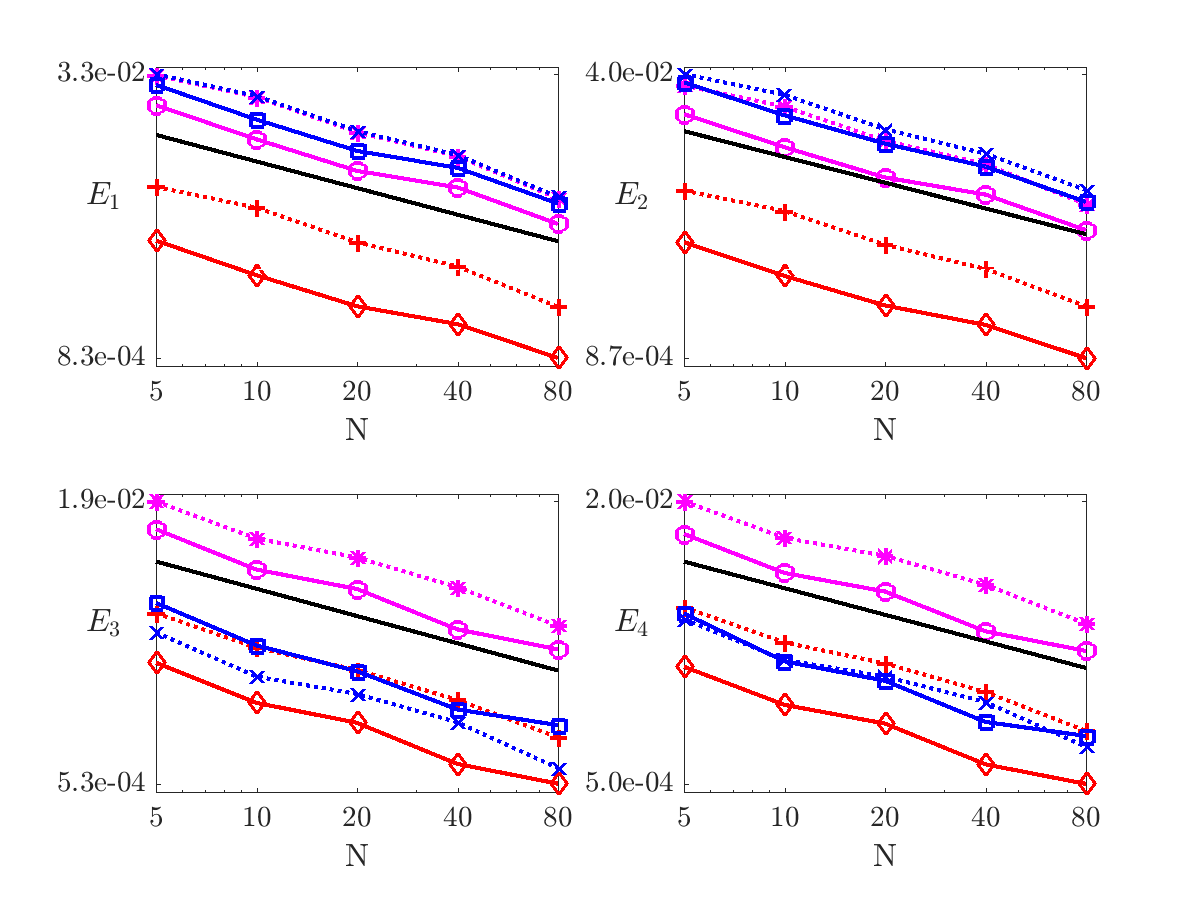}\\
		\includegraphics[width=0.6\textwidth]{graphs/legend}	
	\caption{  \small{Example \ref{example3}:  Statistical errors of the mean ($E_1, E_2$) and the deviation/variance ($E_3, E_4$).
	The black solid lines without any marker denote the reference slope of $N^{-1/2}$.  }}\label{ex3-dev}
\end{figure}

%

\section{Concluding remarks}
It is well-known that strong solutions of the compressible Navier--Stokes equations are uniquely determined by data, however these solutions may exist only locally in time.
In addition, the recent results of Buckmaster et al.~\cite{BuCLGS} and Merle et al.~\cite{MeRaRoSz}  indicate that originally
regular solutions may develop a blow up in a finite time. Further remarkable result of Sun, Wang, and Zhang \cite{SuWaZh} confirms that if a (local) strong solution remains bounded on the whole time interval then the strong solution is global in time.
Taking these results into account we anticipate that regularity of the Navier--Stokes solutions is a generic property and perform statistical analysis of the random compressible Navier--Stokes equations by means of the Monte Carlo finite volume method. The analysis is done under a rather weak assumption
that the numerical solutions are bounded in probability, cf.~hypothesis~\eqref{bip}.

In Theorem~\ref{FVT1} we have proved  the convergence of the Monte Carlo finite volume method applying intrinsic stochastic
compactness arguments via the Skorokhod representation theorem
and the Gy\"ongy–Krylov method. Consequently, we have proved that under the boundedness hypothesis~\eqref{bip} a statistical  solution of the Navier--Stokes system exists, more precisely, the strong solution exists $\prst-$a.s. Consequently, we could generalize the deterministic error estimates obtained for the finite volume method \eqref{scheme}, cf.~Proposition~\ref{PF2},  to the
Monte Carlo finite volume method. In Theorem~\ref{FVT2} we present the error estimates consisting of the statistical errors estimated in the expected values and the approximation errors estimated in probability.
Convergence of the deviation for the density, momentum and the variance of the velocity is proved as well, see Corollary~\ref{coro}.
The results are presented for the finite volume discretization in space and time, cf.~\eqref{scheme} but any other consistent approximation satisfying structure preserving properties \eqref{positivity}, \eqref{dei} can be used well. In particular, our theoretical results directly generalize to the Marker-and-Cell finite difference method, cf.~\cite{BS_2}.
Numerical experiments presented in Section~\ref{num} confirm theoretical results.

\appendix
\section{Appendix}\label{APDA}
For completeness, we provide the proof of \eqref{SP}.
\begin{proof}
First, let us denote $\bfF_u = \frac{1}{|\Td|}\intTd{\vu}$.
By Poincar\'e's inequality we have
\[ \norm{\vu - \bfF_u }_{L^q(\Td;R^d)} \aleq  \norm{\Grad \vu }_{L^2(\Td;R^{d\times d})}, \]
where  $1 \leq q < \infty$ for $d=2$ and $1 \leq q \leq 6$ for $d=3.$
This implies
\begin{align*}
  \norm{\vu}_{L^q(\Td)}^2  \aleq  \norm{\Grad \vu }_{L^2(\Td;R^{d\times d})}^2  +  |\bfF_u|^2.
\end{align*}
Recalling the assumption that  $\min{ \vr_0 } = \underline{\vr} > 0$ and the property of mass conservation we have
 $\intTd{ \vr (t, \cdot)} \geq |\Td| \underline{\vr} > 0$  for $t \in (0,T)$ and consequently,
\begin{align*}
& |\bfF_\vu|^2  = \frac{1 }{\intTd{\vr}} \intTd{\vr}  |\bfF_\vu|^2 \aleq
 \intTd{\vr  |\bfF_u|^2  } \aleq   \intTd{\vr  |\bfF_u  - \vu |^2  }  +  \intTd{\vr  |  \vu |^2  }
 \\& \aleq  \norm{\vr}_{L^{\gamma}(\Td)} \norm{\vu - \bfF_u }_{L^6(\Td;R^d)}^2 + \intTd{\vr  |  \vu |^2  }
  \aleq  \norm{\vr}_{L^{\gamma}(\Td)} \norm{\Grad \vu }_{L^2(\Td;R^{d\times d})}^2 + \intTd{\vr  |  \vu |^2  }.
\end{align*}
Using the uniform energy bounds \eqref{energy_bb} and \eqref{bounds1} we further obtain
\begin{align*}
& \intT{  \norm{\vu}_{L^q(\Td)}^2 } \aleq   \intT{\norm{\Grad \vu }_{L^2(\Td;R^{d\times d})}^2  }
+  \norm{\vr}_{L^\infty(0,T; L^{\gamma}(\Td))}\intT{ \norm{\Grad \vu }_{L^2(\Td;R^{d\times d})}^2}
\br
& + \sup_{\tau \in (0,T)}\intTd{\vr  |  \vu |^2 (\tau, \cdot) }\aleq  1 + \left(\intTd{E(\vr_0, \vm_0)}\right)^\frac{\gamma + 1}{\gamma},
\end{align*}
for $1\leq q \leq 6$ if $d=3$ and $1 \leq q < \infty$ for $d=2.$
\end{proof}

\def\cprime{$'$} \def\ocirc#1{\ifmmode\setbox0=\hbox{$#1$}\dimen0=\ht0
  \advance\dimen0 by1pt\rlap{\hbox to\wd0{\hss\raise\dimen0
  \hbox{\hskip.2em$\scriptscriptstyle\circ$}\hss}}#1\else {\accent"17 #1}\fi}

\end{document}